\documentclass[11pt,reqno]{amsart}

\usepackage{amsmath,amsthm,verbatim,amscd,amssymb,setspace,enumitem,exscale,color,slashed,mathtools,tikz-cd}
\usepackage[colorlinks]{hyperref}

\renewcommand{\epsilon}{\varepsilon}
\renewcommand{\P}{\mathbb{P}}

\renewcommand{\leq}{\leqslant}
\renewcommand{\geq}{\geqslant}

\newcommand{\N}{\mathbb{N}}
\newcommand{\Z}{\mathbb{Z}}
\newcommand{\Q}{\mathbb{Q}}
\newcommand{\R}{\mathbb{R}}
\newcommand{\C}{\mathbb{C}}
\newcommand{\T}{\mathbb{T}}
\newcommand{\av}{\declareslashed{}{-}{-0.165}{0}{\int_{S^1}}\slashed{\int_{S^1}}}

\newtheoremstyle{fancy}{}{}{\itshape}{}{\textbf\bgroup}{.\egroup}{ }{}
\newtheoremstyle{fancy2}{}{}{\rm}{}{\textbf\bgroup}{.\egroup}{ }{}

\theoremstyle{fancy}
\newtheorem{theorem}{Theorem}[section]
\newtheorem{lemma}[theorem]{Lemma}

\newtheorem{prop}[theorem]{Proposition}

\newcounter{mtheorem}
\newtheorem{mtheorem}[mtheorem]{Theorem}

\theoremstyle{fancy2}
\newtheorem{definition}[theorem]{Definition}
\newtheorem{example}[theorem]{Example}
\newtheorem{remark}[theorem]{Remark}

\numberwithin{equation}{section}

\begin{document}

\title{Classification of\\asymptotically conical Calabi-Yau manifolds}

\date{\today}

\author{Ronan J.~Conlon}
\address{Department of Mathematical Sciences, UT Dallas, Richardson, TX 75080, USA}
\email{ronan.conlon@utdallas.edu}

\author{Hans-Joachim Hein}
\address{Mathematisches Institut, WWU M\"unster, 48149 M\"unster, Germany\newline\hspace*{9pt}
Department of Mathematics, Fordham University, Bronx, NY 10458, USA}
\email{hhein@uni-muenster.de}

\begin{abstract}
A Riemannian cone $(C, g_C)$ is by definition a warped product $C = \R^+ \times L$ with metric $g_C = dr^2 \oplus r^2 g_L$, where $(L,g_L)$ is a compact Riemannian manifold without boundary. We say that $C$ is a Calabi-Yau cone if $g_C$ is a Ricci-flat K\"ahler metric and if $C$ admits a $g_C$-parallel holomorphic volume form; this is equivalent to the cross-section $(L,g_L)$ being a Sasaki-Einstein manifold. In this paper, we give a complete classification of all smooth complete Calabi-Yau manifolds asymptotic to some given Calabi-Yau cone at a polynomial rate at infinity. As a special case, this includes a proof of  Kronheimer's classification of ALE hyper-K\"ahler $4$-manifolds without twistor theory.
\end{abstract}

\maketitle

\markboth{Ronan J.~Conlon and Hans-Joachim Hein}{Classification of asymptotically conical Calabi-Yau manifolds}

\section{Introduction}

\subsection{Motivation and overview}\label{s:intro}

Recall that a complete noncompact $d$-dimensional Riemannian manifold $(M,g)$ with one end is called \emph{asymptotically locally Euclidean} (ALE) \emph{of order $\lambda < 0$} if there exist a compact set $K \subset M$, a finite subgroup $\Gamma$ of SO$(d)$ acting freely on the unit sphere $\mathbb{S}^{d-1}$, and a diffeomorphism $\Phi: (\R^d \setminus B_1(0))/\Gamma \to M\setminus K$ such that for all $j \in \N_0$,
\begin{equation}\label{ALEdefn}
|\nabla_{g_0}^j(\Phi^*g - g_0)|_{g_0} = O(r^{\lambda-j}),
\end{equation}
where $g_0$ denotes the standard Euclidean metric on $\R^d$ or on the flat cone $\R^d/\Gamma$.

Kronheimer's classification of  ALE hyper-K\"ahler $4$-manifolds \cite{Kronheimer, Kronheimer2} is a foundational result of $4$-dimensional Riemannian geometry. Its significance can be explained as follows. On one hand, the classification is completely explicit in terms of the geometry and topology of the so-called \emph{Kleinian surface singularities} $\C^2/\Gamma$ (here $\Gamma$ is a finite subgroup of SL$(2,\C)$ acting freely on $\C^2\setminus\{0\}$) and of their resolutions and deformations \cite{durfee, lamotke}. On the other hand, {a complete Ricci-flat Riemannian $4$-manifold of maximal volume growth is necessarily ALE} \cite{ALE, CheegerSurvey, ChNa, TianSurvey}, so if it is K\"ahler, it has to be a finite quotient of a hyper-K\"ahler ALE space \cite{suvaina, wright}. As a consequence, Kronheimer's results give us a handle on the possible metric degenerations of K\"ahler-Einstein surfaces with fixed volume and bounded diameter. See the surveys \cite{AndersonSurvey, ViaclovskySurvey} for some of the many developments of this idea.

Our results in this paper may be viewed as a generalization of Kronheimer's work to dimensions $d > 4$. One main difficulty of any generalization of this kind is that the strictly $4$-dimensional tools of twistor theory that Kronheimer used in \cite{Kronheimer,Kronheimer2} are no longer available. Let us briefly sketch the existing higher-dimensional theory to show how our results fit into a general picture.

Every $d$-dimensional complete Ricci-flat Riemannian manifold  $(M, g)$ of maximal volume growth has \emph{tangent cones at infinity}: any sequence $t_i \to 0$ has a subsequence $t_{i_j} \to 0$ such that the sequence $(M, t_{i_j}^2 {\rm dist}_g, p)$ of pointed Riemannian manifolds (here $p \in M$ is an arbitrary basepoint) converges in the pointed Gromov-Hausdorff sense to the metric cone $C = C(L)$ over some complete geodesic metric space $L$ of diameter $\leq \pi$ \cite{ChBook, CheegerColding}. The \emph{link} $L$ is a smooth Riemannian manifold of dimension $d-1$ away from a closed rectifiable subset of Minkowski dimension $\leq d-5$ \cite{ChNa, JiangNaber}. The possible singularities of $L$ are a source of major difficulties if $d > 4$. There exist many interesting (mostly very recent) examples where $L$ is indeed singular \cite{biq-delc, BiqGaud, CDR, CR, JoyceQALE, LiYang, Szekel}.

It is widely expected that every sequence $t_i \to 0$ and subsequence $t_{i_j} \to 0$ as above leads to the same tangent cone $C = C(L)$. This is the well-known open problem of \emph{uniqueness of tangent cones}. Uniqueness is known if at least one tangent cone has a smooth link $L$ \cite{Cheeger, CM}. In the {K\"ahler} case, this smoothness assumption is unnecessary thanks to the breakthrough work of Donaldson-Sun \cite{DS2} if $M$ itself arises as a blow-up limit of a sequence of compact polarized K\"ahler-Einstein manifolds. Crucially, in this situation, \cite{DS2} establish a close algebraic relationship between $M$ and $C$. Liu \cite{LiuGang} has removed the assumption of $M$ being a blow-up limit from \cite{DS2} if at least one tangent cone of $M$ has a smooth link. The results of \cite{DS2, LiuGang} on the algebraic relationship between $M$ and $C$ amount to a more general but less precise version of some of our results, which we will now describe.\footnote{Our methods in this paper are independent of \cite{DS2, LiuGang}. A previous version of this paper dealing with the case where $L$ is a regular or quasi-regular Sasaki-Einstein manifold was posted to the arXiv in 2014. The current version includes the irregular case, relying crucially on the latest version of Li's work \cite{ChiLi}, which is also independent of \cite{DS2, LiuGang}.}

Consider a complete Ricci-flat K\"ahler manifold of maximal volume growth at least one of whose tangent cones has a smooth link, so that the tangent cone is unique by \cite{Cheeger,CM,DS2,LiuGang}. The known proofs of uniqueness yield a {convergence rate} of $M$ to $C$ no better than $O((\log r)^{\lambda})$ for some $\lambda< 0$ unless $C$ satisfies a certain integrability property \cite{Cheeger}. While this logarithmic rate is expected to be optimal in general, all known examples with a tangent cone with a smooth link actually converge at a \emph{polynomial rate}, exactly as in \eqref{ALEdefn}. In this paper, we completely characterize the polynomial rate case, thus completing our previous work in \cite{Conlon, Conlon3}. For example, the following uniqueness theorem will be a typical application of our methods (see Theorem \ref{uniqueness} in Section \ref{s:results}):

\begin{quote}
Let $M$ be a complete Ricci-flat K\"ahler manifold of real dimension $d = 2n \geq 4$ whose metric is polynomially asymptotic to the unique SO$(n+1)$-invariant Ricci-flat K\"ahler cone metric on $C = \{z_0^2 + \cdots + z_n^2 = 0\} \subset \C^{n+1}$ up to diffeomorphism. Then modulo scaling and diffeomorphism, either $M = T^*\mathbb{S}^n$ together with Stenzel's metric \cite{Stenzel},  or $n = 3$ and $M = \mathcal{O}_{\P^1}(-1)^{\oplus 2}$ together with Candelas-de la Ossa's metric \cite{delaossa}.
\end{quote}

\noindent In addition, we also prove an existence theorem (Theorem \ref{existence}) that exhausts all possible examples of asymptotically conical Calabi-Yau manifolds, including those asymptotic to \emph{irregular} cones.

\subsection{Preliminaries}\label{s:prelims}

\subsubsection{Riemannian, K\"ahler, and Calabi-Yau cones}\label{sss:cones}

\begin{definition}\label{cone}
Let $( L, g_L)$ be a compact connected Riemannian manifold. The \emph{Riemannian cone with link} $L$ is defined to be the manifold $C = \R^+ \times L$ with metric $g_C = dr^2 \oplus r^2g_L$. For simplicity we will usually write $g_0$ instead of $g_C$ if the cone $C$ is given.
\end{definition}

We now consider Riemannian cones whose metric is a K\"ahler metric. The links of such cones are called \emph{Sasaki manifolds}. The book \cite{book:Boyer} is an excellent general reference for Sasaki geometry.

\begin{definition}
A \emph{K{\"a}hler cone} is a Riemannian cone $(C,g_0)$ such that $g_0$ is K{\"a}hler, together with a choice of $g_0$-parallel complex structure $J_0$. This will in fact often be unique up to sign. We then have a K\"ahler form $\omega_0(X,Y) = g_0(J_{0}X,Y)$, and $\omega_0 =\frac{i}{2}\partial\overline{\partial} r^2$ with respect to $J_0$.
\end{definition}

The \emph{Reeb vector field} $\xi=J_{0}(r\partial_r)$ of a K\"ahler cone is a holomorphic Killing field tangent to the link. The closure of the $1$-parameter subgroup of ${\rm Isom}(C,g_0)$ generated by $\xi$ is a compact torus $\T$ of holomorphic isometries called the \emph{Reeb torus} of the cone. We say that the cone is \emph{regular} if $\T = S^1$ acting freely, \emph{quasi-regular} if $\T = S^1$ not acting freely, and \emph{irregular} if $\dim \T > 1$.

A \emph{transverse automorphism}  of a K\"ahler cone is an automorphism of the complex manifold $(C,J_{0})$ that preserves the Reeb vector field $\xi$, or equivalently the scaling vector field $r\partial_r = -J_0\xi$. We denote the group of all transverse automorphisms by $\operatorname{Aut}^{T}(C)$. This contains the scaling action of $\mathbb{R}^+$.

The completion $C \cup \{o\}$ carries a unique structure of a normal affine algebraic variety extending the given complex manifold structure on $C$. Moreover, we can take the action of $\T$ to be linear on the ambient affine space in such a way that all weights of the action of $\xi$ lie in $\R^+$. See \cite[\S 3.1]{vanC4} and \cite[\S 2.1]{Tristan}, and see \cite[\S 4.2.3]{Conlon3} for an example. The key point is that every {periodic} vector field $\xi' \in {\rm Lie}(\T)$ with $\langle \xi, \xi' \rangle > 0$ exhibits $C$ as the total space of a negative holomorphic $\C^*$-orbibundle over the compact complex orbifold $C/ \langle e^{t\xi'} \rangle$. The desired algebraic structure on $C \cup \{o\}$ is given by the Remmert reduction \cite[p.336, \S2.1]{Grau:62} (i.e., the contraction of the zero section) of the total space of the associated negative holomorphic line orbibundle. This structure is independent of $\xi'$.

\begin{definition}\label{d:cycone}
A quadruple $(C,g_0,J_0,\Omega_0)$ is a \emph{Calabi-Yau cone} if $(C,g_0, J_0)$ is a Ricci-flat K\"ahler cone of complex dimension $n$, the canonical bundle $K_{C}$ of the complex manifold $(C,J_0)$ is trivial, and $\Omega_{0}$ is a $g_0$-parallel holomorphic section of $K_{C}$ with $\omega_0^n = i^{n^2}\Omega_0 \wedge \overline{\Omega}_0$. We note here that $J_0$ can be recovered from $\Omega_0$ using the fact that $\Lambda^{1,0}_{J_0}C = {\rm ker}[\Lambda^1_{\C}C \ni \alpha \mapsto \Omega_0 \wedge \alpha \in \Lambda^{n+1}_{\C}C]$.
\end{definition}

If we fix the complex manifold $(C,J_0)$ and the scaling vector field $r\partial_r$, then the Calabi-Yau cone structure $(g_0,\Omega_0)$ is unique up to scaling and up to the action of ${\rm Aut}^T_0(C)$ \cite{Sparks2, nitta}.

The links of Calabi-Yau cones are called \emph{Sasaki-Einstein manifolds}. See \cite[\S 11.4]{book:Boyer} for a nice set of regular, quasi-regular, and irregular examples. The recent paper \cite{GaborTristan2} characterizes Sasaki-Einstein manifolds completely in terms of $K$-stability, providing many new examples.

\subsubsection{Asymptotically conical Calabi-Yau manifolds}

\begin{definition}\label{ACCY}
Let $(C,g_0,J_0,\Omega_0)$ be a Calabi-Yau cone as above. Let $(M,g,J,\Omega)$ be a Ricci-flat K\"ahler manifold with a parallel holomorphic volume form such that $\omega^n = i^{n^2}\Omega \wedge \overline\Omega$. We call $M$ an \emph{asymptotically conical Calabi-Yau manifold with asymptotic cone $C$} if there exist a compact subset $K\subset M$ and a diffeomorphism $\Phi: \{r > 1\} \to M\setminus K$ such that for some $\lambda_1, \lambda_2 < 0$ and all $j \in \N_0$,
\begin{align}
|\nabla_{g_0}^j(\Phi^{*}g-g_{0})|_{g_{0}} =O(r^{\lambda_1-j}),\label{e:asympt:g}\\
|\nabla_{g_0}^j(\Phi^{*}J-J_{0})|_{g_{0}} =O(r^{\lambda_2-j}),\label{e:asympt:J}\\
|\nabla_{g_0}^j(\Phi^{*}\Omega-\Omega_{0})|_{g_{0}} = O(r^{\lambda_2-j}).\label{e:asympt:Om}
\end{align}
We abbreviate the words ``asymptotically conical'' by AC.
\end{definition}

\begin{remark}
(1) $\Phi$ is not required (and typically cannot be chosen) to be $(J_0, J)$-holomorphic.

(2) Condition (\ref{e:asympt:g}) already implies that (\ref{e:asympt:J}) and (\ref{e:asympt:Om}) hold, with $\lambda_2 \leqslant \lambda_1$, for {some} $g_0$-parallel tensors $J_0, \Omega_0$, simply because $J,\Omega$ are $g$-parallel. However, it is often possible to take $\lambda_2 < \lambda_1$.

(3) \cite[Lemma 2.14]{Conlon} tells us that the optimal exponents in (\ref{e:asympt:J}) and (\ref{e:asympt:Om}) are a priori equal.
\end{remark}

In this paper, we work with AC Calabi-Yau manifolds for simplicity even though it may be more natural to consider the larger class of all {AC Ricci-flat K\"ahler manifolds}. The following proposition, which was inspired by \cite[\S4]{suvaina}, clarifies the difference between these two classes.

\begin{prop}\label{p:acrfk}
Let $(M,g,J)$ be a Ricci-flat K\"ahler manifold asymptotic to a Ricci-flat K\"ahler cone  $(C,g_{0},J_0)$ in the sense that \eqref{e:asympt:g} and \eqref{e:asympt:J} hold. Then $\pi_1(M)$ is finite and the universal cover $(\tilde{M},\tilde{g},\tilde{J})$ is an {\rm AC} Calabi-Yau manifold. Moreover, $(M,g,J)$ is itself an {\rm AC} Calabi-Yau manifold if and only if $(C,g_0,J_0)$ is a Calabi-Yau cone, i.e., admits a compatible Calabi-Yau structure $\Omega_0$.
\end{prop}

\begin{proof}
Let $L$ denote the link of $C$. Fix a  compact set $K\subset M$ such that $M\setminus K$ is diffeomorphic to $(0,\infty)\times L$. Let $p:\tilde{M}\to M$ be the universal cover of $M$ and consider the preimage $\tilde{K}=p^{-1}(K)$. Since $\pi_1(M)$ is finite \cite{AndFund, LiFund}, $\tilde{K}$ will be compact and $\tilde{M}\setminus\tilde{K}$ will consist of finitely many connected components. If the number of components was two or greater, then $\tilde{M}$ would split off a line by the Cheeger-Gromoll splitting theorem applied to $\tilde{g}$, contradicting the maximal volume growth of $\tilde{g}$. So there is only one connected component, and this then takes the form $(0,\infty) \times \tilde{L}$ for some connected

\newpage

\noindent finite covering space $\tilde{L}\to L$. This is a normal covering with deck transformation group $\pi_1(M)$. In particular, $\tilde{M}$ is an AC Ricci-flat K\"ahler manifold asymptotic to a Ricci-flat K\"ahler cone $\tilde{C}$ with link $\tilde{L}$ via a $\pi_1(M)$-equivariant diffeomorphism $\tilde\Phi$. Since $\pi_1(\tilde{M}) = \{0\}$, we may trivialize $K_{\tilde{M}}$ by a $\tilde{g}$-flat section $\tilde\Omega$, which then clearly has a limit $\tilde\Omega_0$ on $\tilde{C}$, proving that $\tilde{M}$ is AC Calabi-Yau.

It remains to prove that if $(C,g_0,J_0)$ has a compatible Calabi-Yau structure $\Omega_0$, then $(M,g,J)$ has a compatible Calabi-Yau structure $\Omega$ asymptotic to $\Omega_0$ via $\Phi$. For this we pass to the universal cover $\tilde{M} \to M$ as above, noting that $\tilde{C}$ now comes equipped with a Calabi-Yau structure $\tilde\Omega_0$ which is $\pi_1(M)$-invariant. Because $K_{\tilde{M}}$ is flat without holonomy and (crucially) because $\tilde{M}$ is one-ended, we can easily construct a Calabi-Yau structure $\tilde\Omega$ on $\tilde{M}$ asymptotic to $\tilde\Omega_0$ via $\tilde\Phi$. As $\tilde\Phi$ is $\pi_1(M)$-equivariant and $\tilde\Omega_0$ is $\pi_1(M)$-invariant, we obtain that $\gamma^*\tilde\Omega$ is asymptotic to $\tilde\Omega$ for all $\gamma \in \pi_1(M)$. But $\gamma^*\tilde\Omega = h(\gamma)\tilde\Omega$ for some homomorphism $h: \pi_1(M) \to {\mathrm{U}}(1)$, so $h$ is actually trivial.
\end{proof}

\subsubsection{Deformations of negative $\xi$-weight}\label{sss:defns}

Let $C$ be a K\"ahler cone. We will not distinguish between $C$ and the associated normal affine variety $C \cup \{o\}$. Let $\T$ be the Reeb torus of $C$, generated by the flow of the Reeb vector field $\xi = J(r\partial_r) \in {\rm Lie}(\T)$. Our main results (Theorems \ref{existence}--\ref{classification} in Section \ref{s:results}) characterize AC Calabi-Yau manifolds in terms of \emph{deformations of negative $\xi$-weight} of their cones. Unfortunately our definition of these deformations is rather technical; however, it turns out to be easy to check in practice. We suspect that there exists a more natural definition.

\begin{definition}\label{negativeweight}
An affine variety $V$ is a \emph{deformation of negative $\xi$-weight of $C$} if and only if there exists a sequence $\xi_i$ of elements of ${\rm Lie}(\T)$ and a sequence $c_i$ of positive real numbers such that

(1) $\xi_i \to \xi$ as $i \to \infty$;

(2) the vector field $-J(c_i \xi_i)$ generates an effective algebraic $\C^*$-action on $C$;

(3) there exists a $\C^*$-equivariant deformation of $V$ to $C$, i.e., a triple $(W_i,p_i,\sigma_i)$, where
\begin{itemize}
\item $W_i$ is an irreducible affine variety,
\item $p_i: W_i \to \C$ is a regular function with $p_i^{-1}(0) \cong C$ and $p_i^{-1}(t) \cong V$ for $t\neq0$, and
\item $\sigma_i:\C^{*}\times W_i\to W_i$ is an effective algebraic $\C^{*}$-action on $W_i$ such that
\begin{itemize}
\item[$\bullet$] $p_i(\sigma_i(z,x))=z^{\mu_i} p_i(x)$ for some $\mu_i \in \N$ and all $z \in \C^*$, $x \in W_i$, and
\item[$\bullet$] $\sigma_i$ restricts to the $\C^*$-action on $p_i^{-1}(0) \cong C$ generated by $-J(c_i\xi_i)$;
\end{itemize}
\end{itemize}

(4) $\lim_{z\to 0}\sigma_i(z,x)=o$ for every $x\in W_i$, where $o\in C$ is the apex of the cone; and

(5) the sequence $\lambda_{i}:= -(k_i \mu_i)/c_i$ is uniformly bounded away from zero, where $k_i \in \N \cup \{\infty\}$ is the \emph{vanishing order} of the deformation $(W_i, p_i)$, i.e., the supremum of all $k \in \N$ such that $(W_i, p_i)$ becomes isomorphic to the trivial deformation of $C$ after base change to ${\rm Spec}\;\C[t]/(t^k)$.

If $V$ satisfies this condition, then we define the $\xi$-\emph{weight} of $V$ to be the infimum over all possible sequences $\xi_i$ and $(W_i, p_i,\sigma_i)$ as above of $\limsup_{i\to\infty} \lambda_i$. This is a negative real number $\lambda$.
\end{definition}

\subsection{Main results}\label{s:results}

\begin{mtheorem}[Existence]\label{existence}
Let $C$ be a Calabi-Yau cone with $\dim C \geq 2$, with Ricci-flat K\"ahler form $\omega_0$ and Reeb vector field $\xi$. View $C$ as a normal affine variety in the only possible way.

{\rm (1)} Let the affine variety $V$ be a deformation of negative $\xi$-weight of $C$. Then $V$ is Gorenstein with at worst finitely many singularities and with trivial canonical bundle. Moreover, complements of suitable compact subsets of $C$ and of $V$ are diffeomorphic to each other.

{\rm (2)} Let $\pi: M \to V$ be a holomorphic crepant resolution such that the complex manifold $M$ admits a K\"ahler form. Then $M$ is a quasi-projective algebraic manifold.

{\rm (3)} There exists a diffeomorphism $\Phi$ between complements of compact subsets of $C$ and $V$ such that for any resolution $\pi: M \to V$ as in item {\rm (2)}, for any class $\mathfrak{k} \in H^2(M,\R)$ such that $\langle \mathfrak{k}^d, Z\rangle > 0$ for all irreducible subvarieties $Z$ of  ${\rm Exc}(\pi)$, $d = \dim Z > 0$, and for every $g\in\operatorname{Aut}^{T}(C)$, there exists an {\rm AC} Calabi-Yau metric $\omega_{g} \in \mathfrak{k}$ such that $\Phi^*\omega_{g}$ is asymptotic to $g^{*}\omega_0$.
\end{mtheorem}

The rate $\lambda_2$ of \eqref{e:asympt:J}, \eqref{e:asympt:Om} may be taken to be $ \lambda + \epsilon$ for any $\epsilon>0$, where $\lambda$ is the $\xi$-weight of $V$. The best possible metric rate $\lambda_1$ as in \eqref{e:asympt:g} is $\lambda_1 = \max\{\lambda,-2\} + \epsilon$ if $\mathfrak{k}|_L \neq 0$, where $L$ denotes the link at infinity of $V$, and $\lambda_1 = \max\{\lambda, -2n\} + \epsilon$ if $\mathfrak{k}|_L = 0$, where $n=\dim C$.

Theorem \ref{existence} contains all known Tian-Yau type existence theorems for AC Calabi-Yau manifolds \cite{BK, Conlon, Conlon3, goto, JoyceALE, Tian, vanC2} as a special case but is strictly more general. In particular, Theorem \ref{existence} is the first result of this type where $C$ is allowed to be an irregular cone; cf.~Section \ref{ss:examples}.

Note that if $n \geq 3$, then $C$ may admit deformations that are not of negative $\xi$-weight. Also, if $V$ is a deformation of negative $\xi$-weight of $C$, then $V$ may not admit any crepant resolutions at all, or may admit several distinct crepant resolutions not all of which are K\"ahler.

\begin{mtheorem}[Classification]\label{classification}
Every {\rm AC} Calabi-Yau manifold $(M,g,J,\Omega)$ is equivalent to one of the examples produced by the proof of Theorem \ref{existence} up to diffeomorphism.
\end{mtheorem}

As a consequence of Theorem \ref{classification}, every AC Calabi-Yau manifold with a {quasi-regular} asymptotic cone arises from the refinement of the Tian-Yau construction \cite{Tian} given in \cite{Conlon3}.

Any complex affine variety with finite singular set admits a versal family of deformations, which is an affine family over an affine base \cite{Artin, Elkik2}. Any small deformation of the given variety is locally isomorphic to a fiber of this family. In the presence of a $\C^*$-action, it is often possible to globalize this isomorphism. This idea is one of the keys to Kronheimer's classification of ALE hyper-K\"ahler $4$-manifolds \cite[(2.5)]{Kronheimer2}, so we recover this classification here as a consequence of Theorem \ref{classification}.

The same idea applies more generally. For example, it allows us to prove the following.

\begin{mtheorem}[Uniqueness]\label{uniqueness}
Let $(M,g,J,\Omega)$ be an {\rm AC} Calabi-Yau manifold with asymptotic cone $C$. Then the following uniqueness statements hold for $(M,g,J,\Omega)$ up to scaling and diffeomorphism.

{\rm (1)} If $C = \{z_0^2 + \cdots + z_n^2 = 0\} \subset \C^{n+1}$ with its canonical {\rm SO}$(n+1)$-invariant Calabi-Yau cone structure, then either $M = T^*\mathbb{S}^n$ with the Calabi-Yau structure of \cite{Stenzel} $($if $n = 2$, this is true only up to hyper-K\"ahler rotation$)$, or $n = 3$ and $M = \mathcal{O}_{\P^1}(-1)^{\oplus 2}$ with the Calabi-Yau structure of \cite{delaossa}.

{\rm (2)} If $D = {\rm Bl}_{p}\P^2$ and $C$ is the blow-down of the zero section of $K_D$ with the irregular Calabi-Yau cone structure of \cite{Sparks}, then $M$ is either $K_D$ or its flop with one of the Calabi-Yau structures of \cite{goto}.

{\rm (3)} If $C = \C^4/\{\pm {}\rm Id\}$, then $M$ does not exist.
\end{mtheorem}

See Section \ref{ss:proofC} for additional examples as well as alternative proofs of some of these uniqueness results using the classification theory of del Pezzo varieties rather than deformation theory.

In \cite{SpiroJason} a $G_2$ analog of Theorem \ref{uniqueness}(1) concerning the Bryant-Salamon manifolds was proved using a very interesting method of extending Killing fields from the asymptotic cone.

Our proofs of Theorems \ref{classification}--\ref{uniqueness} crucially rely on the fact that $K_M$ is trivial, but for {ALE spaces} as in Theorem \ref{uniqueness}(3) it turns out that no conditions on $K_M$ are required at all \cite{HRS}.

\subsection{Open questions}

(1) We still know only one concrete example of an AC Calabi-Yau manifold that is a smoothing of an irregular Calabi-Yau cone. This is the example of our previous paper \cite{Conlon3}, which was constructed using complicated ad hoc computations and some lucky numerology. Thanks to Theorem \ref{existence}, it is now very easy to reproduce this example via toric geometry and to find its true decay rate, which was not known (see Theorem \ref{toric-eg}). Can one find new examples in this way?

(2) Is there a more natural definition of a deformation of negative $\xi$-weight that does not require approximating the action of $\xi$ by a sequence of $\C^*$-actions?

(3) It should be possible to describe the $\xi$-weight $\lambda$ of $V$ as an eigenvalue of the linear operator induced by $\xi$ on the infinitesimal deformation space of $C$. What can be said about this eigenvalue? For regular Calabi-Yau cones, $\lambda < -1$, and for $3$-dimensional toric Calabi-Yau cones, $\lambda = -3$.

(4) Is the parameter $g$ of Theorem \ref{existence} always caused by diffeomorphism and scaling? This is easy to see in special cases but seems far from clear in general.

(5) Even in view of recent progress \cite{biq-delc, Chiu, CDR, CR, DS2, LiYang, LiuGang, LiuSz1, LiuSz2, Szekel2, Szekel} it seems difficult to generalize our work to the setting of asymptotic cones with singular links, one problem being that algebraic cones with non-isolated singularities have an infinite-dimensional deformation theory. But even the setting of smooth links with {logarithmic} convergence rates is not well-understood.

\subsection{Organization of the paper}

In Section \ref{s:existence} we will prove Theorem \ref{existence}. This proof requires a large number of auxiliary results for which we were unable to find a  reference in the literature, and which may be of some independent interest. We collect these together in four appendices:~a suitable concept of deforming an orbifold to the normal cone of a suborbifold (Appendix \ref{s:deform-to-nc}), a discussion of Type I deformations of Sasaki structures with estimates (Appendix \ref{s:type1}), a new Gysin type theorem for orbifold $S^1$-bundles (Appendix \ref{s:gysin}), and a review and extension to the quasi-regular case of Li's main results in \cite{ChiLi} (Appendix \ref{s:chili}). Theorem \ref{classification} will be proved in Section \ref{s:classification}. This proof is relatively easy compared to the proof of Theorem \ref{existence} because we have tried to make Theorem \ref{existence} as strong as possible. In Section \ref{s:uniqueness}, we first review some background from deformation theory (Section \ref{ss:def-thy}) and then apply this to prove Theorem \ref{uniqueness} (Section \ref{ss:proofC}) and to state a clean existence theorem in the case of toric cones (Section \ref{ss:examples}). In Section \ref{ss:def-thy}, we will also explain why the same arguments allow us to recover Kronheimer's classification of $4$-dimensional hyper-K\"ahler ALE spaces \cite{Kronheimer, Kronheimer2}.

\subsection{Acknowledgments}

We thank V.~Apostolov, T.~Collins, A.~Corti, S.~Donaldson, M.~Faulk, M.~Haskins, C.~LeBrun, C.-C.~M.~Liu, I.~Morrison, R.~Rasdeaconu, F.~Rochon, S.~Sun, I.~\c{S}uvaina, G. Sz\'ekelyhidi, R.~Thomas, F.~Tong, V.~Tosatti, M.~Verbitsky and J.~Viaclovsky for useful discussions over many years. Special thanks go to C.~Li for writing \cite{ChiLi} and for showing us Example \ref{r:chili}, which revealed a mistake in an earlier version of this paper. RC is partially supported by NSF grant DMS-1906466. HH is partially supported by NSF grant DMS-1745517 and by the DFG under Germany's Excellence Strategy EXC 2044-390685587 ``Mathematics M\"unster:~Dynamics-Geometry-Structure," as well as by the CRC 1442 ``Geometry:~Deformations and Rigidity'' of the DFG.

\section{Theorem \ref{existence}}\label{s:existence}

Let $C$ be an $n$-dimensional Calabi-Yau cone ($n \geq 2$) with Reeb vector field $\xi$ and Reeb torus $\T$. Let the affine variety $V$ be a deformation of negative $\xi$-weight of the cone $C$, viewed as a normal affine variety as usual. Thus, we have sequences $\xi_i \in {\rm Lie}(\T)$ and $c_i \in \R^+$ such that

\begin{enumerate}
\item $\xi_i \to \xi$ as $i \to \infty$,

\item the vector field $-J(c_i \xi_i)$ generates an effective $\C^*$-action on $C$,

\item there exists a $1$-parameter degeneration $p_i: W_i \to \C$ of $V$ to $C$, equivariant with respect to some $\C^*$-action $\sigma_i: \C^* \times W_i \to W_i$ that restricts to the action of (2) on the central fiber $C$ and with respect to the standard $\C^*$-action of weight $\mu_i \in \N$ on the base,

\item $\lim_{z\to 0} \sigma_i(z,x) = o$, the apex of $C$, for all $x \in W_i$,

\item and the sequence $\lambda_i := -(k_i\mu_i)/c_i$ is uniformly bounded away from zero, where $k_i \in \N \cup \{\infty\}$ is the supremum of all $k \in \N$ such that the base change of the family $p_i$ from ${\rm Spec}\;\C[t]$ to ${\rm Spec}\;\C[t]/(t^k)$ is isomorphic to a trivial product family.
\end{enumerate}

\noindent Given this set-up, Theorem \ref{existence} will be proved as follows.\medskip

\noindent \emph{Step 1}. Let $D_i$ be the orbifold quotient of $C$ by the effective $\C^*$-action of item (2). Generalizing an idea of Pinkham \cite{Pinkham}, we construct a $\C^*$-equivariant affine embedding $\phi_i: W_i \to \C \times\C^{N_i}$ such that $p_i = {\rm pr}_{\C} \circ \phi_i$ and such that the corresponding weighted projective closure $\overline{V}_i$ of $V \cong \phi_i(p_i^{-1}(1))$ is an orbifold near infinity with compactifying divisor $\overline{V}_i \setminus V = D_i$. Moreover, after removing $D_i$, the deformation of $\overline{V}_i$ to the normal cone of $D_i$ in $\overline{V}_i$ (see Appendix \ref{s:deform-to-nc}) is $\C^*$-equivariantly isomorphic to the base change of $p_i$ from ${\rm Spec}\;\C[t]$ to ${\rm Spec}\;\C[s]$ under the map $t = s^{\mu_i}$.\medskip

\noindent \emph{Step 2}. The existence of the projective compactifications $\overline{V}_i$ allows us to prove Theorem \ref{existence}(1). The triviality of $K_V$ requires many ingredients, including the theory of Type I deformations of Sasaki structures (see Appendix \ref{s:type1}) and a new Gysin type theorem for orbifolds (see Appendix \ref{s:gysin}).\medskip

\noindent \emph{Step 3}. A minor generalization of one of the key results of \cite{ChiLi} (see Appendix \ref{s:chili}) immediately tells us that the embedding of $D_i$ into $\overline{V}_i$ is $(k_i\mu_i-1)$-comfortable. Using a method of \cite{Conlon3, helb, ChiLi}, this allows us to construct a diffeomorphism $\Phi_i$ between neighborhoods of infinity in $C$ and in $V$ such that $|\Phi_i^*(J) - J_0|_{g_{0,i}} = O_i(r_i^{\lambda_i})$ as $r_i \to \infty$, where $J$ is the complex structure on $V$, $g_{0,i}$ is a certain K\"ahler cone metric on $C$ with Reeb vector field $\xi_i$ (see Appendix \ref{s:type1}), and $r_i$ is the radius function of $g_{0,i}$. By construction, $r^{1-\epsilon_i} \leq r_i \leq r^{1+\epsilon_i}$ on $\{r \geq 1\}$, where $\epsilon_i \to 0$. Thus, $J$ is asymptotic to $J_0$ with respect to $g_0$ at rate $\lambda + \epsilon$ for all $\epsilon > 0$. Note that $\lambda := \limsup_{i \to \infty} \lambda_i < 0$ by assumption.  \medskip

\noindent \emph{Step 4}. For a K\"ahler crepant resolution $\pi: M \to V$ and a class $\mathfrak{k}  \in H^2(M,\R)$ that pairs positively with all subvarieties of ${\rm Exc}(\pi)$, we will prove that $\mathfrak{k}$ contains the restriction of a K\"ahler form $\omega_i$ on $M \cup D_i$. As an aside, we will also prove in this step that $M$ is quasi-projective, i.e., Theorem \ref{existence}(2). Any such form $\omega_i$ automatically satisfies $|\Phi_i^*(\omega_i)|_{g_{0,i}} = O_i(r_i^{-2})$ as $r_i \to \infty$. Thus, in analogy with Step 3, for all $\epsilon > 0$ there exist K\"ahler forms on $M$ of decay rate $-2+\epsilon$ with respect to $g_0$.\medskip

\noindent \emph{Step 5}. We are now able to construct AC Ricci-flat K\"ahler metrics on $M$ in the class $\mathfrak{k}$ by solving a complex Monge-Amp\`ere equation as in \cite{Conlon}. This proves Theorem \ref{existence}(3).

\subsection{Step 1}

This step follows from Theorem \ref{normalform} below, which generalizes the idea of \cite[Thm 4.2]{Pinkham} by using weighted rather than unweighted (cf.~the set-up in \cite[\S 3.1]{Pinkham}) projective completions.

The following definition is a minor extension of a definition of Tian-Yau \cite[Defn 1.1(iii)]{Tian}.

\begin{definition}\label{d:adm}
Let $Y$ be a reduced, irreducible, compact complex space. A reduced and irreducible hypersurface $D$ of $Y$ is an \emph{admissible divisor} if there is an open neighborhood $U$ of $D$ such that $U$ is a complex orbifold without codimension-$1$ singularities and $D$ is a complex suborbifold of $U$ with $U^{\rm sing}\subset D$. We then have an associated $\Q$-Cartier divisor or a $\Q$-line bundle $[D]$ on $Y$.
\end{definition}

Note that, unlike in \cite{Tian}, $Y$ is allowed to have singularities in the complement of the open set $U$, and these singularities are not required to be of orbifold type. We will always work in the orbifold category on $U$. Thus, for instance, a \emph{section} of $[D]$ is a section in the usual sense on $Y \setminus D$, and is given by invariant sections on the local uniformizing charts of the orbifold structure on $U$.

If $[D]$ is ample on $Y$, then $Y$ admits a deformation to the ``normal cone'' of $D$; see Appendix \ref{s:deform-to-nc}.

\begin{theorem}\label{normalform}
Let $C$ be an irreducible affine variety with a unique singular point $o \in C$. Let $W$ be an irreducible  affine variety and let $p: W \to \C$ be a regular function such that $p^{-1}(0) \cong C$. Let $\sigma: \C^* \times W \to W$ be an effective algebraic $\C^*$-action such that
\begin{itemize}
\item[$\bullet$] $p(\sigma(z,x)) = z^\mu p(x)$ for some $\mu \in \N$ and all $z \in \C^*, x \in W$, and

\item[$\bullet$] $\lim_{z\to 0} \sigma(z,x) = o$ for all $x \in W$.
\end{itemize}
\noindent Then the following hold.

{\rm (1)} There exist $\mu_1, \ldots, \mu_N \in \N$ coprime and an embedding $\phi: W \to \C_t \times \C^N_{z_1, \ldots, z_N}$ such that
\begin{itemize}
\item[$\bullet$] $p= t \circ \phi$, and

\item[$\bullet$] $\phi(\sigma(z,x)) = {\rm diag}(z^\mu, z^{\mu_1}, \ldots, z^{\mu_N})\cdot \phi(x)$ for all $z \in \C^*$, $x \in W$.
\end{itemize}

{\rm (2)} Denote the affine variety $p^{-1}(1)$ by $V$. Let $\overline{V}$ be the closure of $\phi(V)$ in the weighted projective space $\P(\mu, \mu_1, \ldots, \mu_N, 1) \supset \C_t \times \C^N_{z_1,\ldots,z_N}$ and let $D = \overline{V} \setminus \phi(V)$ be the compactifying divisor. Then $D$ is admissible in $\overline{V}$ with respect to the orbifold structure inherited from $\P(\mu,\mu_1,\ldots,\mu_N,1)$.

{\rm (3)} The base change of $p$ from $\C_t$ to $\C_s$ under the map $t = s^\mu$ is $\C^*$-equivariantly isomorphic to the deformation of $\overline{V}$ to the normal cone of $D$ with its compactifying divisor removed.
\end{theorem}

\begin{proof}
(1) Decompose the coordinate ring of $W$ as the algebraic direct sum of the eigenlines of the $\C^*$-action induced by $\sigma$. Note that $p$ is a $\mu$-eigenfunction. If $f$ is any nonconstant $\lambda$-eigenfunction, then $\lambda > 0$ by letting $z \to 0$ in the identity $f(\sigma(z,x)) = z^\lambda f(x)$, where $f(x) \neq f(o)$. Given an affine embedding of $W$, we can decompose each coordinate function into eigenfunctions of the $\C^*$-action. The set of all eigenfunctions obtained in this way defines a new affine embedding of $W$ because the new functions still separate points and Zariski tangent directions. The new embedding is obviously equivariant, with coprime eigenvalues because $\sigma$ is effective. We may add $p$ as a coordinate.

To prove (2) and (3), we begin by observing that the ideal of $\phi(W)$ is generated by a finite set of polynomials that are homogeneous with respect to the $\C^*$-action with weights $\mu, \mu_1, \ldots, \mu_N$. Clearly these polynomials can be written in the form $f_i(z_1, \ldots, z_N) + t g_i(t, z_1, \ldots, z_N)$ ($i = 1, \ldots, I$), where $f_i,g_i$ are homogeneous polynomials as well. Let $[\tau,\zeta_1,\ldots,\zeta_N,w]$ denote the natural homogeneous coordinates on $\P(\mu,\mu_1,\ldots,\mu_N,1)$, so that the inclusion of $\C \times \C^N$ as an affine chart is given by
\begin{equation}\tau = t w^\mu, \;\, \zeta_n = z_n w^{\mu_n} \;\,(n = 1, \ldots, N)\label{e:fcku}\end{equation}
and the closure of $\phi(W)$ is cut out by the equations
\begin{align}\label{e:fcku2}
f_i(\zeta_1, \ldots, \zeta_N) + \tau g_i(\tau, \zeta_1, \ldots, \zeta_N) = 0 \;\,(i = 1, \ldots, I).
\end{align}

For $t \in \C$, let $V_t = p^{-1}(t)$, let $\overline{V}_t$ denote the closure of $\phi(V_t)$, and let $D_t = \overline{V}_t \setminus \phi(V_t)$ denote the compactifying divisor. A priori $\overline{V}_t$ may be arbitrarily singular at or near the hyperplane $w = 0$, and if $\overline{V}_t$ does not intersect this hyperplane transversely then $D_t$ may carry multiplicities. The content of item (2) is precisely to rule out these possibilities for $t = 1$ by using the fact that the picture for $t = 0$ is completely understood and that we have a family of varieties connecting $\overline{V}$ to $\overline{V}_0$.

By construction, $\overline{V}_0$ intersects the hyperplane $w = 0$ transversely in a reduced and irreducible subvariety $D_0$, which is a connected suborbifold of $\P(\mu_1,\ldots,\mu_N)$ isomorphic to $(C \setminus \{o\})/\C^*$ as an orbifold. In particular, because $\phi(V_0)$ is smooth in a neighborhood of infinity, $D_0$ is admissible in $\overline{V}_0$. Moreover, $D_0$ is cut out by the equations $f_i(\zeta_1, \ldots, \zeta_N) = 0$. For $t \neq 0$, it directly follows from \eqref{e:fcku} and \eqref{e:fcku2} that $D_t$ is a subscheme of $D_0$, but in general this inclusion could be strict. However, since $W$ is irreducible, all fibers of $p$ are purely of dimension $(\dim W)-1$. Thus, $D_t$ is a full-dimensional projective subscheme of the irreducible projective variety $D_0$, and hence $D_t  = D_0$. In particular, $D_t$ is reduced, so $\overline{V}_t$ intersects the hyperplane $w = 0$ transversely. We abbreviate $D_t = D_0 = D$.

We now prove item (2) by studying the family $\overline{V}_t$ locally near an arbitrary point $x$ of its common compactifying divisor $D$. Pick any $n \in \{1,\ldots,N\}$ such that $\zeta_n(x) \neq 0$. As usual, the locus $\zeta_n \neq 0$ in $\P(\mu,\mu_1,\ldots,\mu_N,1)$ can be identified with $\C^{N+1}/\Z_{\mu_n}$, where a $\mu_n$-th root of unity $\xi \in \Z_{\mu_n}$ acts by the diagonal matrix ${\rm diag}(\xi^\mu, \xi^{\mu_1}, \ldots, \xi^{\mu_{\hat{n}}}, \ldots, \xi^{\mu_N},\xi)$. Let $\tilde{x}$ be a preimage of $x$ in $\C^{N+1}$ and let $\Gamma$ be the stabilizer of $\tilde{x}$ in $\Z_{\mu_n}$. Then we may write a small open neighborhood of $x$ as $\tilde{U}/\Gamma$, where $\tilde{U}$ is a ball centered at $\tilde{x}$ in $\C^{N+1}$. Let $\tilde{V}_t, \tilde{D}$ denote the corresponding local lifts of $\overline{V}_t, D$. Thus, $\tilde{D}$ is a smooth submanifold of $\tilde{U}$ containing $\tilde{x}$, and $\tilde{V}_t$ is a pure-dimensional analytic subvariety of $\tilde{U}$ containing $\tilde{D}$ as a divisor. In particular, $\dim \tilde{V}_t = (\dim \tilde{D}) + 1$ for all $t \in \C$. We already know that $\tilde{{V}}_0$ is smooth in $\tilde{U}$, and to prove (2) we need to show that $\tilde{V}_1$ is smooth at $\tilde{x}$.

Let $(\tilde\tau, \tilde\zeta_1, \ldots, \tilde\zeta_{\hat{n}}, \ldots, \tilde\zeta_N, \tilde{w})$ denote the natural affine coordinates on $\C^{N+1}$ provided by the above construction. Then an element $z \in \C^*$ acts via $\sigma$ by fixing the first $N$ coordinates and multiplying $\tilde{w}$ by $z^{-1}$. Thus, by construction, if $t$ is a small positive real number, then the action of $t^{-1/\mu}$ via $\sigma$ identifies $\tilde{V}_t$ with a (smaller and smaller as $t \to 0^+$) tubular neighborhood of $\tilde{D}$ in $\tilde{V}_1$. We therefore reduce to proving that $\tilde{V}_t$ is smooth at $\tilde{x}$ for all sufficiently small values of $t \neq 0$.

To prove this, recall that the union of the varieties $\tilde{V}_t$ in $\tilde{U}$ is cut out by the equations
$$f_i(\tilde{\zeta}_1, \ldots, 1, \ldots, \tilde{\zeta}_N) + \tilde\tau g_i(\tilde\tau,\tilde{\zeta}_1, \ldots, 1, \ldots \tilde{\zeta}_N) = 0,$$
and that, for any fixed $t\in\C$, the variety $\tilde{V}_t$ is cut out by the additional equation $$\tilde\tau = t \tilde{w}^\mu.$$
Since $\tilde{V}_t$ is purely $((\dim \tilde{D})+1)$-dimensional, the proof of \cite[p.100, Prop 1.104]{singing} now tells us that there exists an explicit set of polynomials in $\tilde\tau, \tilde\zeta_1, \ldots, \tilde\zeta_{\hat{n}}, \ldots, \tilde\zeta_N,\tilde{w}$ and $t$ that cut out the singular locus of $\tilde{V}_t$ for all $t \in \C$. Since $\tilde{V}_0$ is smooth, it follows that $\tilde{V}_t$ is smooth at $\tilde{x}$ for $t \neq 0$ small.

(3) We may view the weighted homogeneous coordinate function $w$ on $\P(\mu,\mu_1,\ldots,\mu_N,1) \supset \overline{V}$ as a global section of the line orbibundle $\mathcal{O}(1)$. The zero locus of this section on $\overline{V}$ is the admissible divisor $D$ with multiplicity $1$. It follows that the $\Q$-line bundle induced by $D$ on $\overline{V}$ is $L = \mathcal{O}(1)|_{\overline{V}}$. This is a line bundle away from $D$ and a line orbibundle in a small tubular neighborhood of $D$, and is ample as a $\Q$-line bundle. Let $\pi:L \to \overline{V}$ denote the bundle projection. Let $\sigma$ denote the section of $L$ that cuts out $D$, which is given by the restriction to $\overline{V}$ of the section $w$. Then away from its central fiber, the deformation of $\overline{V}$ to the normal cone of $D$ can be written as
\begin{align}\label{e:sickness1}
\{(q,s) \in L \times \C^*: s \cdot q = \sigma(\pi(q))\},
\end{align}
where $s$ is the family parameter. We now pass to the dual line bundle and contract its zero section to a point. By definition, $(L^*)^\times$ is the weighted affine cone $C(\overline{V})$ over $\overline{V}$ in $\C^{N+2}$.
Write a general point of $\C^{N+2}$ as $p = (\tau, \zeta_1, \ldots, \zeta_N,w)$. In this picture, \eqref{e:sickness1} turns into
\begin{equation}\label{e:sickness2}
\{(p,s)\in C(\overline{V}) \times \C^*: w=s\}
\end{equation}
because away from the divisor $D$, the section of $L^* = \mathcal{O}(-1)|_{\overline{V}}$ dual to $\sigma$ picks out the unique point with $w = 1$ in each $\C^*$-orbit in $C(\overline{V})$, so the equation dual to $s \cdot q = \sigma(\pi(q))$ is exactly $w = s$. This equation is still correct at $s = 0$ because the central fiber of the deformation to the normal cone is the hyperplane section $w = 0$ of $(L^*)^\times = C(\overline{V})$. Now recall that $C(\overline{V})$ is cut out by
$$f_i(\zeta_1,\ldots,\zeta_N) + \tau g_i(\tau,\zeta_1, \ldots, \zeta_N) = 0, \;\, \tau = w^\mu.$$
Then it is clear that the family \eqref{e:sickness2}, extended to $s = 0$, is the base change of the original family $W$ or $\phi(W) \subset \C_t \times \C^{N}_{z_1,\ldots,z_N}$ under the map $t = s^\mu$, as required.
\end{proof}

\begin{remark}\label{r:chili3}
As in \cite[Thm 4.2]{Pinkham}, the singularity $o \in C$ is not required to be normal, even though in our applications it always will be. If this singularity is not normal, then Proposition \ref{normalform} yields an example of a ``deformation to the normal cone'' whose central fiber is not actually isomorphic to the normal cone of the compactifying divisor (but is still normalized by it); cf.~Example \ref{r:chili2}.
\end{remark}

\subsection{Step 2}

We now establish Step 2 of the outline at the beginning of this section. Step 1 yields a sequence of projective compactifications $\overline{V}_i$ of our negative weight deformation $V$ of the Calabi- Yau cone $C$. For $i\gg 0$ each of these can be used to prove Theorem \ref{existence}(1). This is the content of our next theorem. The proof actually yields a more technical but stronger statement (see Remark \ref{rk:strongerprop}), which will be needed later on to complete Step 4 of the outline.

\begin{theorem}\label{theo:properties}
The affine variety $V$ is Gorenstein with at worst finite singular set and $K_V$ is trivial. Complements of suitable compact subsets of $C$ and of $V$ are diffeomorphic to each other.
\end{theorem}

\begin{proof}
We treat the surface case separately, using an argument due to Kronheimer \cite[(2.5)]{Kronheimer2} that will also give us Kronheimer's classification of hyper-K\"ahler ALE $4$-manifolds once Theorem \ref{classification} has been proved. As the details of this argument will be explained in a broader context in Section \ref{ss:def-thy}, we will be brief here. If $\dim C = 2$, then $C = \C^2/\Gamma$ for some finite group $\Gamma \subset {\rm SU}(2)$ acting freely on $S^3$. In particular, $C$ is quasi-regular, i.e., the Reeb torus of $C$ is a circle, so the above sequence $W_i$ collapses to a single element $W$. It is a classical result going back to Klein that $C$ can be embedded into $\C^3$ as a quasi-homogeneous surface. Thanks to Slodowy's construction \cite{Slodowy} of a $\C^*$-equivariant semi-universal deformation of any quasi-homogeneous complete intersection, we can now conclude that $V$ itself embeds into $\C^3$ as a perturbation of $C$ by polynomials of lower weighted degree. Here we crucially rely on the presence of a $\C^*$-action in order to globalize the classifying map from $W$ to the semi-universal family of $C$. This classification makes the claimed properties of $V$ checkable.

We now give a more abstract argument that covers all cases where $\dim C \geq 3$.

Proposition \ref{normalform}(2) tells us that $\overline{V}_i$ is an orbifold in a neighborhood of the compactifying divisor $D_i$ such that all of the singularities of $\overline{V}_i$ in this neighborhood are contained in $D_i$. In particular, this neighborhood is orbifold diffeomorphic (fixing the zero section) to a neighborhood of the zero section in the orbifold normal bundle to $D_i$ in $\overline{V}_i$, i.e., in the compactified cone $C \cup D_i$. Also, the singular set of the affine variety $V$ is then clearly compact and hence finite. The fact that the singular points are Gorenstein follows from \cite[Thm 9.1.6]{Ishii}, using the given deformation of $V$ to $C$.

Since $V$ is Gorenstein, the canonical divisor class of $V$ defines a line bundle $K_V$. Similarly, since $\overline{V}_i = V \cup D_i$ is an orbifold in a tubular neighborhood of $D_i$, the canonical bundle $K_{\overline{V}_i}$ makes sense as a line orbibundle near $D_i$ and as a $\Q$-line bundle globally. We will prove that for all $i \gg 0$ there exists an integer $\alpha_i$ (in fact $\alpha_i > 1$) such that $K_{\overline{V}_i} + \alpha_i[D_i]$ is an honest line bundle on $\overline{V}_i$, and is actually trivial as a line bundle. This clearly suffices to prove the theorem.

We first note that $D_i$ must be a Fano orbifold for all $i \gg 0$. This is because the theory of Type I deformations of Sasaki structures (see Theorem \ref{p:typeIest}) provides us with a sequence of cone metrics $g_{0,i}$ on $C$ that are K\"ahler with respect to the given complex structure, with Reeb vector field $\xi_i$, such that $g_{0,i}$ converges to the Ricci-flat K\"ahler cone metric $g_0$ locally smoothly as $i \to \infty$. It follows that the link of $g_{0,i}$ has strictly positive Ricci curvature for $i \gg 0$, and hence that $D_i$ is Fano.

Fix any large enough index $i$ such that $D_i$ is Fano, and suppress the subscript $i$.

Our next step is to prove that the desired result is true to leading order around $D$, i.e., that an analogous statement holds on the normal cone to $D$ in $\overline{V}$. More precisely, let $p: N \to D$ denote the total space of the normal line orbibundle to $D$ in $\overline{V}$. We wish to prove that there exists an integer $\alpha > 1$ such that $K_N + \alpha[D]$ is an honest line bundle on $N$, and is actually trivial as a line bundle. To this end, we first observe that $K_N$ is trivial away from the zero section of $N$ because the affine cone $(N^{-1})^{\times}$ is exactly the given Calabi-Yau cone $C$. On the other hand, one easily shows by using transition functions that
$K_{N}=p^{*}(K_{D} - N)$, where $p^*$ is understood in the sense of \cite[Rmk 4.30]{faulk}. Thus, applying Theorem \ref{th:gysin} and switching to additive notation, we see that there exists an $\alpha \in \Z$ such that $K_D - N = -\alpha N$ as $C^\infty$ line orbibundles on $D$. Notice that $\alpha > 1$ because $D$ is Fano and $N$ is positive. Again because $D$ is Fano, the orbifold Picard group of $D$ is discrete, so it follows that $K_D - N = -\alpha N$ as holomorphic line orbibundles. Since $p^*N = [D]$, we learn that the holomorphic line orbibundle $K_N +  \alpha[D]$ is trivial, hence in particular that it is an honest line bundle.

We now consider the $\Q$-line bundle $Q=K_{\overline{V}} + \alpha[D]$ on $\overline{V}$. By construction, this is an honest line bundle away from $D$ and a line orbibundle in a tubular neighborhood of $D$. By the previous step, $Q|_D$ is trivial. In particular, $Q|_D$ is an honest line bundle on $D$. Moreover, since $Q$ and $K_N + \alpha[D]$ are isomorphic as $C^\infty$ line orbibundles in a tubular neighborhood of $D$, it follows that $Q$ is a true line bundle in this neighborhood, hence globally on $\overline{V}$. Pick a trivializing section $s \in H^0(D,Q|_D)$. We will now prove the following sequence of statements, where (3) completes the proof.

(1) \emph{The section $s$ extends to every infinitesimal neighborhood $mD$ $(m \in \N$, $m \geq 2)$ of $D$ in $\overline{V}$.}

\emph{Proof}. Taking cohomology in the sequence
$$0\to\mathcal{O}_{D}(Q)\otimes\mathcal{J}_D^m|_D\to\mathcal{O}_{mD}(Q)\to\mathcal{O}_{(m-1)D}(Q)\to 0,$$
whose exactness can be checked on stalks by averaging over the uniformizing groups, we get
\begin{equation*}
H^{0}(\mathcal{O}_{mD}(Q))\to H^{0}(\mathcal{O}_{(m-1)D}(Q))\to
H^{1}(\mathcal{O}_{D}(Q)\otimes\mathcal{J}_D^m|_D).
\end{equation*}
A section of $Q|_{(m-1)D}$ therefore lifts to a section of $Q|_{mD}$ if
$H^{1}(\mathcal{O}_{D}(Q)\otimes\mathcal{J}_D^m|_D)=0$. Since $Q|_{D}$ is trivial, this is equivalent to $H^{1}(\mathcal{J}_D^m|_D)=0$. As in the manifold case, the sheaf $\mathcal{J}_D^m|_D$ is naturally  isomorphic to the sheaf of sections of the line orbibundle $-mN$. The desired vanishing of $H^1$ then follows from the Kodaira vanishing theorem for orbibundles because $N$ is positive and $\dim D\geq 2$. By induction, $s \in H^0(D, Q|_D)$ lifts to a section of $Q|_{mD}$ for all $m\geq2$.

(2) \emph{The section $s$ extends to a small pseudoconcave tubular neighborhood $U$ of $D$ in $\overline{V}$.}

\emph{Proof}. For this we take cohomology in the exact sequence
$$0\to\mathcal{O}_{U}(Q) \otimes \mathcal{J}_D^m \to\mathcal{O}_{U}(Q)\to\mathcal{O}_{mD}(Q)\to0,$$
which results in the exact sequence
$$H^{0}(\mathcal{O}_{U}(Q))\to H^{0}(\mathcal{O}_{mD}(Q))\to H^{1}(\mathcal{O}_{U}(Q) \otimes \mathcal{J}_D^m).$$
If $D$ is a smooth divisor with $\dim D \geq 2$ and with positive normal bundle, then the $H^1$ term vanishes for $m \gg 1$ by \cite[p.379, Thm III]{Griffiths}. This is proved using the identification of $\mathcal{J}_D^m$ with the sheaf of sections of $-m[D]$, the theory of $\overline{\partial}$-harmonic forms on manifolds with boundary, and the Bochner technique. All of these go through for admissible orbifold divisors. Together with the fact proved in item (1) above that $s$ extends to $mD$ for all $m \geq 2$, this then tells us that $s$ extends to $U$.

(3) \emph{The line bundle $Q$ is globally trivial on $\overline{V}$.}

\emph{Proof}. Thanks to item (2), $Q$ is trivial in some tubular neighborhood $U$ of $D$ in $\overline{V}$, after shrinking $U$ to ensure that $s$ does not vanish on $U$. Choose $k\in \N$ sufficiently large and divisible such that $k[D]$ is a very ample Cartier (rather than $\Q$-Cartier) divisor on $\overline{V}$. We may then perturb the supporting Weil divisor $k D$ in its linear system, replacing it with a generic smooth hyperplane section $H \subset U$. Then $Q|_H$ will be trivial because $Q|_U$ is. Because $\dim H \geq 2$, the Grothendieck-Lefschetz theorem for normal projective varieties \cite[Thm 1]{Srinivas} now tells us that $Q$ is globally trivial.
\end{proof}

\begin{remark}\label{rk:strongerprop}
We note for later reference that the above proof shows that $V$ is Gorenstein, $\overline{V}_i$ is $\Q$-Gorenstein for all $i$, $D_i$ is Fano for all $i \gg 0$, and if $D_i$ is Fano then there exists an integer $\alpha_i > 1$ such that the $\Q$-line bundle $K_{\overline{V}_i} + \alpha_i [D_i]$ is an honest line bundle on $\overline{V}_i$, and is actually trivial.
\end{remark}

\subsection{Step 3}\label{ss:three}

By Theorem \ref{chili}, the orbidivisor $D_i$ is $(k_i\mu_i - 1)$-comfortably embedded in $\overline{V}_i$. (This result is valid only for $n \geq 3$, but for $n = 2$ there is no need to work with approximating sequences, so Theorem \ref{chili} is not needed either.) Thus, by \cite[\S 3.1]{ChiLi} (see also \cite[App B.1]{Conlon3} and \cite[Prop 4.5]{helb}), there exists a diffeomorphism $\Phi_i$ of neighborhoods of infinity in $C$ and $V$ such that $\Phi_i^*J$ converges to $J_0$ at rate $-k_i\mu_i$ with respect to any $J_0$-K\"ahler cone metric on $C$ whose scaling vector field generates the effective $\C^*$-action by which we have quotiented $C$ to obtain $D_i$. We would now like to rewrite this information in terms of $\xi_i$. Recall that $c_i$ was defined as the unique positive real number such that $-J(c_i\xi_i)$ generates the $\C^*$-action we just mentioned, and $-k_i\mu_i/c_i = \lambda_i$ also holds by definition. Now, the metric growth rate of a section of $T^*C \otimes TC$ with respect to any cone metric is equal to the section's homogeneity with respect to the scaling vector field of the cone metric. Thus, if $g_{0,i}$ is a $J_0$-K\"ahler cone metric on $C$ with Reeb vector field $\xi_i$, and if $r_i$ denotes the radius function of this cone metric, then $|\nabla_{g_{0,i}}^k(\Phi_i^*J - J_0)|_{g_{0,i}} = O_{k,i}(r_i^{\lambda_i-k})$ as $r_i \to \infty$ for all $k \in \N_0$. Such metrics $g_{0,i}$ exist by Theorem \ref{p:typeIest}. In addition, by \eqref{e:saviors}, we can assume that for all $K \in \N$ and $\epsilon > 0$, $g_{0,i}$ converges to $g_0$ in the weighted $C^{K}_\epsilon$ norm associated with $g_0$ as $i \to \infty$. Thus, for all $i \geq i(K,\epsilon)$, we have that $|\nabla^k_{g_0}(\Phi_i^*J - J_0)|_{g_0} = O_{K,\epsilon,i}(r^{\lambda+\epsilon-k})$ as $r \to \infty$ for all $k \in \{0,\ldots,K\}$. This will be good enough for us even though $\Phi_i$ and the constants in the $O_{K,\epsilon,i}$ notation diverge as $i \to \infty$.

\subsection{Step 4}

Given a K\"ahler crepant resolution $\pi: M \to V$, our goal in this step is to find K\"ahler forms with good asymptotics in as many classes $\mathfrak{k} \in H^2(M,\R)$ as possible. We will explain how to do this after proving the following technical theorem.

\begin{theorem}\label{wh}
Let $V$ be an affine variety with at worst isolated singularities. Let $\overline{V}$ be a projective compactification of $V$ such that $D = \overline{V}\setminus V$ is an ample admissible divisor in $\overline{V}$. Let $M$ be a complex manifold together with a proper surjective holomorphic map $\pi: M \to V$ which is a biholomorphism away from compact sets. Construct a compact complex orbifold $X = M \cup D$ by using $\pi$ to identify neighborhoods of infinity in $V$ and in $M$. Then the following hold.

{\rm (1)} The orbifold de Rham cohomology of $X$ satisfies Hodge decomposition and Hodge symmetry.

{\rm (2)} If $V$ is Gorenstein and if $\pi$ is crepant, then the singularities of $V$ are canonical.

{\rm (3)} If in addition there exists an $\alpha \in \N$ such that the $\Q$-line bundle $K_{\overline{V}} + \alpha[D]$ is an honest line bundle on $\overline{V}$, and is actually trivial as a line bundle, then $h^{0,i}(X) = 0$ for all $i > 0$.

{\rm (4)} If in addition $\alpha > 1$,  then every smooth closed real $2$-form on $M$ is de Rham cohomologous to the restriction to $M$ of a smooth closed real $(1,1)$-form on $X$.

{\rm (5)} Assume in addition that $M$ admits a smooth closed real $(1,1)$-form $\beta$ whose restriction to some open neighborhood $U$ of ${\rm Exc}(\pi)$ is $i\partial\overline{\partial}$-cohomologous to a K\"ahler form on $U$. Then $X$ is projective, and any such $\beta$ is $i\partial\overline{\partial}$-cohomologous to the restriction to $M$ of a smooth K\"ahler form on $X$.
\end{theorem}

\begin{proof}
(1) Abuse notation by denoting the obvious extension of $\pi$ from $M$ to $X$ by $\pi$ as well.

We  first show that the blow-up $\sigma: \hat{X} \to X$ of $X$ along a suitable ideal sheaf supported in $X \setminus D$ produces a projective orbifold $\hat{X}$. This follows from \cite[p.321, Thm 1]{HiRossi}. More precisely, by applying the proof of this theorem to a small neighborhood of every singularity of $V$, we obtain a coherent analytic ideal sheaf $\mathcal{J}$ on $V$ with support on $V^{\rm sing}$ such that the blow-up $f: \hat{X} \to \overline{V}$ of $\overline{V}$ along $\mathcal{J}$, which is projective by GAGA, factors as $f = \pi \circ \sigma$, where $\sigma: \hat{X} \to X$ is then the desired map.

If $X$ is smooth, the modification $\sigma: \hat{X} \to X$ can be used to prove that the de Rham cohomology of $X$ satisfies Hodge decomposition and Hodge symmetry \cite[Thm 2.2.18]{Ma1}. The same proof goes through in the general orbifold case if one uses orbifold differential forms, thanks to the fact that $\sigma$ is an orbifold diffeomorphism near the orbifold singularities of $\hat{X}$ and $X$. Moreover, this proof also shows that the pullback map $\sigma^*$ in orbifold de Rham or Dolbeault cohomology is injective.

(2) By assumption, $\pi$ is a crepant resolution of the singularities of $V$. If $\pi$ is algebraic, it follows by definition that $V$ has canonical singularities. In general, we just need to consider the algebraic map $f = \pi \circ \sigma$ instead of $\pi$, whose discrepancies over $V^{\rm sing}$ are nonnegative because $M$ is smooth.

(3) Thanks to the injectivity of $\sigma^*$, it suffices to prove that $H^i(\hat{X}, \mathcal{O}_{\hat{X}}) = 0$ for $i > 0$. Note that the direct image functor $R^q f_*$ on coherent sheaves can be computed locally in the analytic topology. This implies that $R^q f_*\mathcal{O}_{\hat{X}} = 0$ for $q > 0$ because $\overline{V}$ has canonical singularities away from $D$ by item (2) above, which are rational by \cite[Thm 5.22]{KM}, and because $f$ is a biholomorphism onto its image in a neighborhood of $D$. Thus, $H^i(\hat{X}, \mathcal{O}_{\hat{X}}) = H^i(\overline{V}, \mathcal{O}_{\overline{V}})$ by the Leray spectral sequence. To prove that $H^i(\overline{V}, \mathcal{O}_{\overline{V}}) = 0$ for $i > 0$, we begin with the following claim; cf.~\cite[p.526]{Conlon3}.\medskip

\noindent \emph{Claim 1.} If a holomorphic line orbibundle $Q$ on $\hat{X}$ admits a Hermitian metric whose curvature form is nonnegative on $\hat{X}$, and is strictly positive somewhere, then $H^i(\hat{X}, K_{\hat{X}} + Q) = 0$ for all $i > 0$.\medskip

\noindent \emph{Proof of Claim 1.} Since $\hat{X}$ is projective, hence K\"ahler, this follows from the usual Bochner formula proof of the Kodaira vanishing theorem together with the real-analyticity of harmonic forms. This argument is due to \cite[Thm 6]{Riemen2} in the smooth case and it clearly works for orbifolds as well. \hfill$\Box$\medskip

To proceed, write $\hat{D}$ for the preimage of $D$ in $\hat{X}$. Note that $K_{\hat{X}} + \alpha[\hat{D}]$ is an $f$-exceptional divisor on $\hat{X}$ because $K_{\overline{V}} + \alpha[D] = 0$ on $\overline{V}$. This implies that $f_*(K_{\hat{X}} + \alpha[\hat{D}]) = 0$ and hence
$$H^i(\overline{V}, \mathcal{O}_{\overline{V}}) = H^i(\overline{V}, f_*(K_{\hat{X}} + \alpha[\hat{D}])).$$
We can further rewrite this as a cohomology group on $\hat{X}$ by using the Leray spectral sequence one more time. To this end, we require the following claim.\medskip

\noindent \emph{Claim 2.} We have that $R^q f_*(K_{\hat{X}} + \alpha[\hat{D}]) = 0$ for all $q > 0$. \medskip

\noindent \emph{Proof of Claim 2.} By \cite[Prop 2.69]{KM}, it suffices to prove that $H^q(\hat{X}, K_{\hat{X}} + \alpha[\hat{D}] + f^*H) = 0$ for all $q > 0$ and all sufficiently ample line bundles $H$ on $\overline{V}$. This vanishing follows from Claim 1. Indeed, $D$ is an ample $\Q$-Cartier divisor on $\overline{V}$. Thus, for $m \in \N$ sufficiently large and divisible, the map
$$\hat{X} \stackrel{f}{\longrightarrow} \overline{V} \stackrel{|mD|}{\longrightarrow} \mathbb{P}^N$$
is everywhere defined, and is an isomorphism onto its image on $\hat{X}\setminus {\rm Exc}(f)$. Then the pullback of the Fubini-Study Hermitian metric on $\mathcal{O}_{\P^N}(1)$ under this map is a metric of nonnegative curvature on $m[\hat{D}]$ with strictly positive curvature on $\hat{X}\setminus {\rm Exc}(f)$. Taking the $m$-th root, we obtain a metric with the same properties on $[\hat{D}]$, and this metric can be used in Claim 1.\hfill$\Box$  \medskip\

Thanks to Claim 2, the Leray spectral sequence now tells us that
$$H^i(\overline{V}, f_*(K_{\hat{X}} + \alpha[\hat{D}]) ) = H^i(\hat{X}, K_{\hat{X}} + \alpha[\hat{D}]).$$
This vanishes again due to Claim 1 and the fact that $[\hat{D}]$ admits a suitable Hermitian metric.

(4) We need to show that the restriction map $H^{1,1}(X,\R) \to H^2(X\setminus D, \R)$ is surjective. This will be similar to the proof of \cite[Prop 2.5]{Conlon3}. By items (1) and (3), $H^{1,1}(X,\R) = H^{2}(X,\R)$. Moreover, since $\alpha>1$, $D$ is Fano by adjunction. Thus, by the orbifold Calabi-Yau theorem,
$D$ admits K\"ahler metrics of positive Ricci curvature, so $H^{1}(D,\R)=0$ by the usual Bochner argument and orbifold Hodge theory. Now recall the Gysin exact sequence
\begin{equation*}
\cdots \to H^{k-2}(D,\R) \to H^{k}(X,\R)\stackrel{i^*}{\to} H^{k}(X\setminus D,\R)\to
H^{k-1}(D,\R)\to\cdots
\end{equation*}
in orbifold de Rham cohomology \cite[Prop B.4]{Conlon3}, where $i:X\setminus D \to X$ is the inclusion. Combining these facts, surjectivity of the restriction map $H^{1,1}(X,\R)\to H^{2}(X\setminus D,\R)$ is now obvious.

(5) Let $\beta$ be a smooth closed real $(1,1)$-form on $M$ such that for some open neighborhood $U$ of ${\rm Exc}(\pi)$ we have that $\beta|_U = \omega - i\partial\overline{\partial}\varphi$, where $\omega$ is a K\"ahler form on $U$ and $\varphi: U \to \R$ is a smooth function. Thanks to item (4) above, there exist a closed $(1,1)$-form $\xi$ on $X$ and a $1$-form $\eta$ on $M$ such that $\beta =\xi|_M+d\eta$. Note that $M$ is strictly pseudoconvex at infinity because the compactifying divisor $D = X \setminus M$ has positive normal bundle. Moreover, $K_M$ is trivial because $K_V$ is trivial and $\pi: M \to V$ is crepant. Thus, $H^{1}(M,\mathcal{O}_{M})=0$ by \cite[p.278, Korollar]{GRie}. Since $d\eta$ is of type $(1,1)$, the usual proof of the $i\partial\overline{\partial}$-lemma then tells us that $d\eta=i\partial\overline{\partial}u$ for some function $u$ on $M$.
Now let $\gamma$ denote the pullback of the Fubini-Study K\"ahler form on $\P^N$ under the map
$$X \stackrel{\pi}{\longrightarrow} \overline{V} \stackrel{|mD|}{\longrightarrow} \P^N$$
for any sufficiently large and divisible $m \in \N$. Then, as in the proof of Claim 2 above, $\gamma$ is a smooth nonnegative closed $(1,1)$-form on $X$, strictly positive on $X \setminus {\rm Exc}(\pi)$ and $i\partial\overline{\partial}$-exact on $M = X \setminus D$. Let $\chi$ be a cut-off function with compact support on $M$ such that the support of $d\chi$ is contained in $U \setminus {\rm Exc}(\pi)$. Then, clearly, for $C \gg 1$ sufficiently large,
$$\xi+i\partial\overline{\partial}(\chi(u+\varphi)) + C \gamma$$ defines a K\"ahler form on $X$ whose restriction to $M$ is $i\partial\overline{\partial}$-cohomologous to $\xi$, hence to $\beta$. This in particular tells us that $X$ is a K\"ahler orbifold. Thus, using the vanishing $h^{0,2}(X) = 0$ from item (3), $X$ is projective algebraic by the Kodaira-Baily embedding theorem \cite{Bai57}.
\end{proof}

\begin{remark}
If $X$ is actually a manifold, then several steps of the above proof can be simplified. For example, it is then not necessary to ensure that $\sigma$ is an isomorphism near $D$, which allows us to quote a standard property of Moishezon manifolds \cite[Thm 2.2.16]{Ma1} instead of \cite{HiRossi}. Moreover, the vanishing $h^{0,i}(X) = 0$ follows from \cite[Satz 2.1]{GRie} because all sheaves in sight are locally free.
\end{remark}

We are now ready to complete Step 4.  If $i \in \N$ is large enough, then by Remark \ref{rk:strongerprop} the projective compactification $\overline{V}_i = V \cup D_i$ constructed in Step 2 satisfies all of the hypotheses of Theorem \ref{wh}. Thus, if $\pi: M \to V$ is a K\"ahler crepant resolution, then $M$ is quasi-projective by Theorem \ref{wh}(5), and every class $\mathfrak{k} \in H^2(M,\R)$ contains a closed $(1,1)$-form $\beta$ by Theorem \ref{wh}(4). Thanks to these two properties, \cite[Thm 1.1]{CollinsTosatti} now tells us that if $\mathfrak{k}$ pairs positively with all irreducible subvarieties of ${\rm Exc}(\pi)$ in the sense of Theorem \ref{existence}(3), then there exist an open neighborhood $U$ of ${\rm Exc}(\pi)$ in $M$  and a smooth function $\varphi: U \to \R$ such that $\beta|_U + i\partial\overline{\partial}\varphi > 0$. Again by Theorem \ref{wh}(5), there exists a smooth K\"ahler form $\omega_i$ on the projective compactification $X_i = M \cup D_i$ such that $\omega_i|_M \in \mathfrak{k}$.

We now return to the diffeomorphisms $\Phi_i$ constructed in Step 3. These are actually exponential-type maps \cite[Defn 4.5]{Conlon3} from some tubular neighborhood of the zero section in $N_{D_i/X_i}$ to a tubular neighborhood of $D_i$ in $X_i$ itself. Since $\omega_i$ is smooth on $X_i$, a scaling argument (compare the end of the proof of \cite[Prop 2.5]{Conlon3}) then shows that $|\nabla^k_{g_{0,i}}\Phi_i^*\omega_i|_{g_{0,i}}= O_{k,i}(r_i^{-2-k}) $ as $r_i \to \infty$ for all $k \in \N_0$. Note the following subtlety:~here, unlike in Step 3, we are discussing sections of $T^*C \otimes T^*C$ rather than $T^*C \otimes TC$, so the same quadratic decay rate holds with respect to \emph{any} $J_0$-K\"ahler cone metric on $C$ whose Reeb vector field is a positive scalar multiple of $\xi_i$. However, we can now conclude by the same approximation argument as in Step 3 that given any arbitrary $K \in \N$ and $\epsilon>0$, it holds for all $i \geq i(K,\epsilon)$ that $|\nabla^k_{g_0}\Phi_i^*\omega_i|_{g_{0}} = O_{K,\epsilon,i}(r^{-2+\epsilon-k})$ as $r \to \infty$ for all $k \in \{0,\ldots,K\}$.

\subsection{Step 5}\label{ss:five}

Let $\pi: M \to V$ be a K\"ahler crepant resolution. Let $\mathfrak{k} \in H^2(M,\R)$ be a class that pairs positively with all irreducible subvarieties of ${\rm Exc}(\pi)$. Given any arbitrary $K \in \N$ and $\epsilon>0$, we now fix a large enough index $i\geq i(K,\epsilon)$ in Steps 3--4. This provides us with a K\"ahler form $\omega_{K,\epsilon} \in \mathfrak{k}$ and a diffeomorphism $\Phi_{K,\epsilon}$ of neighborhoods of infinity in $C$ and in $M$ such that $$|\nabla^k_{g_0}(\Phi_{K,\epsilon}^*J - J_0)|_{g_0} = O_{K,\epsilon}(r^{\lambda+\epsilon-k}), \quad |\nabla^k_{g_0}(\Phi_{K,\epsilon}^*\omega_{K,\epsilon})|_{g_0} = O_{K,\epsilon}(r^{-2+\epsilon-k})$$ as $r \to \infty$ for all $k \in \{0,\ldots,K\}$. We are now finally in a position to invoke the existence theorem for AC Calabi-Yau metrics of \cite[Thm 2.4]{Conlon}. This was proved by solving a complex Monge-Amp\`ere equation on $M$ with respect to an AC reference metric constructed using $\omega_{K,\epsilon}$ in the class $\mathfrak{k}$.

A simplification compared to \cite{Conlon} is that, in the notation of \cite{Conlon}, we can assume $\omega = \xi$ on $M$ here instead of having to consider the case that $\omega - \xi = d\eta$ on $M \setminus K$. This eliminates the complicated $i\partial\overline{\partial}$-lemma of \cite[Cor A.3(ii)]{Conlon}. Thus, while the simpler $i\partial\overline{\partial}$-lemma of \cite[Cor A.3(i)]{Conlon} is still needed, this part of the existence theory now goes through in all dimensions, with no need for any special arguments in dimension $2$; see \cite[Rmk 2.5]{Conlon}. We thank F.~Rochon for these observations.

A small difficulty is that the decay conditions in \cite{Conlon} and in Definition \ref{ACCY} were phrased in terms of infinitely many derivatives. Here we only have a finite but arbitrarily large number $K$ of derivatives for a fixed AC diffeomorphism $\Phi_{K,\epsilon}$. This suffices to prove the existence of AC Calabi-Yau metrics of arbitrarily high but finite order with no changes to the proof. Assuming this, we then obtain the AC property to all orders by constructing a Bianchi gauge in weighted H\"older spaces on $C$.

\section{Theorem \ref{classification}}\label{s:classification}

Let $(M,g,J,\Omega)$ be a Calabi-Yau manifold of complex dimension $n \geq 2$ which is AC with respect to some diffeomorphism $\Phi$ identifying the asymptotic cone $C$ with $M$ at infinity. Assume that the decay rate of $J$ is $\lambda<0$. Write $g_{0},J_{0},\Omega_{0}$ for the Calabi-Yau structure and $\xi=J_0(r\partial_r)$ for the Reeb vector field of $C$. Let $\T$ denote the associated Reeb torus, i.e., the real torus acting holomorphically and isometrically on $(C,g_{0},J_0)$ generated by the flow of $\xi$. Then, by construction, we clearly have that $\xi\in\operatorname{Lie}(\T)$. Let $\xi_{i}$ be a sequence of vector fields in $\operatorname{Lie}(\T)$ such that some scalar multiple of $\xi_i$ generates a $\C^*$-action on $C$ and such that $\xi_i \to \xi$ as $i\to\infty$. For technical reasons, we assume that $\xi_i$ is chosen using Dirichlet's theorem on diophantine approximation for vectors. Thus, for some fixed choice of a norm on $\operatorname{Lie}(\T)$ we have that
\begin{equation}\label{eq:dirichlet}
|\xi_i - \xi| = O(c_i^{-1-\frac{1}{d}})\;\,\text{as}\;\,i\to\infty,
\end{equation}
where $d = \dim \T$ and where $c_{i}$ is the unique positive real number such that $c_{i}\xi_{i}$ generates an effective $\C^{*}$-action on $C$. (This step is not essential but will save us a fair amount of technical work later; see the proof of Claim 1.) By the Reeb field perturbation results of Appendix \ref{s:type1}, for all $i$ there exists a quasi-regular $J_0$-K\"ahler cone metric $g_{0,i}$ on $C$ whose Reeb vector field is exactly $\xi_{i}$. Furthermore, if $r_i$ denotes the radius function of $g_{0,i}$, then, by Theorem \ref{p:typeIest}, for all $K \in \N$ and $\epsilon > 0$ there exists an $i(K,\epsilon)\in\N$ such that for all $i \geq i(K,\epsilon)$ we have that
\begin{equation}\label{eq:cxasympt}
|\nabla^{k}_{g_{0,i}}(\Phi^{*}J-J_{0})|_{g_{0,i}}=O_{K,\epsilon}(r_{i}^{\lambda+\epsilon-k})\;\,\text{as $r_i \to \infty$}
\end{equation}
for all $k=0,\ldots,K$. Precisely, it suffices to chose $i(K,\epsilon)$ so that $|\xi_i - \xi|$ is less than $\epsilon$ times a small constant depending only on $K$ for $i \geq i(K,\epsilon)$. Here, whereas the metric $g_{0,i}$ depends on $i$, the map $\Phi$ and the constants implicit in the $O_{K,\epsilon}$ notation do not, unlike in the proof of Theorem \ref{existence}.

Let $D_{i}$ be the orbifold quotient of $C$ by the effective $\C^*$-action generated by $c_i\xi_i$. As in the proof of Theorem \ref{theo:properties}, $D_{i}$ is Fano for $i \gg 0$ because it admits a K\"ahler metric of positive Ricci curvature. Following the same proof, since $D_i$ is Fano and $K_C$ is trivial, we can show using Theorem \ref{th:gysin} that $C$ with the $\C^*$-action generated by $c_i\xi_i$ is equivariantly isomorphic to \begin{small}$\frac{1}{\alpha_{i}-1}$\end{small}$K_{D_{i}}$ with its zero section blown down for some integer $\alpha_i > 1$. Thus, $r_i^2 = h_i^{c_i}$ for some Hermitian metric $h_i$ on this bundle.

Thanks to \eqref{eq:cxasympt}, using \cite[Thm 1.6]{ChiLi} we can holomorphically compactify $M$ to obtain a compact complex orbifold $X_{i}=M\cup D_{i}$, where $D_i$ is an admissible divisor in $X_i$ whose normal orbibundle is isomorphic to \begin{small}$-\frac{1}{\alpha_{i}-1}$\end{small}$K_{D_{i}}$. Technically, we need to fix $\epsilon < |\lambda|$, $K \geq 2n+1$, $i \geq i(K,\epsilon)$, and use the slightly more specific result of \cite[Prop 6.1]{ChiLi}, which says that $2n+1$ derivatives in \eqref{eq:cxasympt} are enough to compactify. (The extension of these results from manifolds to orbifolds is immediate.)

\begin{prop}\label{properties}
For all $i \gg 0$, the orbifold $X_i = M \cup D_i$ satisfies the following properties.

{\rm (1)} The holomorphic line orbibundle $K_{X_i}+\alpha_i[D_i]$ on $X_i$ is trivial.

{\rm (2)} There exists a holomorphic map $\pi: M \to V$ onto a normal affine variety $V$, both independent of $i$, and a holomorphic extension $\pi_i:X_i \to Y_i$ of $\pi$ onto a normal projective variety $Y_i$ such that

\begin{itemize}
\item[$\bullet$] $\pi_i$ is an isomorphism onto its image in a neighborhood of $D_i$,
\item[$\bullet$] the $\Q$-Cartier divisor $[\pi_i(D_i)]$ is ample on $Y_i = V \cup \pi_i(D_i)$,
\item[$\bullet$] all of the singularities of $V$ are isolated and canonical,
\item[$\bullet$] $\pi$ is a crepant resolution of $V$, and
\item[$\bullet$] the $\Q$-Cartier divisor $K_{Y_i} + \alpha_i[\pi_i(D_i)]$ on $Y_i$ is trivial.
\end{itemize}

{\rm (3)} $X_i$ is projective.
\end{prop}

\begin{proof}
(1) The point is to prove that the given holomorphic volume form $\Omega$ on $M = X_i \setminus D_i$ extends to a meromorphic volume form on $X_i$ with a pole of order $\alpha_i$ along $D_i$.

Consider the model form $\Omega_0$ on the cone $C$, which we identify with the complement of the zero section in $N_{D_i/X_i}$. Note that $\Omega_0$ is bounded with respect to $g_0$, hence grows at worst polynomially with respect to $g_{0,i}$ by \eqref{e:saviors}. Using the analog of \eqref{eq:cxasympt} for $\Omega$, we  then see that the same is true for $\Omega$. Recall that in \cite{ChiLi} the compactification $X_i = M \cup D_i$ was constructed by using \eqref{eq:cxasympt} to find local $J$-holomorphic coordinates asymptotic to local $J_0$-holomorphic coordinates centered at any point of $D_i$ in the compactified cone $C \cup D_i$. Thus, $\Omega$ blows up at worst polynomially near $D_i$ in these local coordinates. By Riemann's removable singularities theorem, $\Omega$ extends meromorphically.

Since $\Omega$ is meromorphic, using local coordinates as above, we can extract a leading term from $\Omega$, which is a meromorphic volume form, $\Omega_i$, on a tubular neighborhood of the zero section in $N_{D_i/X_i}$. Recall that $N_{D_i/X_i}$ is isomorphic to \begin{small}$-\frac{1}{\alpha_i-1}$\end{small}$K_{D_i}$, whose total space carries a tautological volume form $\Omega_{0,i}$ with a pole of order $\alpha_i$ along the zero section. Then $\Omega_i = f \Omega_{0,i}$ for some meromorphic function $f$ without zeros and with poles at worst along $D_i$. Now $\dim C = n > 1$, so $H^1(L,\R) = 0$, and hence away from the zero section we have $f = e^g$ for some holomorphic function $g$. Then $g = O((\log r_i)^2)$ because ${\rm Re}\,g$ $=$ $\log |f|$ $=$ $O(\log r_i)$ and $d({\rm Im}\,g) = -d({\rm Re}\,g)\circ J_0 = O((\log r_i)/r_i)$ by standard scaled elliptic estimates. The $3$-annulus lemmas of \cite[Section 3.2]{DS2} then imply that $g$ is constant because otherwise $g$ would grow at least polynomially (there are no decaying modes because $n > 1$, see \cite[Lemma 2.13(2)]{HS}). Thus, $\Omega_i$ is a constant multiple of $\Omega_{0,i}$, which finishes the proof.

As an aside, note that $\Omega_0 = const \cdot \Omega_{0,i}$ by the same reasoning, and $|\Phi^*\Omega-\Omega_0|_{g_{0,i}} = O(r_i^{\lambda+\epsilon})$ by assumption, but this does not seem to imply the desired behavior of $\Omega$ directly.

(2) Here we can essentially follow the proof of \cite[Lemma 2.1]{Epstein}, as we already did in our previous paper \cite[Lemma 2.3]{Conlon3}. Since $N_{D_i/X_i}$ is positive, we can easily show that $M = X_i \setminus D_i$ is $1$-convex, so that we may take its Remmert reduction $\pi: M \to V$. See e.g.~\cite[App A]{Conlon} for details and references on Remmert reductions, in particular for the fact that $V$ is Stein (although not necessarily an affine variety) and that the singularities of $V$ are normal and isolated.  Moreover, $\pi$ is an isomorphism onto its image in a neighborhood of infinity. Thus, we may compactify $V$ as a normal compact complex space $Y_i$ by adding the orbifold $D_i$ as an admissible divisor. Then we have an obvious holomorphic extension $\pi_i:X_i\to Y_i$ of $\pi$. In the following, we will write $D_i$ instead of $\pi_i(D_i)$ for simplicity.

Let $m \in \N$ be sufficiently large and divisible such that $m D_i$ is a Cartier divisor on $Y_i$. Then, by \cite[p.347, Satz 4]{Grau:62}, the associated line bundle $[mD_i]$ is positive in the sense of \cite[p.342, Defn 2]{Grau:62}. Thus, by the proof of \cite[p.343, Satz 2]{Grau:62}, $Y_i$ admits an embedding into some $\P^{N}$ given by the global sections of $m'[mD]$ for some $m' \in \N$. Thus, $Y_i$ is projective, $D_i$ is ample, and $V$ is affine.

The remaining properties stated in item (2) are now clear from item (1) and from the definitions of a canonical singularity and a crepant resolution, except for the following two subtleties. First, we need to check that $V$ is Gorenstein, i.e., that the Weil divisor $K_V$ is Cartier. Second, the property of being canonical should be tested using an algebraic rather than a complex-analytic resolution of singularities. These two points can be addressed by quoting a local version of Hironaka's flattening theorem \cite[p.321, Thm 1]{HiRossi} as in the proof of Theorem \ref{wh}(1). This gives us a projective variety $\hat{X}_i$ and a morphism $f: \hat{X}_i \to Y_i$, the blowup of $Y_i$ in an ideal sheaf with support in $V^{\rm sing} \subset V \subset Y_i$, such that $f = \pi_i \circ \sigma$ for some holomorphic map $\sigma: \hat{X}_i \to X_i$. Then $\sigma^*\Omega$ is a rational $n$-form on $\hat{X}$ by GAGA, regular on $f^{-1}(V)$ and nowhere vanishing on $f^{-1}(V^{\rm reg})$. Thus, $(f^{-1})^*\Omega$ is an algebraic trivializing section of $K_{V^{\rm reg}}$, and the discrepancies of $f$ over $V^{\rm sing}$ are nonnegative because $f$ factors through a crepant resolution of $V$. This implies the two properties that we needed to show.

(3) This is a direct application of item (2) together with Theorem \ref{wh}(5).
\end{proof}

\noindent \emph{Proof of Theorem \ref{classification}}. Given the original AC Calabi-Yau manifold $(M,g,J,\Omega)$, we have proved that $M$ is a crepant resolution of an affine variety $V$ (with at worst isolated singularities, all of which are canonical, and with trivial canonical bundle). The class $\mathfrak{k} \in H^2(M,\R)$ represented by the Ricci-flat K\"ahler form obviously pairs positively with all components of the exceptional set. We will now show that $V$ is a deformation of negative $\xi$-weight of $C$ by verifying Definition \ref{negativeweight}.

To construct $(W_i,p_i,\sigma_i)$, we consider the projective varieties $Y_i = V \cup D_i$ of Proposition \ref{properties}(2), where abusing notation we write $D_i$ instead of $\pi_i(D_i)$. Recall that $C$ together with the $\C^*$-action generated by $c_i\xi_i$ is equivariantly isomorphic to the conormal bundle to $D_i$ in $Y_i$.  Upon removing the $\C^*$-invariant divisor of Theorem \ref{thm:deform-tech}(1), the test configuration of Theorem \ref{thm:deform-tech} (i.e., the deformation of $Y_i$ to the normal cone of $D_i$) yields a $\C^*$-equivariant degeneration $p_i: W_i \to \C$ with general fiber $p_i^{-1}(1) = V$, and with special fiber $p_i^{-1}(0)$ an affine variety equivariantly normalized by our cone $C$. By Theorem \ref{thm:deform-tech}(2), $p_i^{-1}(0)$ will be isomorphic to $C$, as required, if the restriction maps
$$H^0(Y_i, \mathcal{O}_{Y_i}(mD_i)) \to H^0(D_i, \mathcal{O}_{D_i}(mD_i))$$ are surjective for every $m \in \N$. Taking cohomology in the restriction sequence
$$0 \to \mathcal{O}_{Y_i}((m-1)D_i) \to \mathcal{O}_{Y_i}(mD_i) \to \mathcal{O}_{D_i}(mD_i) \to 0,$$
whose exactness can be checked on stalks by averaging over the local uniformizing groups (a more abstract argument is given in \cite[Lemma 2.9]{HRS}), we find that these restriction maps will be surjective if $H^1(Y_i, \mathcal{O}_{Y_i}((m-1)D_i)) = 0$ for all $m \in \N$. To prove this vanishing, we will imitate the proof of Theorem \ref{wh}(3), replacing the projective orbifold $\hat{X}$ there by our $X_i$ because here we already know that $X_i$ is projective from Proposition \ref{properties}(3) (an application of Theorem \ref{wh}(5)). In short, using the generalized Kodaira vanishing theorem from Claim 1 of that proof and the fact that $-K_{X_i} = \alpha_i[D_i]$ with $\alpha_i > 0$, we obtain that $R^q (\pi_i)_*(\mathcal{O}_{X_i}((m-1)D_i)) = 0$ for all $q > 0$ as in Claim 2 of that proof. Then the Leray spectral sequence tells us that for all $q> 0$,
$$H^q(Y_i, \mathcal{O}_{Y_i}((m-1)D_i)) = H^q(X_i, \mathcal{O}_{X_i}((m-1)D_i)),$$
and the latter vanishes again by Claim 1 and because $-K_{X_i} = \alpha_i[D_i]$ with $\alpha_i > 0$.

We have now constructed a sequence $(W_i, p_i, \sigma_i)$ of $\C^*$-equivariant degenerations of $V$ to $C$ such that the induced $\C^*$-action on the central fiber is the one generated by the flow of $\xi_i$. Definition \ref{negativeweight} also requires us to check that the sequence $-k_i\mu_i/c_i$ is eventually uniformly bounded away from $0$, where $k_i$ is the vanishing order of the $i$-th deformation, $\mu_i$ is the weight of the induced $\C^*$-action on the base (for us, $\mu_i = 1$ by construction), and $c_i$ is the unique positive real scaling factor such that $c_i\xi_i$ generates an effective $\C^*$-action. As mentioned before Proposition \ref{properties}, $c_i$ can be characterized by the property that $r_i^2 = h_i^{c_i}$, where $h_i$ is a Hermitian metric on the conormal bundle to $D_i$ in the projective compactification $Y_i$. Also recall that thanks to Theorem \ref{chili}, $k_i$ can be characterized as the biggest integer such that $D_i$ is $(k_i-1)$-comfortably embedded in $Y_i$. Thus, by the arguments of \cite[pp.1461--1462]{ChiLi}, the complex structure rate of \eqref{eq:cxasympt} provides a lower bound
$$k_i \geq \lceil c_i(|\lambda|-\epsilon)\rceil,$$
valid for all $\epsilon < |\lambda|$, $K \geq 2n+1$, and $i \geq i(K,\epsilon)$. Since $\mu_i=1$, this is exactly the required property that $-k_i\mu_i/c_i$ is eventually uniformly bounded away from $0$. (Again, Theorem \ref{chili} is valid only for $n \geq 3$, but for $n = 2$ there is no need to pick an approximating sequence to begin with.)

We have now exhibited $V$ as a deformation of negative $\xi$-weight of $C$ using a particular rational approximating sequence $\xi_i \to \xi$ and an associated sequence of degenerations $(W_i,p_i,\sigma_i)$ constructed using Proposition \ref{properties} and Theorem \ref{thm:deform-tech}. We will prove Theorem \ref{classification} by showing that if we apply the proof of Theorem \ref{existence} \emph{to this particular sequence of degenerations $(W_i,p_i,\sigma_i)$}, then the resulting AC Calabi-Yau metric in the class $\mathfrak{k}$ on $M$ is equal to our original AC Calabi-Yau metric $g$ for $i \gg 0$.

Fix $\tilde{\varepsilon}>0$ and $\tilde{K} \geq 1$ as parameters for the construction of Theorem \ref{existence}. Then, for all $i\in\N$ that are sufficiently large depending on $\epsilon,K$ and $\tilde\epsilon,\tilde{K}$, we have a diagram
$$
\begin{tikzcd}[column sep=small]
& \mathcal{Y}_i \setminus \mathcal{D}_i \arrow[rrrr, "H_i", "\cong"'] & & & & \overline{\mathcal{V}}_i \setminus \mathcal{D}_i\\
C  \arrow[ur, hook] \arrow[rr, dashed, "\Phi_{{Y}_i}"] \arrow[rrrr, bend right, "H_{0,i}", "\cong"'] & & V \arrow[ul, hook'] \arrow[rrrr, bend right, "H_{1,i}", "\cong"'] & &  \arrow[ur, hook] p_i^{-1}(0)\setminus D_i  \arrow[rr, dashed, "\Phi_{\overline{{V}}_i}"]  &  &  p_i^{-1}(1) \setminus D_i\arrow[ul, hook']
\end{tikzcd}
$$
\vspace{1mm}
\noindent where the solid arrows commute and the spaces and maps are defined as follows:
\begin{itemize}
\item[$\bullet$] $\mathcal{Y}_i \to \C$ denotes the deformation to the normal cone of the pair $D_i \subset Y_i$ of Proposition \ref{properties}; $\mathcal{D}_i \subset \mathcal{Y}_i$ is the compactifying divisor, which together with the induced fibration $\mathcal{D}_i \to \C$ is naturally isomorphic to $(D_i \times \C, {\rm pr}_{\C})$; and $C,V$ are fibers of $\mathcal{Y}_i \setminus \mathcal{D}_i \to \C$ by construction.
\item[$\bullet$] $p_i: \overline{\mathcal{V}}_i \to \C$ is the analogous deformation for the pair $D_i \subset \overline{V}_i$ from the proof of Theorem \ref{existence}.
\item[$\bullet$] $H_i$ is any equivariant isomorphism of test configurations as in Theorem \ref{normalform}(3).
\item[$\bullet$] $\Phi_{\overline{{V}}_i}$ is the diffeomorphism from \cite[\S 3.1]{ChiLi} used in Step 3 of the proof of Theorem \ref{existence}.
\item[$\bullet$] $\Phi_{Y_i}$ is the diffeomorphism used to construct the compactification $Y_i = V \cup D_i$ in the first place. More precisely, $\Phi_{Y_i} = \Phi \circ \mathcal{N}_i$, where $\Phi$ is as in \eqref{eq:cxasympt} and $\mathcal{N}_i$ is the deformation of ${\rm Id}_C$ given by the Newlander-Nirenberg type construction of \cite[Prop 6.1]{ChiLi}, so that $\Phi_{Y_i}^*J$ extends smoothly to the natural compactification of $C$ by $D_i$. 
\end{itemize}

\noindent Moreover, the proof of Theorem \ref{existence} produces an AC Calabi-Yau metric $\omega_i \in \mathfrak{k}$ such that
\begin{align}
\sum_{k=0}^{\tilde{K}} r^k |\nabla_{g_0}^k([H_{1,i}^{-1} \circ \Phi_{\overline{V}_i} \circ H_{0,i}]^*g_i - g_0)|_{g_0} = O_{\tilde{K},\tilde{\epsilon},i}(r^{\max\{\lambda,-2\}+2\tilde{\epsilon}})\;\,\text{as}\;\,r\to\infty,
\label{eq:stillnotdone}
\end{align}
where $g_0$ is the given Calabi-Yau cone metric on $C$ with radius function $r$. To see this, first observe that the construction of Theorem \ref{existence} takes place on the right-hand side of the diagram, so we need to consider the base change of $\pi: (M,\mathfrak{k}) \to V$ via $H_{1,i}^{-1}$ and fix a Calabi-Yau cone metric on $p_i^{-1}(0) \setminus D_i$ as input. For this we use the pushforward of $\omega_0$ under $H_{0,i}$, which is an equivariant isomorphism of cones. Then we apply the proof of Theorem \ref{existence} and perform a base change under $H_{1,i}$.

We can now complete the proof of Theorem \ref{classification} in three steps.\medskip\

\noindent \emph{Claim 1}. For all $i \gg 0$ we define a diffeomorphism of neighborhoods of infinity in $C$ by
$$\Psi_i:=H_{0,i}^{-1} \circ \Phi_{\overline{V}_i}^{-1} \circ H_{1,i}\circ \Phi_{Y_i}.$$
Theorem \ref{classification} will follow if we can prove that for all $i \gg 0$ it holds for $k \in \{0,1\}$ that
\begin{equation}\label{eq:thedream}
r_i^k|\nabla^{k}_{g_{0,i}}\left(\Psi_i^*g_{0,i} - g_{0,i}\right)|_{g_{0,i}}=O_{i}(r_{i}^{-1/c_i})\;\,\text{as}\;\,r_i \to \infty.
\end{equation}
Here we recall again that $r_i$ is the radius function of the K\"ahler cone metric $g_{0,i}$ with Reeb vector field $\xi_i$ chosen at the beginning of this section, and the numbers $c_i > 0$, $c_i \to \infty$, are characterized by the property that $r_i^2 = h_i^{c_i}$ for some Hermitian metric $h_i$ on the conormal bundle $N_{D_i/Y_i}^{-1}$.\medskip\

\noindent \emph{Proof of Claim 1}. Suppose \eqref{eq:thedream} holds for all $i \gg 0$ and for $k \in \{0,1\}$. We would first like to replace $r_i$ by $r$ and $g_{0,i}$ by $g_0$ in this statement, increasing $i$ if necessary. Theorem \ref{p:typeIest} allows us to do so at the cost of increasing the exponent on the right-hand side by $O(|\xi_i - \xi|)$, where we are fixing some norm on ${\rm Lie}(\T)$. Thanks to the Dirichlet approximation estimate, \eqref{eq:dirichlet}, we obtain that
\begin{equation}\label{eq:nice}
r^{k}|\nabla_{g_{0}}^{k}(\Psi_i^*g_0 - g_0) |_{g_0} = O_{i}(r^{\nu_i})\;\,{\rm as}\;\,r\to\infty
\end{equation}
for all $i \gg 0$ and $k \in \{0,1\}$, where the exponents $\nu_i$ go to zero as $i \to \infty$ but are strictly negative. (With a good amount of extra work it is possible to achieve this without using Dirichlet's theorem. The key step is to improve the exponent $-1/c_i$ in \eqref{eq:thedream} to $-k_i/c_i$, which is uniformly negative, by using the $(k_i-1)$-comfortable property of the embeddings $D_i \subset Y_i$ and $D_i \subset \overline{V}_i$.)

Given \eqref{eq:stillnotdone} and \eqref{eq:nice}, it follows from the triangle inequality and a pullback trick that
$$r^k|\nabla_{g_0}^k(\Phi_{Y_i}^*g_i - g_0)|_{g_0} = O_i(r^{\nu_i})\;\,{\rm as}\;\,r\to\infty$$
for all $i \gg 0$ and $k\in\{0,1\}$. On the other hand, the same statement holds for $g$ instead of $g_i$. This follows from $\Phi^*g$ being asymptotic to $g_0$ in the space $C^\infty_\lambda(g_0)$, together with the fact that $\mathcal{N}_i$ in the definition of $\Phi_{Y_i}$ is a $C^\infty$ diffeomorphism of the end of $C$ converging to ${\rm Id}_C$ in $C^{K}_{\lambda+\epsilon}(g_{0,i})$. (As stated, \cite[Prop 6.1]{ChiLi} only implies that $\mathcal{N}_i$ is a $C^{2n+1}$ diffeomorphism converging to ${\rm Id}_C$ in $C^{2n+1}_{\lambda+\epsilon}(g_{0,i})$, but local $C^\infty$ and global $C^{K}_{\lambda+\epsilon}(g_{0,i})$ regularity of $\mathcal{N}_i$ follow from the proof of this proposition.)

We can now apply \cite[Thm 3.1]{Conlon} to conclude that $g_{i} = g$ for all $i \gg 0$. Note that the proof of this theorem only requires the two metrics to be asymptotic in $C^1_{-\delta}$ for some $\delta>0$.\hfill $\Box$\medskip\

Thus, it remains to prove the estimate \eqref{eq:thedream}. This is precisely the expected behavior if $\Psi_i$ extends to an exponential-type map \cite[Defn 4.5]{Conlon3} on a neighborhood of $D_i$ in $\overline{C}_i := C \cup D_i$. By definition this means that $\Psi_i(p) = p$ for all $p \in D_i$, $d\Psi_i|_p$ is complex linear for all $p \in D_i$, and, after identifying $\overline{C}_i = N_{D_i/\overline{C}_i} = T^{1,0}\overline{C}_i|_{D_i}/T^{1,0}D_i$, it holds for all $p \in D_i$ and $v \in N_{D_i/\overline{C}_i,p} \subset T_p^{1,0}N_{D_i/\overline{C}_i}$ that
\begin{align}\label{dergral}
d\Psi_i|_p(v) + T^{1,0}D_i = v.
\end{align}
The next two claims make this picture rigorous, thereby completing the proof of Theorem \ref{classification}.\medskip\

\noindent \emph{Claim 2}. Fix $i \gg 0$ and identify $\overline{C}_i = C\cup D_i$ with the total space of the normal bundle to $D_i$ in $\overline{C}_i$. Then $\Psi_i$ extends to an exponential-type map of class $C^6$ on a neighborhood of $D_i$.\medskip\

\noindent \emph{Proof of Claim 2}. First, there exists a natural $\C^*$-equivariant and fiber-preserving biholomorphism $\overline{H}_i: \mathcal{Y}_i \to \overline{\mathcal{V}}_i$ extending $H_i$. The point here is that the compactifying divisors $D_i$ in the fibers of $\mathcal{Y}_i$ are canonically identified with each other by the deformation to the normal cone construction; this also holds for the compactifying divisors $D_i$ in the fibers of $\overline{\mathcal{V}}_i$; and the map $\overline{H}_i$ on these divisors is simply the $\C^*$-equivariant map of cones $H_{0,i}$ after quotienting by $\C^*$ on each cone.

Second, $\Phi_{\overline{V}_i}$ extends to a smooth exponential-type map by construction.

Third, $\Phi_{Y_i} = \Phi \circ \mathcal{N}_i$ also extends to an exponential-type map by construction. (The $\Phi$ factor may seem confusing but recall that the compactification $Y_i = V \cup D_i$ was \emph{defined} by pulling the complex structure on $V$ back by $\Phi$ and comparing it to the model complex structure on $\overline{C}_i = C \cup D_i$, which then defines the normal bundle to $D_i$ in $Y_i$.) The statement and proof of \cite[Prop 6.1]{ChiLi} only ensure that this extension, while $C^\infty$ away from $D_i$, is $C^{\min\{K,\lceil c_i(|\lambda|-\epsilon)\rceil-1\}}$ at the points of $D_i$. However, by choosing $K,i$ large enough we can certainly arrange that its regularity is at least $C^6$.

Given these facts, it is now clear that $\Psi_i$ extends to a $C^6$ map on a tubular neighborhood of $D_i$ in $\overline{C}_i$ for all $i \gg 0$. Moreover, this extension obviously fixes every point of $D_i$, and its differential at all points of $D_i$ is complex linear because $\overline{H}_i$ is holomorphic and $\Phi_{\overline{V}_i}, \Phi_{Y_i}$ extend to exponential-type maps. It remains to verify \eqref{dergral}. To this end, it suffices to show under the identification of infinity divisors described above that $d\overline{H}_{0,i}|_p(v) = d\overline{H}_{1,i}|_p(v)$, where $\overline{H}_{0,i}:=\overline{H}_i|_{\overline{C}_i}$ and $\overline{H}_{1,i}:=\overline{H}_i|_{\overline{V}_i}$.

This property follows from the equivariance of $\overline{H}_i$. More precisely, let $\varphi_{t}$ denote the time $t$ flow of the vector field on $\mathcal{Y}_i$ and $\overline{\mathcal{V}}_i$ induced by the Euler vector field on $\C^{*}$ via the respective $\C^{*}$-actions. Then the fact that $\varphi_{t}$ preserves $\mathcal{D}_i$ implies that we have an induced map
$$d\varphi_{t}|_{q}:N_{D_i/{\rm fiber},q}\to N_{D_i/{\rm fiber},\varphi_{t}(q)}$$
for all points $q$ on the compactifying divisor $\mathcal{D}_i \cong D_i \times \C$ in either $\mathcal{Y}_i$ or $\overline{\mathcal{V}}_i$. These normal bundles are also naturally identified with each other via the deformation to the normal cone construction, and in this sense
$d\varphi_{t}|_{q}=e^{-t}{\rm Id}$. Hence the
$\C^{*}$-equivariance of $\overline{H}_i$ implies that
\begin{equation*}
d\overline{H}_{1,i}|_p = d\varphi_{-t}|_{\overline{H}_i(\varphi_{t}(p))} \circ d\overline{H}_i|_{\varphi_{t}(p)} \circ d\varphi_{t}|_{p}=d\overline{H}_i|_{\varphi_{t}(p)}.
\end{equation*}
Letting $t\to-\infty$ now gives us what we need.\hfill $\Box$\medskip\

\noindent \emph{Claim 3}. For any given $i \gg 0$, let $W$ be an open neighborhood of $D_{i}$ in $\overline{C}_i$ and let $\Upsilon:W\to\overline{C}_i$ be an exponential-type map of class $C^{6}$.
Then it holds for $k \in \{0,1\}$ that 
\begin{equation*}
r_i^k|\nabla^{k}_{g_{0,i}}(\Upsilon^*g_{0,i} - g_{0,i})|_{g_{0,i}}=O_{\Upsilon,i}(r_{i}^{-1/c_i})\;\,\text{as}\;\,r_i\to\infty.
\end{equation*}
\vskip3mm

\noindent \emph{Proof of Claim 3}. Very similar computations can be found in \cite[\S 2.2]{Conlon3}, so we will be brief here. We identify $\overline{C}_i$ with the total space of the normal bundle to $D_i$ in $\overline{C}_i$. Then, in a neighborhood of an arbitrary point on $D_i$, we let $(w_{1},\ldots,w_{n})$ be holomorphic coordinates on a uniformizing chart with $w_{n}$ the coordinate along the fibers, so that $D_{i}$ is locally given by $\{w_{n}=0\}$. Then $r_{i}=e^{f}|w_{n}|^{-c_{i}}$, where $f=f(w_{1},\ldots,w_{n-1})$ is a smooth real-valued function. We find from \cite[Lemma B.3]{Conlon3} that
\begin{equation}\label{willyalookatthat}
\Upsilon^{*}w_{n}=w_{n}+A_{n,1}^{[4]}w_{n}^{2}+A_{n,2}^{[4]}w_{n}\overline{w}_{n}+A_{n,3}^{[4]}\overline{w}_{n}^{2},
\end{equation}
where a symbol $A^{[\ell]}$ denotes a function of class $C^{\ell}$. Similarly, for all $j < n$,
\begin{equation}\label{easy1}
\Upsilon^{*}(w_{j}-A_j^{[5]}(w_1,\ldots,w_{n-1})w_n)=w_{j}+ A_{j,1}^{[3]}w_{n}^2+A_{j,2}^{[3]}w_n\overline{w}_{n} + A_{j,3}^{[3]}\overline{w}_n^2,
\end{equation}
and any $C^{6}$ function $h$ of the variables $(w_{1},\ldots,w_{n-1})$ satisfies
\begin{equation}\label{easy2}
\Upsilon^{*}h=h+A_{h,1}^{[5]}w_{n}+A_{h,2}^{[5]}\overline{w}_{n}.
\end{equation}
In addition, a scaling argument shows that
\begin{eqnarray}
\label{reg_est1}
w_j = O(1)\;\,\textrm{with infinitely many $g_{0,i}$-derivatives for all $j < n$},\\
w_{n}^{\pm 1}= O(r_{i}^{\mp (1/c_i)})\;\,\textrm{with infinitely many $g_{0,i}$-derivatives},\label{reg_est2}
\end{eqnarray}
and hence that every function $A^{[\ell]}$ is $O(1)$ with $\ell$ many $g_{0,i}$-derivatives.

From \eqref{willyalookatthat}, \eqref{reg_est2} and \eqref{easy2} we easily find that
\begin{equation}\label{finito}
\Upsilon^{*}r_{i}=r_{i}+O(r_{i}^{1-(1/c_i)})\;\,\textrm{with four $g_{0,i}$-derivatives.}
\end{equation}
Using also \eqref{easy1}, the computations of \cite[\S2.2.4]{Conlon3} together with \cite[Lemma 2.14]{Conlon} imply that
\begin{equation}\label{finito2}
\Upsilon^{*}J_{0}-J_{0}=O(r_{i}^{-1/c_i})\;\,\textrm{with two $g_{0,i}$-derivatives}.
\end{equation}
Using the formula $\omega_{0,i} = -\frac{1}{4}d(dr_{i}^2 \circ J_0)$ and \eqref{finito}, \eqref{finito2}, we can now show that 
$$\Upsilon^*\omega_{0,i} - \omega_{0,i} = O(r_{i}^{-1/c_i})\;\,\textrm{with one $g_{0,i}$-derivative}.$$
Combining this with \eqref{finito2} we obtain Claim 3, and Theorem \ref{classification} is proved.\hfill $\Box$\medskip\

To summarize, every AC Calabi-Yau manifold $(M,g,J,\Omega)$ is diffeomorphism equivalent to one of the examples constructed in the proof of Theorem \ref{existence}. Given a Calabi-Yau cone $(C,g_0,J_0,\Omega_0)$, the input for Theorem \ref{existence} consists of a deformation $V$ of negative $\xi$-weight of $C$, a crepant resolution $M$ of $V$, and a K\"ahler class $\mathfrak{k}$ on $M$. But even if we fix all of these data, some freedom to carry out our construction remains, and in the proof of Theorem \ref{classification} we had to make choices to recover the given manifold $(M,g,J,\Omega)$. More precisely, the following degrees of freedom still exist:
\begin{itemize}
\item[(1)] Pull back $g_0$ by an automorphism of $(C,\xi)$ without changing the identification of $V$ and $C$ at infinity. The statement of Theorem \ref{existence} already reflects this.
\item[(2)] Pull back the AC Calabi-Yau metric constructed on $M$ by an automorphism of $(M,\mathfrak{k})$.
\item[(3)] Pick different sequences $\xi_i \to \xi$. Even for a fixed sequence, our uniqueness theorem may only apply if $i$ is bigger than the first index for which we have existence.
\item[(4)] There might exist different equivariant degenerations $p_i: W_i \to \C$ from $V \cong p_i^{-1}(1)$ to $C$.
\end{itemize}
\noindent One can ask whether these choices lead to non-isometric AC Calabi-Yau metrics in the class $\mathfrak{k}$. For (1) we already asked this in the Introduction. For (2) this is not the case by definition. It also seems quite unlikely that (3)--(4) can generate non-isometric metrics; for example, (3) is an issue only for irregular cones, and in Section \ref{s:uniqueness} we will prove that (4) is impossible in certain examples.

\section{Theorem \ref{uniqueness}, and the case of toric cones}\label{s:uniqueness}

\subsection{Background from deformation theory}\label{ss:def-thy}

In order to obtain concrete classification results for AC Calabi-Yau manifolds with a fixed tangent cone $C$ at infinity, we need to be able to determine explicitly the deformations of negative $\xi$-weight of the cone $C$ in question.

Consider an affine algebraic variety with a unique singular point. Schlessinger \cite{schlessinger} constructed a formal deformation of the singularity which is semi-universal in the category of formal deformations. By work of Grauert \cite{Grauert-def}, this deformation can be lifted to a complex-analytic deformation which is semi-universal in the category of complex-analytic deformations. Similarly, by work of Artin \cite{Artin} and Elkik \cite{Elkik2}, Schlessinger's deformation can be lifted to an affine algebraic deformation, although the correct semi-universality property then involves classifying maps which are not necessarily globally defined and algebraic but which only exist in an \'etale neighborhood of the singularity. See \cite[p.175, Example 4.5]{Artin2} for a concise statement of the final result in this direction.

Consider now a Kähler cone $C$ with Reeb vector field $\xi$. Then, by the above, as an affine variety with an isolated singularity, $C$ admits an Artin-Elkik semi-universal deformation $\mathcal{W} \to S$, which is a flat affine morphism over an affine base. To classify the deformations of negative $\xi$-weight of $C$ we require some additional properties of the family $\mathcal{W} \to S$. Let $\tilde\xi$ be a periodic Reeb vector field on $C$ obtained by a Type I deformation of $\xi$. Abusing notation, we identify $\tilde\xi$ with the unique effective algebraic $\C^*$-action on $C$ generated by a positive real multiple of $-J\tilde\xi$. Then we assume that:

\begin{itemize}
\item[(1)] The $\C^{*}$-action $\tilde\xi$ on $C \subset \mathcal{W}$ extends to an algebraic $\C^*$-action on $\mathcal{W}$ and the map $\mathcal{W} \to S$ is equivariant with respect to this action and some algebraic $\C^*$-action on $S$.
\item[(2)] All equivariant flat algebraic deformations of $(C,\tilde\xi)$ admit an equivariant complex-analytic classifying map to the family $\mathcal{W} \to S$ in some small neighborhood of the apex of $C$.
\end{itemize}

\noindent If $C$ is either toric with $\xi$ in the Lie algebra of the torus, or quasi-regular and a complete intersection of quasi-homogeneous hypersurfaces, then (1) holds thanks to work of Altmann \cite{altmann} and Slodowy \cite[p.9, Theorem]{Slodowy}. In these works, the Artin-Elkik deformation is also constructed directly for these two types of cones. Slodowy \cite[pp.12--13, Remarks 1)--2)]{Slodowy} sketches a proof of the fact that (1) implies (2) for arbitrary cones by showing how a classifying map given by formal power series can be made equivariant in such a way that convergent power series remain convergent.

It is unclear to us whether (1) is known for any other class of cones. In the formal category, (1) was established for all cones and all actions of linearly reductive groups \cite{Pinkham,Rim}, which is of course the expected generality (see \cite{doan} for a counterexample without reductivity). In the complex-analytic category, (1) was established for all cones and all $\C^*$-actions \cite[p.25, Thm 5(c)]{Hau}.

Properties (1)--(2) allow for the classification of deformations of negative $\xi$-weight of $C$. In the case of $2$-dimensional Calabi-Yau cones, i.e., Kleinian surface singularities, this argument was used by Kronheimer in his classification of gravitational instantons \cite{Kronheimer2}, relying on a suitable version of Theorem \ref{classification} that he had proved using twistor theory in his setting \cite[p.691]{Kronheimer2}. We now formalize this argument as a lemma, which is implicit in Kronheimer's work \cite[(2.5)]{Kronheimer2}.

\begin{lemma}\label{global}
Let $(C,\xi)$ be a Kähler cone whose Artin-Elkik deformation $\mathcal{W} \to S$ satisfies property $(1)$ above. Let the affine variety $V$ be a deformation of negative $\xi$-weight of $C$. Then $V$ is isomorphic as an affine variety to a connected component of some fiber of the family $\mathcal{W}\to S$.
\end{lemma}

\begin{proof}
Fix any element $\tilde\xi$ of the sequence $\xi_i \to \xi$ from Definition \ref{negativeweight} and consider the corresponding equivariant degeneration $W \to \C$ of $V$ to $(C,\tilde\xi)$. Let $0$ denote the origin in $\C$ as well as the point of $S$ over which $C$ lies. Let $o$ denote the apex of $C$. By \cite[pp.12--13]{Slodowy}, (1) implies (2), so there exist a $\C^*$-equivariant map $H: (\C,0) \to (S,0)$ of germs of complex-analytic spaces and an isomorphism $I: (W,o) \to H^*(\mathcal{W},o)$ of germs of $\C^*$-equivariant deformations of $(C,o)$.

We now globalize $H$ and $I$ by exploiting their equivariance.
The key point is the negative weight condition on the family $W \to \C$. In particular, the induced $\C^*$-action on the base is nontrivial, so for any $x \in \C$ there exists a $t \in \C^*$ such that $t \cdot x$ lies in the domain of $H$ and we can extend $H$ via $\hat{H}(x) := t^{-1} \cdot H(t \cdot x)$. Since $H$ is equivariant, this extension is well-defined and equivariant. Thus, if we embed $S$ into some $\C^N$ in such a way that the $\C^*$-action on $S$ becomes diagonal, then every component of $\hat{H}$ will be a homogeneous polynomial. We can now consider the pullback $\hat{H}^*\mathcal{W} \to \C$ in the algebraic category. By assumption, the map $I$ provides a local analytic isomorphism between the equivariant algebraic deformations $W \to \C$ and $\hat{H}^*\mathcal{W} \to \C$ of $(C,\tilde\xi)$ in some neighborhood of $o$. We may then globalize $I$ in the same way as $H$, using the fact that the $\C^*$-action on $W$ sends every point of $W$ into arbitrarily small neighborhoods of $o$. This yields a well-defined equivariant and fiber-preserving bijection $\hat{I}$ from $W$ onto $\hat{I}(W) \subset \hat{H}^*\mathcal{W}$. In fact, $\hat{I}(W)$ is open in the analytic topology and $\hat{I}$ is a local biholomorphism onto $\hat{I}(W)$ because in the definition of $\hat{I}$, i.e., $\hat{I}(x) := t^{-1} \cdot I(t \cdot x)$ for all $x \in W$ and $t \in \C^*$ such that $t \cdot x$ lies in the domain of $I$, we can choose $t \in \C^*$ to be locally independent of $x \in W$, as can be seen from Theorem \ref{normalform}(1) applied to $W$.

Embed $\hat{H}^*\mathcal{W}$ into some $\C^N$ in such a way that the $\C^*$-action on $\hat{H}^*\mathcal{W}$ becomes diagonal. Every component function of $\C^N$ either restricts to zero on the germ $(\hat{H}^*\mathcal{W},o)$ or has positive $\C^*$-weight because composition with $I$ produces a holomorphic function on the germ $(W,o)$ that vanishes at $o$ and is homogeneous under the given $\C^*$-action, which contracts $W$ into $o$. (As an aside, this shows that the components of $I$, hence of $\hat{I}$, are actually polynomials.) Now, a regular function on $\hat{H}^*\mathcal{W}$ vanishes identically on this germ if and only if it vanishes identically on all irreducible components of $\hat{H}^*\mathcal{W}$ that meet the germ, or equivalently (by continuity of the $\C^*$-action) that meet its $\C^*$-orbit, i.e., the complex-analytically open set $\hat{I}(W)$. Thus, the union of all the irreducible components of $\hat{H}^*\mathcal{W}$ that meet $\hat{I}(W)$ lies in an invariant linear subspace of $\C^N$ on which the $\C^*$-action has positive weights. Scaling now shows that this union is equal to $\hat{I}(W)$.  Thus, $\hat{I}(W)$ is a connected component of $\hat{H}^*\mathcal{W}$, hence is a disjoint union of connected components of Artin-Elkik fibers.

We already know that $\hat{I}$ is injective, locally (hence globally) biholomorphic onto some connected component of $\hat{H}^*\mathcal{W}$, and algebraic. Moreover, as the key to the proof of all of this, $\hat{I}$ is equivariant with respect to $\C^*$-actions on both sides that contract the respective variety into the point $o$. The same properties can now be proved for $\hat{I}^{-1}$ by reversing the role of domain and target. This shows that $\hat{I}$ is an isomorphism of affine varieties. Because $\hat{I}, \hat{I}^{-1}$ are by construction fiber-preserving, the lemma now follows by restricting $\hat{I}$ to the generic fiber, $V$, of the family $W \to \C$.
\end{proof}

\subsection{Proof of Theorem \ref{uniqueness}}\label{ss:proofC}

We now prove the following expanded version of Theorem \ref{uniqueness}. Thanks to Theorem \ref{classification} it suffices to identify the underlying complex manifolds, which is what we do here.

\begin{theorem}
Let $D$ be a K\"ahler-Einstein $($resp.~a toric non-K\"ahler-Einstein$)$ Fano manifold. For $k \in \N$ dividing $c_1(D)$ $($resp.~for $k=1)$, let $M^n$
be an {\rm AC} Calabi-Yau manifold with asymptotic cone $C = (\frac{1}{k}K_{D})^{\times}$ given by the Calabi ansatz $($resp.~by \cite{futaki}$)$. If $D$ is a del Pezzo surface of degree $\geq7$, $\mathbb{P}^{3}$, or a quadric $Q^{n-1}\subset\mathbb{P}^{n}$ with $k = n-1$, then Table \ref{t:table} lists the possibilities for $M \neq \C^n$.
\end{theorem}

\begin{table}
\begin{tabular}{|l|l|c|c|c|c|c|}\hline
				& $D$							& $k$	&  realizations of $C = (\frac{1}{k}K_{D})^{\times}$			& $M$							& structure of $M$			\\\hline\hline
\textnormal{(i)} 		& $\mathbb{P}^{2}$	 				& $1$ 	& $\C^{3}/\Z_{3}$ [rigid] 								& $K_{D}$ 						& $D$ 					\\\hline
\textnormal{(ii)} 	& $\mathbb{P}^{1}\times\mathbb{P}^{1}$	& $1$ 	& $\{\sum_{i=1}^{4}z_{i}^{2}=0$ in $\C^{4}\}/\Z_{2}$			& $K_{D}$ 						& $D$					\\\cline{5-6}
				& 								&  		& 												& $T^*\R\P^3$ 						& smoothing  				\\\hline
\textnormal{(iii)} 	& $\mathbb{P}^{1}\times\mathbb{P}^{1}$ 	& $2$ 	& $\sum_{i=1}^{4}z_{i}^{2}=0$ in $\C^{4}$					& $\mathcal{O}_{\P^1}(-1)^{\oplus 2}$	& $\mathbb{P}^{1}$ 			\\\cline{5-6}
				& 								&  		& 												& $T^*\mathbb{S}^3$				& smoothing   				\\\hline
\textnormal{(iv)} 	& $Q^{n-1}$, $n\geq4$ 				& $n-1$ 	& $\sum_{i=1}^{n+1}z_{i}^{2}=0$ in $\C^{n+1}$ 				& $T^*\mathbb{S}^n$				& smoothing				\\\hline
\textnormal{(v)} 	& ${\rm Bl}_{p}\P^2$					& $1$ 	& $20$ quadrics in $\C^{12}$ \cite{dP1-physics}			& $K_{D}$ 						& $D$					\\\cline{5-6}
				& 								&  		& [rigid]	    										& the flop of $K_{D}$		& $\mathbb{P}^{2}\bigvee_{p}\mathbb{P}^{1}$   \\\hline
\textnormal{(vi)} 	& ${\rm Bl}_{p,q}\P^2$ 				& $1$	& $14$ quadrics in $\C^{9}$ \cite{dP2-physics}				& $K_{D}$							& $D$ 	  				\\\cline{5-6}
				& 								&  		& 								& one of three  	& ${\rm Bl}_{p}\mathbb{P}^{2}\bigvee_{q}\mathbb{P}^{1}$   				\\\cline{6-6}
				& 								&  		& 								& distinct flops 	& $\mathbb{P}^{2}\bigvee_{p}\mathbb{P}^{1}\bigvee_{q}\mathbb{P}^{1}$  		\\\cline{6-6}
				& 								&  		& 								& of $K_{D}$ 	& $(\mathbb{P}^{1}\times\mathbb{P}^{1})\bigvee\mathbb{P}^{1}$  			\\\cline{5-6}
				& 								&  		& 								& ${\rm Bl}_{p}\mathbb{P}^{3}\setminus{\rm Bl}_{p}Q^{2}$ 	& smoothing   				\\\hline
\textnormal{(vii)} 	& $\mathbb{P}^{3}$					& $2$ 	& $\C^{4}/\Z_{2}$ [rigid]								& none 							& none				 	\\\hline
\textnormal{(viii)} 	& $\mathbb{P}^3$					& $1$ 	& $\C^{4}/\Z_{4}$ [rigid]								& $K_D$ 							& $D$ 					\\\hline
\end{tabular}
\bigskip\
\caption{\label{t:table}Classification results for special cones. $M$ is always either a smoothing or a resolution of $C$ (necessarily the latter if $C$ is rigid). For resolutions the last column shows the exceptional set. $\bigvee$ denotes a one-point union of subvarieties.}
\end{table}

The general construction of the flops mentioned here can be found in \cite[Example 4.8]{Kollar}.

\begin{proof}
Each cone in Table \ref{t:table} is toric or a complete intersection singularity, so its Artin-Elkik family satisfies property (1) from Section \ref{ss:def-thy} thanks to \cite{altmann, Slodowy}. Hence, by Lemma \ref{global}, every $\C^{*}$-equivariant deformation of negative $\xi$-weight of one of these cones is isomorphic to a connected component of a fiber of the Artin-Elkik family of the cone. Now, the cone in (i), (v), (vii) and (viii) is rigid, whereas the others have exactly one deformation, which is smooth. For (i), (vii), (viii) this follows from \cite{Schlessinger2}, for (v), (vi) from \cite[(9.1)]{altmann}, and for (iii), (iv) from \cite{Kas}. In principle, case (ii) is also covered by \cite{altmann}, but since we have no explicit reference for this computation we instead argue as follows.

In fact, some of these cases can also be treated by applying the classification of log-Fano varieties \cite[Defn 2.1.1]{AG5} to the compactified Remmert reduction $(Y, [D])$ of $M$, rather than via deformation methods. In (i) and (ii), $(Y,[D])$ is a del Pezzo $3$-fold of degree $9$ resp.~$8$ \cite[Defn 3.2.1]{AG5}. According to \cite[Rmk 3.2.6 and Thm 3.3.1]{AG5}, the only possible examples are $(\P^3, \mathcal{O}(2))$ and projective cones. In (iii) and (iv), $Y$ must be a quadric by \cite[Thm 3.1.14]{AG5}, and hence a projective cone.

It remains to classify all possible crepant resolutions $M$ of $C$, or at least those carrying a K\"ahler form. To this end, we will make use of the fact that $M$ is quasi-projective by Theorem \ref{existence}(2). For (i), (ii), (iii), (v), and (vi), observe that $C$ admits an obvious crepant resolution $M_0$. By a result of Mori \cite[Thm 3.5.1]{flops}, it therefore suffices to classify all possible flops of $M_0$ \cite[Defn 2.2.1]{flops}. In (i), (ii), and (viii), $M_0 = K_D$ cannot be flopped because $D$ does not contain any contractible curves; see \cite[Defn 2.1.1.2]{flops}. Regarding (iv) and (vii), it is easy to see that in these cases $C$ is terminal, so that the blow-down morphism $M \to C$ would have to be small. But \cite[p.2879, footnote]{Conlon} shows that this is not possible because in these cases $D$ has Picard rank $1$. Finally, in (iii), $M_{0}=\mathcal{O}_{\mathbb{P}^{1}}(-1)^{\oplus 2}$, and in (v) and (vi), $M_{0}=K_{D}$. The possible flops of these resolutions are commensurate with the contractible curves in their exceptional set \cite[Prop 2.1.2]{flops}, with the uniqueness of each flop
guaranteed by \cite[Prop 2.1.6]{flops}. In case (iii), the unique flop of $M_{0}$ is isomorphic to $M_{0}$ itself, whereas in (v) and (vi), we obtain one resp.~three distinct flops with exceptional sets as outlined in the table.
\end{proof}

\subsection{The case of toric cones}\label{ss:examples}

An application of Theorem \ref{existence} allows for a clean way of constructing examples of AC Calabi-Yau manifolds asymptotic to toric Calabi-Yau cones.
Because toric Calabi-Yau cones of dimension at least $4$ are rigid \cite[(6.3)]{altmann}, we only consider the case of dimension $3$. Via the Delzant construction, a toric K\"ahler cone of dimension $3$ can be identified combinatorially with a ``good'' rational polyhedral cone in $\R^{3}$ \cite[Defn 3.1]{cho}. Such cones of ``height $1$'', i.e., those defined by
primitive vectors whose first component may be taken to be $1$ \cite[Defn 3.2]{cho}, describe Gorenstein cones, which are in fact Calabi-Yau cones by the results of \cite{futaki}. A construction of a torus-equivariant semi-universal deformation (with connected fibers) of the Gorenstein cone singularity is then given by Altmann \cite{altmann},
where the torus action on the total space restricts to the action of the Reeb torus on the cone. By Lemma \ref{global}, every deformation of negative $\xi$-weight of the cone must be isomorphic to some fiber of Altmann's family. The following theorem provides a converse to this statement.

\begin{theorem}\label{toric-eg}
Let $C$ be a $3$-dimensional toric Calabi-Yau cone with cone metric $\omega_{0}$ and Reeb vector field $\xi$. Then every fiber $V$ of the Altmann family of $C$ is a deformation of negative $\xi$-weight of $C$, of $\xi$-weight $\leq -3$. Let $\pi:M\to V$ be a K\"ahler crepant resolution. Assume that $\mathfrak{k} \in H^2(M,\R)$ pairs positively with every component of ${\rm Exc}(\pi)$. Then $\mathfrak{k}$ contains a family of {\rm AC} Calabi-Yau metrics $\omega_{g}$ $(g \in (\C^*)^3)$ asymptotic to $g^*\omega_{0}$ under a fixed diffeomorphism independent of $\mathfrak{k},g$.
\end{theorem}

As remarked after the statement of Theorem \ref{existence}, the complex structure rate $\lambda_2$ of these new AC Calabi-Yau metrics may be taken to be $-3 + \epsilon$ for any $\epsilon>0$, and the metric rate $\lambda_1$ may be taken to be $\lambda_1 = -2 + \epsilon$ if $\mathfrak{k}|_L \neq 0$, where $L$ denotes the link at infinity of $V$, and $\lambda_1 = -3 + \epsilon$ if $\mathfrak{k}|_L = 0$. Clearly, if $V$ is itself smooth, which is of course the generic case, then $M=V$ and $\mathfrak{k} \in H^2(M,\R)$ is arbitrary. Thus, one recovers the main theorem of our previous article \cite{Conlon3} on the smoothing of the irregular cone over ${\rm Bl}_{p,q}\mathbb{P}^2$, with essentially zero effort and with vastly improved decay rates:~in the trivial K\"ahler class, $\mathfrak{k} = 0$, the rate is now seen to be $-3+\epsilon$ as compared to $-0.0128$ in \cite{Conlon3}.

\begin{proof}
Identify the Lie algebra $\mathfrak{t}$ of the maximal compact torus acting on $C$ with $\Z^3 \otimes \R$ by choosing a basis of the weight lattice of $\mathfrak{t}$. As mentioned before Theorem \ref{toric-eg}, this basis can be chosen in such a way that the Delzant polytope of $C$ is generated by vectors in $\Z^3$ whose first coordinate is $1$. Let $\Omega_{0}$ denote the canonical holomorphic volume form on $C$, let $r$ denote the radius function of $\omega_{0}$, and let $J_{0}$ denote the complex structure on $C$. As in \cite[Section 2]{sparky}, the Reeb vector field $\xi$ of $\omega_{0}$ lies in the set $S \subset \mathfrak{t}$ of all Reeb vector fields $\tilde{\xi}$ for which $\mathcal{L}_{\tilde{r}\partial_{\tilde{r}}}\Omega_{0}=3\Omega_{0}$, where $J_{0}(\tilde{r}\partial_{\tilde{r}})=\tilde{\xi}$. By \cite{sparky}, $S$ maps to an open polygon in the plane
$\{(3,x,y)\in\R^3: x,y\in\R\}$. Thus, $\xi=(3,a,b)$ for some $a,b\in\R$.

Choose a sequence of quasi-regular Reeb vector fields $\xi_{i}=(3,a_{i},b_{i})$ in $S$ with $a_{i},b_{i}\in \Q$ such that $\xi_{i}\to\xi$. For each $i$ we obtain a $\C^*$-equivariant degeneration from $V$ to $C$ by considering the flow of $-J(c_i\xi_i)$ on the total space of the Altmann family, where $c_{i}\in\N$ is minimal such that $c_{i}\xi_{i}\in\Z^{3}$. The base of the Altmann family is a subscheme of $T^1_C$ and the torus action induced by equivariance on the base agrees with the torus action on $T^1_C$ induced by the torus action on $C$. The only weight of $\xi_i$ on $T^1_C$ is $-3$ (see \cite[p.168, Theorem (i)]{Alty} or \cite[(2.9)]{altmannn}), so the weight of the induced $\C^*$-action on the base of this $1$-parameter degeneration is $\mu_i = 3c_i > 0$ and Definition \ref{negativeweight}(4) is satisfied. The fact that the Altmann family is semi-universal implies that each of these sub-deformations vanishes to order $k_{i}=1$. Thus, $\lambda_{i}=-k_i\mu_i/c_i=-3$ for all $i$ as in Definition \ref{negativeweight}(5), so that $V$ is seen to be a deformation of negative $\xi$-weight of $C$ in our sense, with $\xi$-weight at most $-3$.

The remainder of Theorem \ref{toric-eg} then follows from Theorem \ref{existence}.
\end{proof}

In the example of \cite{Conlon3} one can see explicitly how this proof works. In this example, $C$ is cut out by $14$ homogeneous quadrics in $\C^8$ \cite[(B.11)]{Conlon3} and $V$ is obtained by adding on linear terms to some of these quadrics \cite[(B.8)]{Conlon3}. The action of $\xi$ on $\C^8$ is diagonal \cite[(4.3)]{Conlon3}, and the $\xi$-weighted degree of the linear smoothing terms is exactly $3$ less than the $\xi$-weighted degree of the main terms.

Returning to Theorem \ref{toric-eg}, if $V$ is singular, the question remains as to which crepant resolutions of $V$, if any, carry a K\"ahler form. $V$ itself is almost never toric but its isolated
singularities do have this property; in fact, they themselves are toric cones \cite[Cor 2.12]{Ilten}. Isolated toric Gorenstein non- terminal $3$-dimensional conical singularities always admit a crepant resolution via iterated blowups \cite[Prop 3.3.15 and 11.4.17]{cox}, and $V$ is quasi-projective by Theorem \ref{existence}(2). Thus, $V$ always admits a quasi-projective partial crepant resolution $V'$ with only terminal isolated toric Gorenstein conical singularities. Such singularities are ordinary double points \cite[Thm 11.4.21(b)]{cox}, and consequently $V'$ admits
small (hence crepant) resolutions. Since $V'$ is $1$-convex, a small resolution of $V'$ is K\"ahler if and only if no positive integral linear combination of the exceptional curves is homologous to zero  \cite{alessandrini}. Thus, $V$ will admit at least one K\"ahler crepant resolution $M$ if and only if this is the case.

\appendix

\renewcommand{\thesection}{\Roman{section}}

\section{Deformation to the normal cone}\label{s:deform-to-nc}

Here we slightly generalize the standard fact that a tubular neighborhood of a  smooth complex hypersurface of a complex space can be degenerated to the total space of the hypersurface's normal bundle. There exist several versions of this ``deformation to the normal cone'' in the literature. We will show that the version of interest to us, which, assuming that the ambient space is compact and the divisor is ample, degenerates the complement of the tube to a singular point compactifying the normal bundle, still works for admissible divisors. Moreover, we clarify the subtle but crucial point that the complex structure of the resulting ``Thom space'' is not entirely determined by the normal bundle itself. This is less standard even in the smooth case but is well-known to experts.

\begin{theorem}\label{thm:deform-tech}
Let $Y$ be a projective variety. Let $D \subset Y$ be an admissible divisor whose associated $\Q$-line bundle $L$ is ample. Let $Y_0$ denote the normal projective variety obtained by contracting the
$\infty$-section of the $\P^1$-orbibundle $\P(N\oplus \C)$, where $N$ is the normal orbibundle to $D$ in $Y$.

{\rm (1)} There exists a test configuration \cite{donaldson, ross-thomas} $p: (\mathcal{Y}, \mathcal{L}) \to \C$ with general fiber $(Y,L)$ such that
\begin{itemize}
\item[$\bullet$] there exists an equivariant holomorphic homeomorphism $F$ from $Y_0$ onto the central fiber $\mathcal{Y}_0$ such that if $v$ is the apex of $Y_0$, then $F|_{Y_0 \setminus \{v\}}$ is a biholomorphism onto its image, and

\item[$\bullet$] $\mathcal{L}$ is the $\Q$-line bundle associated with a $\C^*$-invariant admissible divisor on $\mathcal{Y}$ that intersects the general fiber in $D \subset Y$ and the central fiber in $D = F(\P(0 \oplus \C))$.
\end{itemize}

{\rm (2)} The map $F: Y_0 \to \mathcal{Y}_0$ is the normalization morphism of $\mathcal{Y}_0$. It is an isomorphism if and only if the restriction map $H^0(Y, L^m) \to H^0(D, N^m)$ is surjective for every $m \in \N$.
\end{theorem}

\begin{remark}
As is implicit in the definition of a test configuration, the $\C^*$-action on $\mathcal{Y}$ covers the standard $\C^*$-action of weight $1$ on the base $\C$.
\end{remark}

\begin{proof}[Proof of Theorem \ref{thm:deform-tech}] We begin by recalling the well-known general construction that leads to (1). Let $\hat{p}: \hat{\mathcal{Y}} \to \P^1$ be the deformation to the normal cone \cite[\S 5.1]{Fulton} associated with $(Y,D)$. Thus, $\hat{\mathcal{Y}}$ is the blowup of $\P^1 \times Y$ in $\{0\} \times D$, and $\hat{p}$ is the induced projection onto $\P^1$. Then:
\begin{itemize}
\item[$\bullet$] $\hat{p}$ is a flat projective morphism.

\item[$\bullet$] All fibers of $\hat{p}$ except for the central one, $\hat{\mathcal{Y}}_0 = \hat{p}^{-1}(0) =Y \cup \mathbb{P}(N \oplus \C)$, are isomorphic to $Y$.

\item[$\bullet$] The two components of $\hat{\mathcal{Y}}_0$ intersect along $D = \mathbb{P}(N \oplus 0) \subset \P(N\oplus\C)$.

\item[$\bullet$] $\hat{p}$ is equivariant with respect to the natural $\C^*$-action on $\P^1$ and its lift to $\hat{\mathcal{Y}}$.

\item[$\bullet$] The strict transform of $\P^1 \times D$ defines a $\C^*$-invariant admissible divisor $\hat{\mathcal{D}}$.

\item[$\bullet$] $\hat{\mathcal{D}}$ intersects the general fiber in $D \subset Y$ and the central fiber in $D = \P(0 \oplus \C) \subset \P(N \oplus \C)$.
\end{itemize}

In order to construct the desired test configuration $p: \mathcal{Y} \to \C$, in addition to removing the fiber $\hat{\mathcal{Y}}_\infty = \hat{p}^{-1}(\infty)$ from $\hat{\mathcal{Y}}$, we also need to contract the component $Y \subset \hat{\mathcal{Y}}_0$ to a point. This is a local process taking place in a small analytic neighborhood of this component, and we use here that $L$ is ample. To see that $Y$ can indeed be contracted, it is helpful to realize that $\hat{\mathcal{Y}}$ can also be written as the blowup of $\mathbb{P}(L \oplus \C)$ in $D \subset Y = \mathbb{P}(0 \oplus \C)$, with exceptional divisor $\hat{\mathcal{D}} = \mathbb{P}(N \oplus N)$. Then $Y \subset \hat{\mathcal{Y}}_0$ equals the preimage of the $\infty$-section $\mathbb{P}(L \oplus 0)$, which can be contracted precisely because $L$ is positive \cite[p.340, Satz 5]{Grau:62}. Item (1) of the theorem is clear now.

To prove item (2), it suffices to compare the coordinate rings of the two affine algebraic varieties $\mathcal{Y}_0 \setminus D$ and $Y_0 \setminus D$. By construction, $\mathcal{Y}_0 \setminus D$ is the image of $N^* = L^*|_D$ under the contraction map $L^* \to (L^*)^\times = {\rm Spec}\, \bigoplus_{m \in \N_0} H^0(Y, L^m)$, whereas $Y_0 \setminus D = {\rm Spec} \,\bigoplus_{m \in \N_0} H^0(D, N^m)$ \cite[p.177, \S 8.8]{EGA2}. This yields a morphism $Y_0 \setminus D \to \mathcal{Y}_0 \setminus D$, which is an isomorphism if and only if the restriction map $H^0(Y,L^m) \to H^0(D, N^m)$ is surjective for every $m \in \N$. The underlying map of topological spaces is clearly equal to the homeomorphism $F$ of item (1), which is a biholomorphism away from $v$. Since $Y_0 \setminus D$ is normal, \cite[Thm 6.6]{Rossi2} now tells us that $F$ must be the normalization of $\mathcal{Y}_0 \setminus D$.
\end{proof}

\begin{example}\label{r:chili}
We are grateful to C.~Li for the following example, which shows that $\mathcal{Y}_0$ need not be isomorphic to $Y_0$. Let $Y$ be a smooth Riemann surface. Let $D$ be a point on $Y$. Then $Y_0 = \P^1$. However, if $\mathcal{Y}_0 = \P^1$, then $Y = \P^1$ because the arithmetic genus is constant in flat families.
\end{example}

\begin{example}\label{r:chili2}
We can apply Proposition \ref{normalform} to construct a more explicit example; cf.~Remark \ref{r:chili3}. The notation in this example will be analogous to the notation in the proof of Proposition \ref{normalform}. We apply the proposition with $W = (z_1^3 = z_2^2 + tz_2) \subset \C_t \times \C^2_{z_1,z_2}$, with weights $(\mu,\mu_1,\mu_2) = (3,2,3)$. Then the central fiber $C$ is the cuspidal cubic $z_1^3 = z_2^2$ in $\C^2$, which is irreducible but not normal, whereas the fibers for $t \neq 0$ are smooth. Compactifying $V = W \cap \{t = 1\}$ in $\P(3,2,3,1)$, we obtain a smooth elliptic curve $\overline{V}$. To see this, we set up an orbifold chart $$\C^3 \to \C^3/\Z_{2} \hookrightarrow \P(3,2,3,1), \;\, (\tilde\tau, \tilde{\zeta}_2, \tilde{w}) \mapsto [\tilde\tau,1,\tilde\zeta_2,\tilde{w}],$$
where $-1 \in \Z_2$ acts via $-{\rm Id}_{\C^3}$. Then $\overline{V}$ lifts to the $\Z_2$-invariant curve $\tilde\zeta_2^2 + \tilde\tau\tilde\zeta_2 = 1$, $\tilde\tau = \tilde{w}^3$, which is smooth at the preimage $\tilde{x} = (0,1,0)$ of the compactifying point $D = [0, 1, 1,0]$. The stabilizer $\Gamma$ of $\tilde{x}$ in $\Z_2$ is trivial, so the orbifold structure at $D$ is trivial and $\overline{V}$ is smooth as desired. We now read from Proposition \ref{normalform}(3) that after a base change $t = s^3$,  the family $W$ turns into the deformation of $Y = \overline{V}$ to the normal cone of $D$ with $\hat{\mathcal{D}}$ removed. Thus, $Y_0 = \mathbb{P}^1$, but $\mathcal{Y}_0$ is a rational curve with a cusp, and $F: Y_0 \to \mathcal{Y}_0$ is a holomorphic homeomorphism but not a biholomorphism.
\end{example}

See \cite{HRS} for some concrete examples of Theorem \ref{thm:deform-tech} with nontrivial orbifold points on $D$.

\begin{remark}
Consider again the smooth elliptic curve degenerating to a cusp of Example \ref{r:chili2}. It is instructive to see why the base change $t = s^3$ is necessary to recover the deformation to the normal cone. Of course, the $\C^*$-weight on the base must be equal to $1$, but there is a more interesting test: the fiber over $\infty$ must be isomorphic to the original variety $Y$. Before performing the base change $t = s^3$ in Example \ref{r:chili2}, the fiber over $t = \infty$ is the curve $\zeta_1^3 = \zeta_2^2 + \tau\zeta_2$  in $\P(3,2,3)$. Except for the compactifying point $D = [0,1,1]$, this curve is contained in the image of the orbifold chart
$$\C^2 \to \C^2/\Z_3 \hookrightarrow \P(3,2,3),\;\, (\tilde{\zeta}_1, \tilde{\zeta}_2) \mapsto [1,\tilde{\zeta}_1, \tilde{\zeta}_2],$$
where $\xi \in \Z_3$ acts on $\C^2$ via ${\rm diag}(\xi^2, 1)$. In this chart our curve lifts to $(\tilde{\zeta}_1^3 = \tilde\zeta_2^2 + \tilde\zeta_2) = Y \setminus D$, but this means that the fiber over $t=\infty$ is not $Y$ but rather a $\Z_3$-quotient of $Y$. The effect of the base change $t = s^3$ is precisely to undo this quotient map. Note that $\C^2/\Z_3 \cong \C^2$ via $(\tilde{\zeta}_1, \tilde{\zeta}_2) \mapsto (\tilde{\zeta}_1^3, \tilde{\zeta}_2)$, identifying $Y/\Z_3$ with a conic. Moreover, the cover $Y \to Y/\Z_3$ has exactly three branch points, all of order $3$; on $\tilde\zeta_1^3 = \tilde\zeta_2^2 + \tilde\zeta_2$ these are the points $(0,0)$ and $(0,-1)$ and the point at infinity.
\end{remark}

\section{Type I deformations of K\"ahler cones}\label{s:type1}

Consider a K\"ahler cone $C$ with radius function $r$, link $L$, cone metric $g_C = dr^2 \oplus  r^2 g_L$, parallel complex structure $J$, and Reeb vector field $\xi = J(r \partial_r)$. This structure can be equivalently encoded in terms of a \emph{Sasaki structure} $(\Phi, \xi, \eta, g_L)$ on $L$, where $\eta$ is a $1$-form on $L$ given by $\eta(X) = g_L(\xi,X)$ and $\Phi$ is an endomorphism of $TL$ given by $\Phi(X) = J(X - \eta(X)\xi)$. Here we will not write down the compatibility conditions between these data that allow one to construct a K\"ahler cone structure on $\R^+ \times L$ from a quadruple of the form $(\Phi,\xi,\eta,g_L)$. However, using these compatibility conditions, Takahashi proved the following remarkable theorem \cite[Thm A]{tak}.

\begin{theorem}[Takahashi]\label{taka}
Let $(L, \Phi,\xi,\eta,g_L)$ be a Sasaki manifold. Let $\tilde\xi$ be a vector field on $L$ that preserves the tensors $\Phi, \xi, \eta, g_L$ and that satisfies $\eta(\tilde{\xi}) = g_L(\xi,\tilde\xi) > 0$ pointwise on $L$. Then a new Sasaki structure on $L$ with Reeb vector field $\tilde\xi$ can be defined by setting
\begin{align}
\tilde{\eta}(X) &:=\eta(\tilde{\xi})^{-1} \eta(X),\\
\tilde{\Phi}(X) &:= \Phi(X - \tilde{\eta}(X)\tilde{\xi}),\\
\tilde{g}_L(X,Y) &:= \eta(\tilde{\xi})^{-1}g_L(X-\tilde{\eta}(X)\tilde{\xi},Y-\tilde{\eta}(Y)\tilde{\xi}) + \tilde{\eta}(X)\tilde{\eta}(Y).
\end{align}
\end{theorem}

This theorem is useful in complex geometry for two reasons. First, if $L$ is irregular as a Sasaki manifold, i.e., $\xi$ generates a torus $\T$ of isometries of dimension $> 1$, then every element $\tilde\xi \in {\rm Lie}(\T)$ sufficiently close to $\xi$ satisfies the hypotheses of the theorem. Second, the K\"ahler cone defined by the Sasaki structure with Reeb vector field $\tilde{\xi}$ is actually biholomorphic to the K\"ahler cone defined by the Sasaki structure with Reeb vector field $\xi$. This is a well-known folklore observation. In our next proposition, we prove this fact as well as a useful estimate for the biholomorphism.

\begin{prop}\label{p:typeIbiholo}
Let $L$ be compact without boundary. Let $(g_C,J)$ and $(\tilde{g}_C,\tilde{J})$ be the K\"ahler cone structures on $C = \R^+ \times L$ associated with the Sasaki structures $(\Phi,\xi,\eta,g_L )$ and $(\tilde\Phi,\tilde\xi,\tilde\eta,\tilde{g}_L)$ on $L$. Extend $\tilde{\xi}$ from $L$ to $C$ by scale invariance. Let $\Psi^t$ denote the time-$t$ flow of the vector field $-J\tilde{\xi}$ on $C$. Define a proper diffeomorphism $\Psi: C \to C$ by setting
\begin{align}
\Psi(r,x) = \Psi^{\log r}(1,x).\end{align} Then $\Psi$ preserves the vector field $\tilde\xi$ and satisfies $\Psi^*J = \tilde{J}$.
Moreover, for all $r > 0$,
\begin{align}\label{lalilu}
\min\{r^{\lambda_1}, r^{\lambda_2}\} \leq r \circ \Psi \leq \max\{r^{\lambda_1},r^{\lambda_2}\},
\end{align}
where $\lambda_1 = \min_L g_L(\xi,\tilde{\xi})$ and $\lambda_2 = \max_L g_L(\xi, \tilde{\xi})$.
\end{prop}

\begin{proof}
For the first property, fix $(r,x) \in C$. Let $(r_s,x_s) \in C$ denote the integral curve of $\tilde\xi$ through $(r,x)$. Then $r_s = r$ because $\tilde\xi$ is scale-invariant and tangent to the slices of the cone, and hence
$$\Psi_*(\tilde\xi|_{(r,x)}) = \frac{d}{ds}\biggr|_0 \Psi^{\log r_s}(1,x_s) = \Psi^{\log r}_*(\tilde\xi|_{(1,x)}) = \tilde{\xi}.$$
The last equality holds because $\Psi^t$ is the flow of $-J\tilde\xi$, and $[-J\tilde\xi,\tilde{\xi}]=0$ because $\tilde\xi$ is $J$-holomorphic. Indeed, both $J$ and $\tilde\xi$ are scale-invariant, and $\tilde\xi$ preserves $\Phi,\xi,\eta$ on the link $L = \{r = 1\}$.

Next, for $\lambda \in \R^+$, let $S_\lambda: C \to C$ denote the scaling map $(r,x) \mapsto (\lambda r, x)$. Then clearly
\begin{equation}\label{e:kommutator}
\Psi \circ S_\lambda = \Psi^{\log\lambda} \circ \Psi.
\end{equation}
Using the fact noted above that $\tilde\xi$, hence $-J\tilde\xi$, is $J$-holomorphic, this implies that
$$
S_\lambda^* \Psi^*J  =  \Psi^*(\Psi^{\log\lambda})^*J = \Psi^*J.
$$
Thus, $\Psi^*J$ is scale-invariant as well. It therefore suffices to prove the equality $\Psi^*J = \tilde{J}$ at $r = 1$. For this we need to prove that $J\Psi_*X = \Psi_*\tilde{J}X$ for all vectors $X \in TC$ at $r = 1$. If $X \in TL \subset TC$, then $\Psi_*X = X$. Thus, if $\tilde{J}X \in TL$ as well, which is equivalent to $X \in \ker\tilde\eta = \ker\eta$, then it suffices to note that $JX = \tilde{J}X$ by definition. It remains to check the cases $X = \tilde\xi$ and $X  = \partial_r$. To do so, first note that taking $\frac{d}{d\lambda}|_{\lambda = 1}$ of the identity \eqref{e:kommutator} at $r = 1$ yields that
$$\Psi_*\partial_r = -J\tilde\xi\;\,\text{at}\;\,r=1.$$ Using this, we easily obtain the desired equalities
\begin{align*}
J\Psi_*\tilde\xi = J\tilde\xi = -\Psi_*\partial_r = \Psi_* \tilde{J}\tilde\xi \;\,\text{and}\;\, J\Psi_*\partial_r = \tilde\xi = \Psi_*\tilde\xi = \Psi_*\tilde{J}\partial_r\;\,\text{at}\;\,r=1.
\end{align*}

Finally, to prove the inequalities stated in \eqref{lalilu}, write $\Psi^t(1,x) = (r_t, x_t)$, so that $r \circ \Psi = r_t$ for $t = \log r$. Then, because the flow $\Psi^t$ is generated by the vector field $-J\tilde{\xi}$,
$$\frac{dr_t}{dt} = g_C( -J\tilde\xi|_{(r_t,x_t)}, \partial_r) = g_L( \tilde\xi|_{x_t}, \xi|_{x_t}) r_t,$$
and hence $e^{\lambda_1 t} \leq r_t \leq e^{\lambda_2 t}$ for $t \geq 0$ and  $e^{\lambda_2 t} \leq r_t \leq e^{\lambda_1 t}$ for $t \leq 0$.
\end{proof}

\begin{remark} Using the fact that $\xi,\tilde\xi$ are commuting Killing fields for $g_L$, one checks that
$$X(g_L(\xi,\tilde{\xi})) = -2g_L(\nabla^{g_L}_{\tilde{\xi}}\xi, X) = -2g_C(J\nabla^{g_C}_{\tilde{\xi}}\partial_r,X) = -2g_L(\Phi\tilde\xi, X)$$
for all $X \in TL$. This shows that the above biholomorphism $\Psi: (C,\tilde{J}) \to (C,{J})$ is an isometry with respect to $\tilde{g}_C$ and $g_C$ if and only if $\tilde{\xi} = \xi$. In fact, already for $\tilde{\xi} = \lambda \xi$  $(\lambda \neq 1)$, the links $(L,\tilde{g}_L)$ and $(L, g_L)$ are not isometric at all, with $(L,\tilde{g}_L)$ being a nontrivial Berger deformation of $(L,g_L)$.
\end{remark}

Proposition \ref{p:typeIbiholo} shows that if we deform a K\"ahler cone as in Theorem \ref{taka} (a ``Type I'' deformation of Sasaki structures), then by applying a diffeomorphism we can arrange that the complex structure stays fixed, but this diffeomorphism will distort the radius function in a polynomial manner. So if the pointwise value of a scalar function on the underlying complex manifold decays polynomially as the original radius goes to infinity, it will also decay polynomially in terms of the deformed radius. However, an analogous statement for the lengths of \emph{tensor fields} with respect to the deformed cone metric (e.g., for the \emph{derivatives} of a scalar function) does not follow from Proposition \ref{p:typeIbiholo}, and work in our previous paper \cite{Conlon3} shows that such a statement can actually be false unless $\tilde{\xi}$ is much closer to $\xi$ than Theorem \ref{taka} requires. In fact, in the example studied in \cite{Conlon3}, the main difficulty was to find a neighborhood $U$ of $\xi$ in ${\rm Lie}(\T)$ large enough such that a certain interesting vector $\tilde{\xi}$ lies in $U$, but small enough such that for $\tilde{\xi}\in U$, certain tensors known to decay with respect to the deformed cone metric also decay with respect to the original one. In this paper, thanks to Li's work \cite{ChiLi}, we only need to consider this issue for arbitrarily small deformations of $\xi$, which is easier.

To state the result we need, let $(g_C, J)$ be a fixed K\"ahler cone structure on $C = \R^+ \times L$ as above, with Reeb vector field $\xi$ and Reeb torus $\T$. Then every $\tilde\xi \in {\rm Lie}(\T)$ with $g_L(\xi,\tilde{\xi})>0$ gives rise to a deformed K\"ahler cone structure $(\tilde{g}_C,\tilde{J})$ as in Theorem \ref{taka}, and we have a diffeomorphism $\Psi$ given by Proposition \ref{p:typeIbiholo} such that $\Psi^*J = \tilde{J}$. Define $\tilde{g}_C' := (\Psi^{-1})^*\tilde{g}_C$. Then $\tilde{g}_C'$ is a cone metric, K\"ahler with respect to $J$, with radius function $\tilde{r}' := r \circ \Psi^{-1}$ and with Reeb vector field $\Psi_*\tilde\xi = \tilde\xi$. Note that by \eqref{lalilu}, the diffeomorphisms $\Psi, \Psi^{-1}$ preserve each of the regions $\{r\geq 1\}$ and $\{r \leq 1\}$.

\begin{theorem}\label{p:typeIest}
For all $K \in \N$ and $\epsilon\in (0,1)$ there exists a neighborhood $U_{K,\epsilon}$ of $\xi$ in ${\rm Lie}(\T)$ such that for all $\tilde\xi \in U_{K,\epsilon}$,
the following estimates hold on the region $\{r \geq 1\}$\textup{:}
\begin{align}\label{e:saviors}
r^{1-\epsilon} \leq \tilde{r}' \leq r^{1+\epsilon}, \quad
(1-\epsilon) r^{-\epsilon}g_C \leq \tilde{g}_C' \leq (1+\epsilon) r^\epsilon g_C, \quad \sum_{k=1}^K r^k|\nabla^k_{g_C}\tilde{g}_C'|_{g_C} \leq \epsilon r^{\epsilon}.
\end{align}
\end{theorem}

In fact, it will be clear from the proof that if we fix any norm on ${\rm Lie}(\T)$, then $U_{K,\epsilon}$ may be chosen to contain a $c_K\epsilon$-ball around $\xi$, where $c_K > 0$ depends only on $(L,g_L)$, $\xi$, and $K$.

\begin{proof}
The statement about the radius functions is clear from Proposition \ref{p:typeIbiholo}.

To prove the pointwise inequalities of metric tensors, we first pull these inequalities back by $\Psi$, obtaining an equivalent statement in terms of $\Psi^*\tilde{g}_C' = \tilde{g}_C$, $\Psi^*g_C$, and $r \circ \Psi$. Up to renaming $\epsilon$, we can then replace $r \circ \Psi$ by $r$ without loss. We can also replace $\tilde{g}_C$ by $g_C$ because $\tilde{g}_C$ and $g_C$ are cone metrics with the same scaling vector field, hence are uniformly equivalent over the whole cone $C$, and by Theorem \ref{taka} the equivalence constants are bounded by $1\pm\epsilon$ for all $\tilde{\xi}$ in a sufficiently small neighborhood of $\xi$. Thus, it suffices to prove that for all $\tilde\xi$ sufficiently close to $\xi$, all $p = (r,x) \in C$ with $r \geq 1$, and all $v \in T_p C$, we have that
\begin{align}\label{ineedyou}
(1-\epsilon)r^{-\epsilon}|v|_{g_C}\leq |\Psi_*|_{p}v|_{g_C} \leq (1+\epsilon)r^\epsilon |v|_{g_C}.
\end{align}

To prove \eqref{ineedyou}, recall that $\Psi(r,x) = \Psi^{\log r}(1,x)$, where $\Psi^t$ denotes the flow of $-J\tilde{\xi}$. Using the fact that the vector field $X = -J\tilde\xi + J\xi$ commutes with $-J\xi = r\partial_r$, we can rewrite this definition as $\Psi(r,x) = \Phi^{\log r}(r,x)$, where $\Phi^t$ denotes the flow of $X$. Thus,
\begin{align}\label{ineedyouthree}
\Psi_*|_pv = \Phi^{\log r}_*|_p (v +r^{-1} dr(v) X|_p).
\end{align}
Now, according to a general identity (the case $k = 0$ of Lemma \ref{l:aneurysm} below, see \eqref{yabbadabba}),
\begin{align}\label{ineedyoutoo}
\frac{\nabla_{g_C}}{dt}(\Phi^t_*|_p w) = \nabla_{g_C} X|_{\Phi^t(p)}(\Phi^t_*|_pw)
\end{align}
for all $w\in T_pC$. Also, because $X$ is scale-invariant, we can assume after shrinking $U_{K,\epsilon}$ that
\begin{align}\label{number4}
\sup\nolimits_C (r^{-1}|X|_{g_C} + |\nabla_{g_C} X|_{g_C}) \leq \epsilon.
\end{align}
Thus, letting $w$ be the argument of $\Phi^{\log r}_*|_p$ on the right-hand side of \eqref{ineedyouthree} and applying Gronwall's lemma to the ODE \eqref{ineedyoutoo} up to time $t = \log r$,we exactly obtain the desired estimate \eqref{ineedyou}.

We are now left with proving the higher-order estimates in \eqref{e:saviors}. View the differential $\Psi_*$ as a section of the vector bundle $T^*C \otimes \Psi^*TC$ equipped with the metric and connection induced by $g_C$. Then we claim that it suffices to prove that for some constant $A = A_K$,
\begin{align}\label{andyoutoo}
\sum_{k=1}^{K} r^k |\nabla^k_{g_C} \Psi_*|_{g_C}\leq A r^\epsilon.
\end{align}
Indeed, after renaming $\epsilon$, it is clear from \eqref{andyoutoo} that \eqref{andyoutoo} also holds for the tensor $\Psi^*g_C$ instead of the tensor $\Psi_*$, which allows us to compare the connections of $g_C$ and of $\Psi^*g_C$. Then notice that $|\tilde{g}_C - g_C|_{g_C} \leq \epsilon$ and $r^k|\nabla_{g_C}^k \tilde{g}_C|_{g_C} \leq \epsilon$ for $k \in \{1,\ldots,K\}$ because $\tilde{g}_C,g_C$ are cone metrics with the same scaling vector field. Together with the above comparison of connections, this yields
$$\sum_{k=1}^K r^k |\nabla^k_{\Psi^*g_C} \tilde{g}_C|_{g_C} \leq \epsilon r^\epsilon.$$
Pulling back by $\Psi^{-1}$ and using the $C^0$ estimates in \eqref{e:saviors}, we obtain the  desired $C^K$ estimate.

Thus, it remains to prove \eqref{andyoutoo}. We will do so by computing the $k$-th derivative of  \eqref{ineedyouthree} using the chain rule. However, before doing this computation, we will first prove a preliminary estimate, \eqref{w00t}, which will allow us to estimate the terms that arise from the chain rule.

The key ingredient is Lemma \ref{l:aneurysm} below. The case $k = 0$ of this lemma was already used above. For a general $k$ and for all $v_1,\ldots,v_{k+1} \in T_pC$, the lemma yields an ODE of the form
\begin{equation}\label{trullala} \frac{\nabla_{g_C}}{dt}[(\nabla^k_{g_C}\Phi_\ast^t)(v_1,\ldots,v_{k+1})] = \nabla_{g_C} X|_{\Phi^t(p)}[(\nabla^k_{g_C}\Phi_\ast^t)(v_1,\ldots,v_{k+1})]  + \Theta^{t}.\end{equation}
The vector field $\Theta^{t}$ along the curve $\Phi^t(p)$ depends
on $\nabla_{g_C}^j\Phi^t_*$ and $\nabla_{g_C}^j R$ for $j \in \{0,\ldots,k-1\}$ and on $\nabla_{g_C}^j X$ for $j \in \{2,\ldots,k+1\}$, where $R$ denotes the curvature tensor of $g_C$. The fundamental solution of the corresponding homogeneous ODE is bounded by $e^{\pm\epsilon t}$ in operator norm thanks to \eqref{number4} and Gronwall's lemma. Using variation of parameters and the explicit form of $\Theta^t$ from \eqref{aneurysm} ($i \geq 2$), one can then prove by induction on $k \in \{0,\ldots,K\}$ that
\begin{align*}
|\Theta^t|_{g_C} \leq A e^{(-1+\epsilon)t},\quad |\nabla^k_{g_C} \Phi^t_*|_{g_C} \leq A e^{\epsilon t}.
\end{align*}
Here $A$ depends only on $K$ but we need to shrink the neighborhood $U_{K,\epsilon}$ in each step of the proof. Similar but slightly easier arguments then also show that for all $k,\ell\in\{0,\ldots,K\}$,
\begin{align}
\label{w00t}
\biggl|\frac{\nabla^\ell_{g_C}}{dt^\ell}(\nabla^k_{g_C}\Phi_\ast^t)\biggr|_{g_C} \leq A e^{\epsilon t}.
\end{align}
Indeed, the case $\ell = 0$ is what we just proved, the case $\ell = 1$ follows from this using \eqref{trullala}, and we then simply differentiate \eqref{trullala} by $t$ to continue (no ODE solution formulas are required).

We are now finally in position to differentiate \eqref{ineedyouthree} and thus prove \eqref{andyoutoo}. To actually calculate the $k$-th derivative of \eqref{ineedyouthree}, it is convenient to use Fa\`a di Bruno's formula. To estimate the resulting terms, we use \eqref{w00t} evaluated at $t = \log r$. We thus obtain that $|\nabla^k_{g_C}\Psi_*|_{g_C}$ is bounded by
\begin{align*}
A \sum_{\substack{d,m_1,\ldots,m_d\in\N_0\\m_1+2m_2+\cdots+dm_d=k}} \left(\sum_{\substack{\ell_1,\ell_2\in\N_0\\\ell_1+\ell_2=m_1+\cdots+m_d}}\left| \frac{\nabla_{g_C}^{\ell_1}}{dt^{\ell_1}}(\nabla_{g_C}^{\ell_2} \Phi^t_*)\Bigg|_{t = \log r}  \right|_{g_C}\right)\prod_{j=1}^d\;\biggl|\nabla^j_{g_C}\biggl(\log r, {\rm Id} +\frac{dr}{r} X\biggr)\biggr|_{g_C}^{m_j},
\end{align*}
which is in turn bounded by
\begin{align*}
A \sum_{\substack{d,m_1,\ldots,m_d\in\N_0\\m_1+2m_2+\cdots+dm_d=k}} \left(\sum_{\substack{\ell_1,\ell_2\in\N_0\\\ell_1+\ell_2=m_1+\cdots+m_d}}r^\epsilon\right)\prod_{j=1}^d r^{-jm_j} \leq A r^{-k+\epsilon}.
\end{align*}
This proves \eqref{andyoutoo} and hence the $C^K$ estimate in \eqref{e:saviors}.
\end{proof}

It remains to prove the covariant differentiation formula for linearized flows that we used above. This should be well-known but since we were unable to find a reference, we will prove it here.

For $k \in \N_0$, we write $S_{k+1}$ to denote the permutation group of $\{1,\ldots,k+1\}$. Given any function $f: S_{k+1} \to \R$ and any real vector space $V$, we define the associated \emph{shuffle operator} by
\begin{equation}
\mathfrak{S}: V^{\otimes(k+1)} \to V^{\otimes(k+1)}, \;\,v_1 \otimes \cdots \otimes v_{k+1} \mapsto \sum f(\sigma)v_{\sigma(1)}\otimes \cdots \otimes v_{\sigma(k+1)}.
\end{equation}
Here the sum runs over all $\sigma \in S_{k+1}$. This is a minor generalization of a well-known definition from combinatorics.
Given $i \in \{1,\ldots,k+1\}$, we also define $I_{i,k} := \{\alpha \in \N_0^i: |\alpha|+i=k+1\}$. If $M$ is a smooth manifold and $\Phi: M \to M$ is a smooth map, we view the differential $\Phi_*$ as a section of the bundle $T^*M \otimes \Phi^*TM$. Any connection $\nabla$ on $TM$ induces a connection on this bundle, which we also denote by $\nabla$. For $\alpha \in I_{i,k}$, we then define a bundle homomorphism
\begin{equation}
\nabla^\alpha\Phi_*:= \nabla^{\alpha_1}\Phi_* \otimes \cdots \otimes \nabla^{\alpha_i}\Phi_*: TM^{\otimes(k+1)} \to \Phi^*TM^{\otimes i}.
\end{equation}
This is a polynomial differential operator of order $\max \alpha$ in terms of $\Phi_*$. Thus, the highest possible order of an operator of this form is $k$, and this is attained only for $\alpha = (k) \in I_{1,k}$.

\begin{lemma}\label{l:aneurysm}
 For all $k \in \N_0$, $i \in \{1,\ldots,k+1\}$, $\alpha \in I_{i,k}$, and $j \in \{0,\ldots,i-2\}$, there exist shuffles $\mathfrak{S}_\alpha$ and $\mathfrak{S}_{\alpha,j}$ such that $\mathfrak{S}_{(k)} = {\rm Id}$ and such that the following holds. Let $M$ be a smooth manifold, let $\nabla$ be a torsion-free connection on $TM$, and let $R$ be the type $(3,1)$ curvature tensor of $\nabla$. Let $X$ be a smooth vector field on $M$ with maximal local flow $\Phi^t$. Then the homomorphism
\begin{equation}
\mathfrak{L}^{k,t}: (TM)^{\otimes(k+1)} \to (\Phi^t)^*TM, \;\, v_1 \otimes \cdots \otimes v_{k+1} \mapsto \frac{\nabla}{dt}[(\nabla^k\Phi^t_\ast)(v_1, \ldots, v_{k+1})],
\end{equation}
can be expressed in terms of $X,R$ and their covariant derivatives via
\begin{align}\label{aneurysm}
\mathfrak{L}^{k,t} = \sum_{i=1}^{k+1} \sum_{\alpha\in I_{i,k}} \biggl\{\nabla^i X|_{\Phi^t}  \circ \nabla^{\alpha}\Phi^t_* \circ \mathfrak{S}_{\alpha}
+  \sum_{j=0}^{i-2}  ((\nabla^j_\bullet R)(\nabla^{i-2-j}_\bullet X, \bullet)\bullet)|_{\Phi^t} \circ  \nabla^{\alpha}\Phi^t_* \circ \mathfrak{S}_{\alpha,j}\biggr\}.
\end{align}
\end{lemma}

Notice that the sum over $j$ is empty unless $i \geq 2$, and if $i \geq 2$ then $\max \alpha \leq |\alpha| \leq k-1$. Thus, \eqref{aneurysm} contains only one term of the highest possible order of differentiation, $k$, with respect to $\Phi^t_*$. As expected, this term takes the form $\nabla X|_{\Phi^t} \circ \nabla^k\Phi^t_*$, with no need to shuffle the arguments.

\begin{proof}
We prove this by induction on $k$. Assume that the lemma is true with $k$ replaced by $k-1$. Choose an arbitrary point $p \in M$ and tangent vectors $v_1, \ldots, v_{k+1} \in T_pM$, and consider the vector field $V(t) = (\nabla^k\Phi^t_*)(v_1,\ldots,v_{k+1})$ along the integral curve $\gamma(t) = \Phi^t(p)$ of $X$ starting at $p$. We need to compute the covariant derivative of $V$ along $\gamma$. To this end, for $\delta = \delta(p) > 0$ sufficiently small, extend $\gamma$ to a map $\gimel: (-\delta,\delta)^{k+2} \to M$ such that $\gimel(0,\ldots,0) = p$ and such that if $s_1,\ldots,s_{k+1},t$ are the standard coordinates on $(-\delta,\delta)^{k+2}$, then
$$\gimel_{s_i}(0,\ldots,0) = v_i\;\,\text{for all $i$, and}\;\, \gimel(s_1,\ldots,s_{k+1},t) = \Phi^t(\gamma(s_1,\ldots,s_k,0)).$$
For $k = 0$, using our assumption that $\nabla$ is torsion-free, we now compute
\begin{align}\label{yabbadabba}
\frac{\nabla V}{dt} = \frac{\nabla\gimel_{s_1}}{\partial t} = \frac{\nabla \gimel_t}{\partial s_1} = \nabla X|_{\Phi_t(p)}(\Phi^t_* v_1),
\end{align}
which agrees with \eqref{aneurysm}. For $k \geq 1$, by unpacking the definitions,
\begin{align}\label{plethysm0}
\frac{\nabla V}{dt} &= \frac{\nabla}{\partial t}\biggl(\frac{\nabla}{\partial s_1}[(\nabla^{k-1}\Phi^t_*)(\gimel_{s_2},\ldots,\gimel_{s_{k+1}})]\biggr) - \frac{\nabla}{dt}\biggl(\sum_{\ell= 2}^{k+1} (\nabla^{k-1}\Phi^t_*)(v_2,\ldots, \frac{\nabla\gimel_{s_\ell}}{\partial s_1}, \ldots, v_{k+1})\biggr).
\end{align}
Commuting covariant derivatives, the first term here can be rewritten as
\begin{align}
&\frac{\nabla}{\partial s_1}\biggl(\frac{\nabla}{\partial t}[(\nabla^{k-1}\Phi^t_*)(\gimel_{s_2}, \ldots, \gimel_{s_{k+1}})]\biggr)  + R(X|_{\Phi^t}, \Phi^t_*v_1)((\nabla^{k-1}\Phi^t_*)(v_2,\ldots,v_{k+1})).\label{plethysm1}
\end{align}
The $R$ term has the desired form with $i = 2$, $\alpha = (0,k-1) \in I_{2,k}$, $j = 0$ in \eqref{aneurysm}. There is no need to permute the arguments $v_1,\ldots,v_{k+1}$ for this term, so we can set $\mathfrak{S}_{\alpha,j} = {\rm Id}$ for these values of $\alpha,j$, and if the same $\alpha, j$ reappear with permuted arguments later, we simply update $\mathfrak{S}_{\alpha,j}$ by adding the relevant permutation. Now, by induction, the first term in \eqref{plethysm1} is equal to
\begin{align*}
\sum_{i=1}^k \sum_{\alpha\in I_{i,k-1}} \biggl(\biggl[\nabla_{v_1}\biggl\{\eqref{aneurysm}\biggr\}\biggr](v_2 \otimes \cdots \otimes v_{k+1}) + \sum_{\ell=2}^{k+1} \biggl\{\eqref{aneurysm}\biggr\}\biggl[v_2 \otimes \cdots
\otimes \frac{\nabla \gimel_{s_\ell}}{\partial s_1} \otimes \cdots \otimes v_{k+1}  \biggr]\biggr).
\end{align*}
The second term here cancels with the second term in \eqref{plethysm0} by induction, and the first term can easily be brought into the desired form by shifting the indices $i,\alpha,j$.  This concludes the inductive step. Notice that terms of the form $(\nabla_{v_1}(\nabla^\alpha\Phi^t_*))(v_2 \otimes \cdots \otimes v_{k+1})$ with $i \geq 2$ force us to introduce nontrivial shuffles of the arguments, and $i \geq 2$ is possible as soon as $k \geq 2$. However, there
are no shuffles for $i = 1$ (the unique term of highest order with respect to $\Phi^t_*$), so that $\mathfrak{S}_{(k)} = {\rm Id}$.
\end{proof}

\section{A Gysin type theorem for orbifold circle bundles}\label{s:gysin}

The purpose of this section is to prove the following theorem.

\begin{theorem}\label{th:gysin}
Let $p: N \to D$ be a $C^\infty$ complex orbifold line bundle on a $C^\infty$ orbifold $D$. Let $L$ be the unit circle bundle of $N$ with respect to some Hermitian metric on $N$. Let $Q$ be a $C^\infty$ complex orbifold line bundle on $D$. Then $(p|_L)^*Q$ is trivial as a $C^\infty$ complex orbifold line bundle on $L$ if and only if $Q$ is isomorphic to $N^{-\alpha}$ as a $C^\infty$ complex orbifold line bundle on $D$ for some $\alpha \in \Z$.
\end{theorem}

Note that there is no obvious general definition of the pullback of an orbifold vector bundle along a map of orbifolds. However, the relevant map in the theorem is the projection of an orbifold fiber bundle onto its base, and for such maps the pullback is defined carefully in \cite[Rmk 4.30]{faulk}.

The convention $-\alpha$ for the exponent is chosen to match conventions elsewhere in this paper.

If $D$ is actually a manifold, then by identifying $C^\infty$ complex line bundles with their first Chern classes, the theorem becomes equivalent to the statement that ${\rm ker}((p|_L)^*: H^2(D,\Z) \to H^2(L,\Z))$ is the subgroup generated by $c_1(N)$. This is a well-known consequence of the Gysin sequence over $\Z$. There does exist an orbifold version of the Gysin sequence over $\Q$ \cite[Prop 4.7.9]{book:Boyer}. This suffices to prove a slightly weaker version of Theorem \ref{existence}, which would only state that $K_V$ is torsion in general, and that $K_V$ is trivial whenever there exists a K\"ahler crepant resolution $\pi: M \to V$.

Our proof of Theorem \ref{th:gysin} bypasses the use of Chern classes even when $D$ is a manifold. We are grateful to C.-C.~M.~Liu for suggesting the key idea, which is to view $(p|_L)^*Q$ as an $S^1$-equivariant line bundle on $L$ and to classify the possible actions of $S^1$ on the trivial line bundle $L \times \C$.

\begin{proof}[Proof of Theorem \ref{th:gysin}] Cover $D$ by an atlas of uniformizing charts that trivializes $L$ as a principal $S^1$-orbibundle and $Q$ as a complex line orbibundle. Thus, for each chart $\varphi: \tilde{U} \to U$ of the atlas we may write $L|_U = (\tilde{U} \times S^1)/\Gamma$ and $Q|_U = (\tilde{U}\times\C)/\Gamma$,  where $\Gamma$ is the local uniformizing group of $D$ acting effectively on $\tilde{U}$ and where the action of $\Gamma$ on the fibers is given by homomorphisms $\Gamma \to S^1$ for $L$ and $\Gamma \to \C^*$ for $Q$. The fiberwise right action of $S^1$ on $\tilde{U} \times S^1$ commutes with the left action of $\Gamma$ and with the left action of the transition functions of $L$. We therefore obtain a right action of $S^1$ on $L$ with finite stabilizers, and with trivial stabilizers on the open dense set $L|_{D^{\rm reg}}$. The pullback bundle $(p|_L)^*Q$ takes the form $(\tilde{U} \times S^1 \times \C)/\Gamma$ locally for the product action of $\Gamma$ on $S^1 \times \C$, and its transition functions are pulled back from $\tilde{U}$ under the projection $\tilde{U} \times S^1 \to \tilde{U}$. This picture allows us to lift the right action of $S^1$ by orbifold diffeomorphisms on $L$ to a right $S^1$-action by orbibundle automorphisms on $(p|_L)^*Q$. In fact, we can simply let $S^1$ act trivially on the local $\C$-factors. Then the original orbibundle $Q$ can be recovered as the $S^1$-quotient of $(p|_L)^*Q$. The same construction would have gone through for any other Lie group $G$ instead of $S^1$ as the structure group of $L$, and amounts to a version of the isomorphism $K_G(X) \cong K(X/G)$ of \cite[Prop 2.1]{Segal}.

We will now use the $S^1$-action on $(p|_L)^*Q$ constructed in the previous paragraph to characterize the case that $(p|_L)^*Q$ is trivial as a complex line orbibundle. More precisely, we will prove that up to a gauge transformation of $L \times \C$, the only possible $S^1$-actions on $L \times \C$ are those of the form
$$(\ell,z)g = (\ell g, g^{\alpha} z)$$  for some $\alpha \in \Z$ and all $(\ell,z) \in L \times \C$ and $g \in S^1$. The theorem follows from this. Indeed, it suffices to observe that $L \times \C$ together with the $S^1$-action $(\ell,z)g= (\ell g, g^{\alpha} z)$ is equivariantly isomorphic to $(p|_L)^*N^{-\alpha}$. To see this, we go through the above construction of an $S^1$-action on $(p|_L)^*N^{-\alpha}$, using the same trivializations for the principal $S^1$-orbibundle $L$ and for the complex line orbibundle $N$. In the resulting local charts, the desired isomorphism from $(p|_L)^*N^{-\alpha}$ to $L \times \C$ is given by
\begin{align*}
\begin{split}
(\tilde{U} \times S^1 \times \C)/\Gamma &\to ((\tilde{U} \times S^1)/\Gamma) \times \C,\\
[(\tilde{u}, h, z)] &\mapsto ([(\tilde{u},h)], h^{\alpha}z).
\end{split}
\end{align*}
One easily checks that this is well-defined, equivariant, and consistent under changes of charts.

It remains to classify all possible $S^1$-actions on $L \times \C$. Any such action takes the form
$$(\ell,z)g = (\ell g, F(\ell,g)z)$$
for some smooth map $F: L \times S^1 \to \C^*$. By definition, for all $\ell \in L$ and $g,h \in S^1$,
\begin{equation}\label{e:cocycle}F(\ell, gh) = F(\ell g,h)F(\ell,g).\end{equation}
We are free to apply a gauge transformation of the trivial complex line bundle $L \times \C$, i.e., a smooth map $E: L \to \C^*$.
This changes $F$ to $F^E$ defined by
$$F^E(\ell,g) := E(\ell g)^{-1}F(\ell,g)E(\ell).$$
Our goal is to find an $E$ such that $F^{E}(\ell,g) = g^{\alpha}$ for some $\alpha \in \Z$ and for all $\ell \in L$, $g \in S^1$.

Fix any point $\ell_0 \in L$. Then the smooth map $F(\ell_0, \cdot): S^1 \to \C^*$ has a winding number $\alpha \in \Z$, which is independent of $\ell_0$. As the winding number is additive under pointwise multiplication, and is equal to zero if and only if the map has a well-defined smooth $\log$, we may write $$F(\ell_0, g) = g^{\alpha} e^{f(\ell_0,g)}$$
for some smooth function $f: L \times S^1 \to \C$ unique up to a constant in $2\pi i \Z$. After changing $f$ by a suitable constant, the cocycle condition \eqref{e:cocycle} may be rewritten as
\begin{equation}\label{e:cocycle2}f(\ell,gh) = f(\ell g,h) + f(\ell,g).\end{equation}
Using \eqref{e:cocycle2}, we will now construct the desired gauge transformation $E$ by working in local charts for the bundle $L$ and checking that our definition is consistent under changes of charts.

Fix a uniformizing chart $\varphi: \tilde{U} \to U$ for $D$ with local group $\Gamma$. Consider the local representation $L|_U = (\tilde{U} \times S^1)/\Gamma$, where $\Gamma$ acts on the left on $S^1$ via some homomorphism $\iota: \Gamma \to S^1$. Note that $\iota$ is not required to be injective, which would be equivalent to $L|_U$ being a manifold. Let $\tilde{F},\tilde{f}$ denote the lifts of $F,f$ from $L|_U \times S^1$ to $(\tilde{U}\times S^1) \times S^1$, respectively. Given any $h_0 \in S^1$, consider
\begin{align*}
\begin{split}
\tilde{E}_{h_0}: \tilde{U} \times S^1 &\to \C^*,\\
(\tilde{u},h) &\mapsto e^{\tilde{f}((\tilde{u},h_0),h_0^{-1}h) - C_{h_0}(\tilde{u})},
\end{split}
\end{align*}
where by definition
$$C_{h_0}(\tilde{u}) := \av f((\tilde{u},h_0),h_0^{-1}g)\,dg . $$
Then the following properties (1)--(5) lead to the desired gauge transformation $E$.

(1) \emph{The lift $\tilde{f}: (\tilde{U} \times S^1) \times S^1 \to \C^*$ satisfies the cocycle property}
\begin{equation}\label{e:cocycle3}\tilde{f}(\tilde{\ell}, gh) = \tilde{f}(\tilde{\ell} g, h) + \tilde{f}(\tilde{\ell}, g) \;\, \text{\emph{for all}}\;\, \tilde{\ell} \in \tilde{U} \times S^1 \;\,\text{\emph{and}}\;\, g, h \in S^1.\end{equation}

\emph{Proof}. For $\tilde{\ell} \in \tilde{U} \times S^1$, let $\ell$ be the projection of $\tilde\ell$ to the quotient $(\tilde{U}\times S^1)/\Gamma = L|_U$. Then it suffices to apply \eqref{e:cocycle2} and observe that $\tilde{\ell}g$ projects to $\ell g$ in the quotient space by construction.

(2) \emph{For any two elements $h_0,h_0' \in S^1$ it holds that $\tilde{E}_{h_0} = \tilde{E}_{h_0'}$.}

\emph{Proof}. Thanks to \eqref{e:cocycle3} it holds for all $(\tilde{u},h) \in \tilde{U}\times S^1$ that
$$\tilde{f}((\tilde{u},h_0'),(h_0')^{-1}h) = \tilde{f}((\tilde{u},h_0),h_0^{-1}h) + \tilde{f}((\tilde{u},h_0'),(h_0')^{-1}h_0).$$
The second term depends on $\tilde{u}$ but not on $h$. Thus,
$$\tilde{f}((\tilde{u},h_0'),(h_0')^{-1}h) - C_{h_0'}(\tilde{u}) = \tilde{f}((\tilde{u},h_0),h_0^{-1}h) - C_{h_0}(\tilde{u}).$$
Property (2) is now obvious. In particular, this allows us to write $\tilde{E}$ instead of $\tilde{E}_{h_0}$.

(3) \emph{The map $\tilde{E}: \tilde{U} \times S^1 \to \C^*$ is smooth and $\Gamma$-invariant.}

\emph{Proof}. Smoothness is clear from the definition of $\tilde{E} = \tilde{E}_{h_0}$ for any fixed $h_0 \in S^1$. To prove $\Gamma$- invariance, consider an arbitrary element $\gamma \in \Gamma$. Then for all $(\tilde{u},h) \in \tilde{U}\times S^1$ and $h_0 \in S^1$,
\begin{align}
\tilde{E}(\gamma(\tilde{u},h)) = \tilde{E}_{h_0}(\gamma\tilde{u},\iota(\gamma)h) = e^{\tilde{f}((\gamma\tilde{u},h_0),h_0^{-1}\iota(\gamma)h )- C_{h_0}(\gamma \tilde{u})}.\label{e:what1}
\end{align}
Note for later reference that
\begin{eqnarray}
C_{h_0}(\gamma\tilde{u}) =  \av f((\gamma\tilde{u},h_0),h_0^{-1}g)\,dg =
 \av f((\gamma\tilde{u},h_0),h_0^{-1}\iota(\gamma)g)\,dg.
\label{e:what2}
\end{eqnarray}
Going back to \eqref{e:what1}, using the fact that $S^1$ is abelian, \eqref{e:cocycle3}, and the $\Gamma$-invariance of $\tilde{f}$,
\begin{align*}
\begin{split}
\tilde{f}((\gamma\tilde{u},h_0),h_0^{-1}\iota(\gamma)h) &= \tilde{f}((\gamma\tilde{u},h_0)\iota(\gamma),h_0^{-1}h) + f((\gamma\tilde{u},h_0),\iota(\gamma))\\
& = \tilde{f}((\gamma\tilde{u},\iota(\gamma)h_0),h_0^{-1}h) + f((\gamma\tilde{u},h_0),\iota(\gamma))\\
&= \tilde{f}((\tilde{u},h_0),h_0^{-1}h) + f((\gamma\tilde{u},h_0),\iota(\gamma)).
\end{split}
\end{align*}
The second term depends on $\tilde{u}$ but not on $h$. Thus, using \eqref{e:what2},
\begin{align*}
\begin{split}
\tilde{f}((\gamma\tilde{u},h_0),h_0^{-1}\iota(\gamma)h) - C_{h_0}(\gamma\tilde{u})&= \tilde{f}((\tilde{u},h_0),h_0^{-1}h) - C_{h_0}(\tilde{u}).
\end{split}
\end{align*}
This tells us that $\tilde{E}(\gamma(\tilde{u},h)) = \tilde{E}(\tilde{u},h)$, as desired.

(4) \emph{The $\Gamma$-invariant smooth maps $\tilde{E}: \tilde{U}\times S^1 \to \C^*$ patch up to a smooth map $E: L \to \C^*$.  }

\emph{Proof}. Let $(\tilde{u},h_0) \in \tilde{U}\times S^1$ correspond to $(\tilde{u}', \psi(\tilde{u}')h_0)$ in a different uniformizing chart $\tilde{U}' \times S^1$, where $\tilde{u}\mapsto\tilde{u}'$ represents a transition map of $D$ and $\psi(\tilde{u}') \in S^1$ represents the associated transition map of the principal orbibundle $L$. Because $f$ is a globally defined function on $L \times S^1$, the lift $\tilde{f}'$ of $f$ to $\tilde{U}' \times S^1$ satisfies for all $h \in S^1$ that
$$\tilde{f}((\tilde{u},h_0),h_0^{-1}h) = \tilde{f}'((\tilde{u}',\psi(\tilde{u}')h_0),h_0^{-1}h) = \tilde{f}'((\tilde{u}',h_0'),(h_0')^{-1}h'),$$
where $h_0' = \psi(\tilde{u}')h_0 \in S^1$ and $h' = \psi(\tilde{u}')h \in S^1$. The average of the left-hand side with respect to $dh$ is the constant $C_{h_0}(\tilde{u})$, and the average of the right-hand side with respect to $dh'$ is the analogous constant $C'_{h_0'}(\tilde{u}')$ defined using the chart $\tilde{U}' \times S^1$. Since $h' = \psi(\tilde{u}')h$, we obtain that
$$C_{h_0}(\tilde{u}) = C'_{h_0'}(\tilde{u}').$$
Thus, writing $\tilde{E}'$ for the analog of $\tilde{E}$ defined using the chart $\tilde{U}' \times S^1$,
$$\tilde{E}(\tilde{u},h) = \tilde{E}_{h_0}(\tilde{u},h) = \tilde{E}_{h_0'}'(\tilde{u}', h') = \tilde{E}'(\tilde{u}', h').$$
This is the desired patching property. Notice that $h_0'$ depends not only on $h_0$ but also on $\tilde{u}'$.

(5) \emph{The smooth gauge transformation $E$ satisfies $F^E(\ell,g) = g^{\alpha} $ for all $\ell \in L$, $g \in S^1$.}

\emph{Proof}. Let $\varphi: \tilde{U} \to U$ be a uniformizing chart for $D$ with $p(\ell) \in U$. Let  $\tilde{\ell} = (\tilde{u},h) \in \tilde{U}\times S^1$ be a lift of $\ell \in L|_U = (\tilde{U}\times S^1)/\Gamma$ under the local group $\Gamma$. Then $\tilde\ell g$ is a lift of $\ell g$ by construction. It therefore suffices to verify for an arbitrary $h_0 \in S^1$ that
$$\tilde{E}_{h_0}(\tilde{\ell}g)^{-1} \tilde{F}(\tilde{\ell},g) \tilde{E}_{h_0}(\tilde{\ell}) = g^{\alpha}.$$
This is straightforward. Indeed, by definition,
$$
\tilde{E}_{h_0}(\tilde{\ell}g)^{-1} \tilde{F}(\tilde{\ell},g) \tilde{E}_{h_0}(\tilde{\ell}) = e^{-\tilde{f}((\tilde{u},h_0),h_0^{-1}hg)+ C_{h_0}(\tilde{u})} g^{\alpha} e^{\tilde{f}(\tilde{\ell},g)} e^{\tilde{f}((\tilde{u},h_0),h_0^{-1}h) - C_{h_0}(\tilde{u})},
$$
so it suffices to show that
$$\tilde{f}((\tilde{u},h_0),h_0^{-1}hg) = \tilde{f}((\tilde{u},h),g)+\tilde{f}((\tilde{u},h_0),h_0^{-1}h).$$
But this is immediate from the cocycle property \eqref{e:cocycle3}.

As explained above, property (5) concludes the proof of the theorem.
\end{proof}

\section{Review and extension of a theorem of Chi Li}\label{s:chili}

Our goal in this appendix is to prove Theorem \ref{chili}, an orbifold version of a difficult result of Li \cite[Thm 1.5]{ChiLi}. We review Li's proof in Section \ref{ss:ChiLiReview}. Fortunately, the most difficult steps carry over verbatim to the spaces studied in this paper. The remaining steps need to be generalized and this generalization is almost completely formal as well. However, to reach the point where the manifold case and the orbifold case become formally identical, we need one nonobvious fact:~the deformation of an orbifold to the normal cone of a suborbifold is trivial on small open sets. This fact allows us to encode properties of this deformation in terms of \v{C}ech cohomology classes in the same way as in the manifold case. We will explain this in Section \ref{ss:ChiLiExtension} as a consequence of Proposition \ref{gauge}.

In this section, cohomology means \v{C}ech cohomology with respect to sufficiently fine coverings by Stein open sets, with coefficients in the sheaf of sections of a holomorphic (orbi-)vector bundle.

\subsection{Review}\label{ss:ChiLiReview} In \cite{ChiLi}, Li considers the following situation. Let $D$ be a smooth compact complex manifold with $\dim D \geq 2$, let $L$ be an ample holomorphic line bundle on $D$, and let the affine cone $C$ be the dual of $L$ with its zero section contracted to a point $o$. Suppose $X$ is a projective variety, $U$ is an open subset of $X^{\rm reg}$, and $D$ is embedded into $U$ as a complex submanifold with normal bundle $L$. Consider the deformation of $X$ to the normal cone of $D$ as in Theorem \ref{thm:deform-tech}. This is a $\C^*$-equivariant degeneration $\{X_t\}_{t \in \C}$ covering the $\C^*$-action of weight $1$ on the base. The central fiber $X_0$ is a projective variety, which may not be normal but is always normalized by the compactified cone $C \cup D$.  Assume (crucially) that $X_0$ actually is normal and is therefore equal to $C \cup D$.

Li's main result for our purposes \cite[Thm 1.5]{ChiLi} may then be stated as follows:

\begin{theorem}\label{main-chili}
In the above situation, the supremum of all positive integers $k$ such that the affine family $\{X_t \setminus D\}_{t\in\C}$ becomes trivial after base change to ${\rm Spec}\;\C[t]/(t^k)$ is equal to the supremum of all positive integers $k$ such that the embedding $D\hookrightarrow U$ is $(k-1)$-comfortable.
\end{theorem}

The first quantity appearing in this statement is simply the vanishing order of the affine family; see Definition \ref{negativeweight}(5). The notion of comfortable embeddedness is much more subtle and we will not
\newpage

\noindent state the general definition here \cite[Defn 3.1]{abt}. It is a concept valid for any complex submanifold of a complex manifold. However, for a divisor of dimension at least $2$ with positive normal bundle it turns out to be equivalent to the more familiar notion of triviality of infinitesimal neighborhoods. In order to explain, we recall the following definition \cite[Defn 4.1]{abt}:

\begin{definition}\label{defn-linble}
A complex submanifold $S$ of a complex manifold $M$ is \emph{$k$-linearizable} if its $k$-th infinitesimal neighborhood $(S, \mathcal{O}_M/\mathcal{J}_S^{k+1})$ in $M$ is isomorphic to its $k$-th infinitesimal neighborhood $(S, \mathcal{O}_{N}/\mathcal{J}_S^{k+1})$ in the normal bundle $N = N_{S/M}$, where $S$ is identified with the
zero section of $N$.
\end{definition}

$(k-1)$-comfortable is an intermediate condition between $(k-1)$-linearizable and $k$-linearizable. In our situation, the embedding $D \hookrightarrow U$ is $(k-1)$-comfortable if and only if the complex structures on $U$ and on $C$ are asymptotic at rate $O(r^{-k/\delta})$ with respect to any K\"ahler cone metric of the form $i\partial\overline\partial h^{-\delta}$ on $C$, where $h$ is any positively curved Hermitian metric on $L$ and $\delta > 0$ \cite[Prop 1.3]{ChiLi}. In particular, the embedding $D \hookrightarrow U$ is always $0$-comfortable,  and the complex structures on $U$ and on $C$ are always asymptotic at rate $O(r^{-1/\delta})$, but the embedding is $1$-linearizable if and only if its tangent sequence splits \cite[Rmk A.5]{ChiLi}, which is not always true \cite[Section 3]{Conlon3}.

\begin{lemma}\label{lem:whohoo}
If $\dim M \geq 3$ and $S$ is a smooth compact divisor in $M$ with positive normal bundle, then $S$ is $(k-1)$-comfortably embedded in $M$ if and only if $S$ is $(k-1)$-linearizable in $M$.
\end{lemma}

This is proved by applying the Kodaira-Nakano vanishing theorem on $S$ \cite[Rmk A.7]{ChiLi}. Thus, in our situation, $(k-1)$-linearizability (rather than $k$-linearizability) is actually already enough to ensure $O(r^{-k/\delta})$ convergence of the complex structures with respect to $i\partial\overline\partial h^{-\delta}$.

We now comment on the proof of Theorem \ref{main-chili}. The main point is a remarkable comparison of cohomology classes. Using ideas of Artin-Schlessinger, and crucially using the fact that $X_0 \setminus D = C$ has a {normal} and {isolated} singularity, one first constructs a Kodaira-Spencer type class $\mathbf{KS}^{(k)}$ in $T^1_C$ (a finite-dimensional cohomology space attached to the punctured cone $C \setminus \{o\}$) if the vanishing order of the affine family is at least $k$. This satisfies $\mathbf{KS}^{(k)} = 0$ if and only if the vanishing order is at least $k+1$. On any finite annulus $Y \subset C \setminus \{o\} = L \setminus D$ we have two restriction maps
\begin{equation}\label{goodsequence1}
T^1_C \hookrightarrow H^1(Y,\Theta_Y) \stackrel{\cong}{\leftarrow} H^1(L,\Theta_L).
\end{equation}
Then the key result of Li's paper \cite[Prop 2.15]{ChiLi} identifies the image of $\mathbf{KS}^{(k)}$ in $H^1(Y,\Theta_Y)$ with the image of the so-called $k$-th order Kodaira-Spencer class $\theta_k \in H^1(L,\Theta_L)$ of the deformation to the normal cone in $U \supset D$. The class $\theta_k$ can be defined only if the embedding $D \hookrightarrow U$ is at least $(k-1)$-comfortable, but then its definition is analogous to the definition of the classical Kodaira-Spencer class of a deformation of complex manifolds in terms of a suitable atlas.

By construction, $\theta_k$ lies in the $(-k)$-weight space $H^1(L,\Theta_L)[-k] \subset H^1(L,\Theta_L)$ with respect to the natural action of $\C^*$ on $C$ or on $L$. \cite[Prop 3.3]{ChiLi} provides us with an exact sequence
\begin{equation}\label{goodsequence2}
H^{1}(D,(N_{D}^{*})^{k})\to H^1(L,\Theta_L)[-k]\to
H^{1}(D,\Theta_{D}\otimes(N_{D}^{*})^{k})
\end{equation}
under which $\theta_k$ maps to the obstruction, $\mathfrak{g}_{k}$, to $k$-splitting of the embedding $D \hookrightarrow U$ \cite[Defn 2.1]{abt}. $k$-splitting is a necessary condition for being able to define $k$-comfortable embeddedness. If $\mathfrak{g}_k = 0$, then \cite[Prop 3.3]{ChiLi} also shows that $\theta_k$ lifts to the obstruction, $\mathfrak{h}_k$, to the latter property.

To prove Theorem \ref{main-chili}, we can assume by induction that the vanishing order of the affine family is at least $k$ and that the embedding is at least $(k-1)$-comfortable. If $\mathfrak{g}_k \neq 0$, then the embedding cannot be $k$-comfortable by definition. On the other hand, necessarily $\theta_k \neq 0$ and hence $\mathbf{KS}^{(k)} \neq 0$, so the vanishing order must be equal to $k$. If $\mathfrak{g}_k = 0$, then the embedding is at least $k$-comfortable if and only if $\mathfrak{h}_k = 0$. However, $H^1(D, (N_D^*)^k) = 0$ by the Kodaira-Nakano vanishing theorem (this step is essentially equivalent to the proof of Lemma \ref{lem:whohoo}), so $\mathfrak{h}_k = 0$ trivially. But then $\theta_k$ vanishes as well, so $\mathbf{KS}^{(k)} = 0$, so the vanishing order is at least $k + 1$. This proves Theorem \ref{main-chili}.

\newpage

\subsection{Extension}\label{ss:ChiLiExtension}

We would now like to generalize Theorem \ref{main-chili} to the setting of this paper, where $D$ is a compact complex orbifold, $L$ is a positive holomorphic orbifold line bundle on $D$, and $D$ is embedded as an admissible divisor with normal orbibundle $L$ into some open subset $U \subset X$ which is itself an orbifold. By the definition of an admissible divisor, $U \setminus D$ is actually a manifold. We are also still assuming that $\dim D \geq 2$ and that $X_0$ is normal, so that $X_0 = C \cup D$.

Formally the definitions of $k$-splitting, $k$-comfortable and $k$-linearizable in \cite{abt} still make sense for an orbifold embedding $S \hookrightarrow M$ because they only involve the sheaves $\mathcal{J}_S \subset \mathcal{O}_M$ and $\mathcal{J}_S \subset \mathcal{O}_{N_{S/M}}$. Thus, the verbatim extension of Theorem \ref{main-chili} to our new setting is meaningful as well:

\begin{theorem}\label{chili}
In the above situation, the supremum of all positive integers $k$ such that the affine family $\{X_t \setminus D\}_{t\in\C}$ becomes trivial after base change to ${\rm Spec}\;\C[t]/(t^k)$ is equal to the supremum of all positive integers $k$ such that the embedding $D\hookrightarrow U$ is $(k-1)$-comfortable.
\end{theorem}

We now need to ask whether this theorem is true and whether the $(k-1)$-comfortable property has the same meaning in terms of complex structure asymptotics as in the manifold case. Fortunately, the construction and properties of $\mathbf{KS}^{(k)}$ as well as the embedding $T^1_C \hookrightarrow H^1(Y,\Theta_Y)$ of \eqref{goodsequence1} can be stated and proved intrinsically in terms of the affine family $\{X_t \setminus D\}_{t \in \C}$. Thus, as far as this part of the proof is concerned, the manifold setting and the orbifold setting are actually identical.

On the other hand, the fact that the restriction map $H^1(L,\Theta_L) \to H^1(Y,\Theta_Y)$ is an isomorphism in every degree \cite[Lemma 4.6]{ChiLi}, the construction of the exact sequence \eqref{goodsequence2}, the construction and properties of the classes $\theta_k, \mathfrak{g}_k, \mathfrak{h}_k$, and the interpretation of comfortable embeddedness in terms of complex structure asymptotics need to be generalized from manifolds to orbifolds. In the manifold case, all of these properties follow from explicit computations with atlases and \v{C}ech cocycles. Thus, one would expect that these computations can be extended verbatim to the orbifold case by using orbifold atlases and invariant \v{C}ech cocycles on the local uniformizing charts. Perhaps surprisingly, this is \emph{not} true unless the orbifold atlases used satisfy a strong compatibility condition.

\begin{definition}\label{defn:adat}
Let $M^n$ be a complex orbifold. Let $S^{n-k} \subset M^n$ be a complex suborbifold. Call a holomorphic orbifold atlas $\{(U_\alpha,\Gamma_\alpha,\varphi_\alpha)\}_{\alpha\in A}$ of $M$ \emph{adapted to $S$} if, for all $\alpha\in A$:
\begin{enumerate}
\item $\Gamma_\alpha$ is a subgroup of $\text{GL}(n,\C)$ acting linearly on $U_\alpha \subset \C^n$;
\item $\varphi_\alpha^{-1}(S) = \{z \in U_\alpha: z_{n-k+1} = \cdots = z_n = 0\}$ in the standard coordinates of $\C^n$; and
\item $\Gamma_\alpha$ is actually contained in the block diagonal subgroup ${\rm GL}(n-k,\C) \times {\rm GL}(k,\C)$.
\end{enumerate}
\end{definition}

We have been unable to find this definition in the literature. It is a classical fact that (1) by itself can be achieved for the local uniformizing groups of any complex orbifold, and (2) by itself is simply the definition of a complex suborbifold. (1) and (2) imply that the groups $\Gamma_\alpha$ are contained in the block \emph{upper triangular} matrices relative to the splitting $\C^n = \C^{n-k} \oplus \C^k$, and then the induced atlas on the total space of the normal bundle to $S$ in $M$ satisfies (3). Taken together, (1)--(3) imply that the deformation of $M$ to the normal cone of $S$ is trivial on small open sets.

We will now prove that adapted atlases always exist in our situation. After this, we will briefly explain why the adapted condition allows us to generalize the computations in \cite{ChiLi} to orbifolds.

\begin{prop}\label{gauge}
Let $U \subset \C^{n}$ be a domain containing the origin. Let $\Gamma$ be a finite cyclic subgroup of ${\rm Aut}(U)$ fixing the origin and preserving the section $H = U \cap \{z_n = 0\}$. Then there exist domains $\tilde{U} \subset U$ and $U' \subset \C^n$ containing the origin, with sections $\tilde{H} = \tilde{U} \cap \{z_n = 0\}$ and $H' = U' \cap \{z_n = 0\}$, and an isomorphism $\psi: \tilde{U} \to U'$ with $\psi(0) = 0$ and $\psi(\tilde{H}) = H'$ such that $\tilde{U}$ is $\Gamma$-invariant and such that $\psi \circ \Gamma \circ \psi^{-1}$ acts on $U'$ by linear transformations in ${\rm GL}(n-1,\C) \times \C^*$.
\end{prop}

\begin{proof}
Proceeding as in \cite[Lemme 1]{cartan}, we define a holomorphic map $\sigma: U \to \C^n$ by
$$\sigma :=\frac{1}{|\Gamma|}\sum_{\gamma\in\Gamma}(d\gamma)^{-1}\circ \gamma,$$
where the summation, scalar multiplication and action of $d\gamma$ are with respect to the linear structure of $\C^n$. Then $\sigma$ maps $H$ into the hyperplane $\{z_n = 0\}$ because $\Gamma$ preserves $H$. Moreover, $d\sigma|_0 = {\rm Id}$, so $\sigma$ restricts to an isomorphism onto its image on any sufficiently small concentric domain $\tilde{U} \subset U$ and $\sigma(\tilde{H})$ contains a neighborhood of $0$ in $\{z_n = 0\}$. By replacing $\tilde{U}$ by $\bigcap_{\gamma \in \Gamma} \gamma(\tilde{U})$, we can make $\tilde{U}$ invariant under $\Gamma$. Observe that $\sigma \circ \gamma =d\gamma \circ\sigma$ for each $\gamma\in\Gamma$. Thus, the subgroup $\sigma \circ \Gamma \circ \sigma^{-1}$ of ${\rm Aut}(\sigma(\tilde{U}))$ acts by linear transformations preserving the hyperplane $\{z_n = 0\}$. For any sufficiently small invariant open set $V \subset \sigma(\tilde{U})$ with $0 \in V$, the section $V \cap \{z_n = 0\}$ will be contained in $\sigma(\tilde{H})$, so by replacing $\tilde{U}$ by $\sigma^{-1}(V)$ we can assume without loss that $\sigma(\tilde{H}) = \sigma(\tilde{U}) \cap \{z_n = 0\}$.

Let $A \in {\rm GL}(n,\C)$ be a generator of the finite cyclic group $\sigma \circ \Gamma \circ \sigma^{-1}$. Our goal is now to find a matrix $R \in {\rm GL}(n,\C)$ preserving the hyperplane $\{z_n = 0\}$ such that $R A R^{-1}$ is not only block upper triangular but in fact block diagonal with respect to the decomposition $\C^n = \C^{n-1} \oplus \C$. If this can be done, then setting $\psi := R \circ \sigma$ and $U' := \psi(\tilde{U})$ will complete the proof of the proposition.

Since $A$ preserves $H$, we must have that
$$A=\begin{pmatrix} B & {\bf c}  \\ {\bf 0}^T & d\end{pmatrix}$$
with $B \in {\rm GL}(n-1,\C)$, $\mathbf{c}$ and $\mathbf{0}$ column vectors in $\C^{n-1}$, and $d\in\C^*$. Because $A$ generates a finite subgroup of ${\rm GL}(n,\C)$, it follows that $B$ preserves a Hermitian metric on $\C^{n-1}$, hence is conjugate to a unitary matrix in ${\rm GL}(n-1,\C)$ by the Gram-Schmidt process, and hence to a diagonal matrix $\Lambda = {\rm diag}(\lambda_1, \ldots, \lambda_{n-1})$ with all entries unit complex numbers (indeed, roots of unity). Similarly, $d$ is a root of unity as well. We may thus assume without loss of generality that
\begin{align}\label{eq:formofA}
A=\begin{pmatrix}\Lambda & {\bf c}  \\ {\bf 0}^T & d\end{pmatrix}
\end{align}
because if $R$ can be found for $A$ of this form, then the general case follows by multiplying $R$ by an element of the block diagonal subgroup ${\rm GL}(n-1,\C) \times \{1\}$ of ${\rm GL}(n,\C)$.

For $A$ as in \eqref{eq:formofA} we now seek $R\in{\rm GL}(n,\C)$ of the form
$$R=\begin{pmatrix}{\rm Id} & \textbf{r} \\ \textbf{0}^T & 1\end{pmatrix}$$
for some column vector $\mathbf{r} \in \C^{n-1}$ such that
\begin{equation*}
R A R^{-1}=\begin{pmatrix}\Lambda  & \textbf{0}  \\ \textbf{0}^T & d\end{pmatrix}.
\end{equation*}
This is easily seen to be equivalent to the linear system
\begin{equation}\label{linear2}
(\lambda_i-d)r_{i}=c_i\quad(1 \leq i \leq n-1)
\end{equation}
for the components of $\mathbf{r}$ and $\mathbf{c}$. Let $N$ denote the number of $\lambda_{i}$ that are equal to $d$. If $N=0$, then \eqref{linear2} has a solution, so we are done. If $N\geq 1$, then we may assume without loss of generality that $d=\lambda_{1}=\lambda_{2}=\ldots=\lambda_{N}$, and it remains to prove that $c_i = 0$ for all $1 \leq i \leq N$ because then \eqref{linear2} again has a solution. We do so by contradiction. If $c_{i}\neq0$ for some $1\leq i\leq N$, then, for all $k \in \N$, an easy computation shows that
the $(1,i)$-th entry of $A^k$ has absolute value equal to
$$\left|\left(\sum_{j=0}^{k-1} d^j \lambda_i^{k-1-j}\right)c_{i}\right|=|k d^{k-1} c_i |=k|c_{i}|\to\infty\;\,\textrm{as}\;\,k\to\infty,$$
but clearly this is not possible because $A$ is of finite order.
\end{proof}

We now sketch how Proposition \ref{gauge} lets us complete the proof of Theorem \ref{chili}. In the setting of the theorem, let $\{(U_{\alpha},\Gamma_{\alpha},\varphi_{\alpha})\}_{\alpha\in A}$ be an open covering of a neighborhood of $D$ by uniformizing charts of $X$, each of which is an isotropy chart centered at a point of $D$. On each open set $U_{\alpha}$, we have centered holomorphic coordinates ${z}_{\alpha}=(z_{\alpha}^{1},\ldots,z_{\alpha}^{n})$ such that $D\cap U_{\alpha}=\{z^{n}_{\alpha}=0\}$ and $\Gamma_{\alpha}$ is a cyclic subgroup of ${\rm Aut}(U_\alpha)$ (if $\Gamma_\alpha$ was not cyclic, $U_\alpha \setminus D$ could not be a manifold). Proposition \ref{gauge} now allows us to assume without loss of generality that in every uniformizing chart $(U_{\alpha},\Gamma_{\alpha},\varphi_{\alpha})$, the action of $\Gamma_{\alpha}$ on $U_\alpha$ is linear in the coordinates ${z}_{\alpha}$ and the matrices representing the elements of $\Gamma_\alpha$ are block diagonal. Given any two uniformizing charts $(U_{\alpha},\Gamma_{\alpha},\varphi_{\alpha})$ and $(U_{\beta},\Gamma_{\beta},\varphi_{\beta})$ and any point $p \in D\cap\varphi_{\alpha}(U_{\alpha})\cap\varphi_{\beta}(U_{\beta})$, there exists a third uniformizing chart $(U,\Gamma,\varphi)$ with $p\in\varphi(U)$, together with embeddings $\lambda_{\alpha}:U\to U_{\alpha}$ and $\lambda_{\beta}:U\to U_{\beta}$.
We then have a change of coordinates map
$${F}_{\beta\alpha}:=\lambda_{\alpha}\circ\lambda_{\beta}^{-1}:\lambda_{\beta}(U)\subset U_{\beta}
\to\lambda_{\alpha}(U)\subset U_{\alpha}.$$
We write the components of ${F}_{\beta\alpha}$ as
$${F}_{\beta\alpha}({z}_{\beta})=(F^{1}_{\beta\alpha}({z}_{\beta}),\ldots,F^{n}_{\beta\alpha}({z}_{\beta})) = z_\alpha = (z_\alpha^1, \ldots, z_\alpha^n).$$
Clearly, this map is equivariant with respect to the action of $\Gamma$.

Using this setup, a lengthy routine check shows that those parts of the proof of Li's Theorem \ref{main-chili} that rely on computations with \v{C}ech cocycles go through verbatim in the orbifold setting. Here we only point out how this works for one of the most delicate computations, which would actually fail without part (3) of the adapted atlas condition (Definition \ref{defn:adat}). Specifically, we will show that the Abate-Bracci-Tovena cocycle $(\mathfrak{g}_k)_{\beta\alpha}$ (see \cite[Prop 2.2]{abt}) is $\Gamma$-invariant and hence defines an orbifold cohomology class. For clarity, we will only consider the case $k = 1$. By definition,
\begin{equation}\label{equiv}
(\mathfrak{g}_{1})_{\beta\alpha}(p)=\sum_{i=1}^{n-1}\frac{\partial F^{i}_{\beta\alpha}}{\partial z_{\beta}^{n}}\biggr|_{z_\beta(p)}\frac{\partial}{\partial z^{i}_{\alpha}}\biggr|_{z_\alpha(p)}\otimes dz_{\beta}^{n}|_{z_\beta(p)}
\end{equation}
for all $p \in D$. Similarly, for every $\gamma\in\Gamma$ we have that
\begin{equation}\label{equiv2}
(\mathfrak{g}_{1})_{\beta\alpha}(\gamma p)=\sum_{i=1}^{n-1}\frac{\partial F^{i}_{\beta\alpha}}{\partial z_{\beta}^{n}}\biggr|_{z_\beta(\gamma p)}\frac{\partial}{\partial z^{i}_{\alpha}}\biggr|_{z_\alpha(\gamma p)}\otimes dz_{\beta}^{n}|_{z_\beta(\gamma  p)}.
\end{equation}
To prove that $\gamma$ maps {\eqref{equiv} to \eqref{equiv2}, we first note that the action of $\gamma$ is represented by matrices ${^\alpha}\gamma, {^\beta}\gamma \in {\rm GL}(n,\C)$ in the two charts. We will assume for now that these are block upper triangular but not necessarily block diagonal. Then we have that
\begin{align*}
\begin{split}
\gamma \cdot \frac{\partial}{\partial z^{i}_{\alpha}}\biggr|_{z_\alpha(p)} &= \sum_{j=1}^n {^\alpha}\gamma_{ji}\frac{\partial}{\partial z^{j}_{\alpha}}\biggr|_{z_\alpha(\gamma p)},\\
\gamma \cdot dz^n_\beta|_{z_\beta(p)} &= \sum_{k=1}^n ({^\beta}\gamma^{-1})_{nk} dz^k_\beta|_{z_\beta(\gamma p)} = \frac{1}{{^\beta}\gamma_{nn}}dz^n_\beta|_{z_\beta(\gamma p)}.
\end{split}
\end{align*}
Thus, the desired equality, $\gamma \cdot \eqref{equiv} = \eqref{equiv2}$, holds if and only if
\begin{align}\label{desire}
\sum_{i=1}^{n-1} {^\alpha}\gamma_{ji}\frac{\partial F_{\beta\alpha}^i}{\partial z^n_\beta}\biggr|_{z_\beta(p)}=
\begin{cases}{^\beta}\gamma_{nn}\frac{\partial F^j_{\beta\alpha}}{\partial z^n_\beta}\Bigr|_{z_\beta(\gamma p)}&\text{for}\;\,1 \leq j \leq n-1,\\
0&\text{for}\;\,j = n.
\end{cases}
\end{align}
This is true for $j = n$ because ${^\alpha}\gamma$ is block upper triangular. To deal with the case $1 \leq j \leq n-1$, we differentiate the relation
${F}_{\beta\alpha}({^\beta}\gamma \cdot {z}_{\beta})={^\alpha}\gamma \cdot {F}_{\beta\alpha}({z}_{\beta})$
with respect to $z_{\beta}^{n}$. This yields that
\begin{equation}\label{derivative}
\sum_{i=1}^{n}{^\beta}\gamma_{in}\frac{\partial F^{j}_{\beta\alpha}}{\partial z_{\beta}^{i}}\biggr|_{z_\beta(\gamma p)}
=\sum_{i=1}^{n}{^\alpha}\gamma_{ji}\frac{\partial F^{i}_{\beta\alpha}}{\partial z_{\beta}^{n}}\biggr|_{z_\beta(p)}\;\,\text{for all}\;\,1 \leq j \leq n.
\end{equation}
For $j = n$, again because ${^\alpha}\gamma$ is block upper triangular, we can deduce from this that
\begin{equation}\label{ableitung666}
\sum_{i=1}^{n}{^\beta}\gamma_{in}\frac{\partial F^{n}_{\beta\alpha}}{\partial z_{\beta}^{i}}\biggr|_{z_\beta(\gamma p)}
={^\alpha}\gamma_{nn}\frac{\partial F^{n}_{\beta\alpha}}{\partial z_{\beta}^{n}}\biggr|_{z_\beta(p)}.
\end{equation}
Fixing $1 \leq j \leq n-1$ and combining \eqref{derivative} with \eqref{ableitung666}, we can rewrite the left-hand side of \eqref{desire} for this particular value of $j$ as
$$\sum_{i = 1}^n {^\beta}\gamma_{in}\left[\frac{\partial F^j_{\beta\alpha}}{\partial z_\beta^i}\biggr|_{z_\beta(\gamma p)} - \frac{{^\alpha}\gamma_{jn}}{{^\alpha}\gamma_{nn}}\frac{\partial F^n_{\beta\alpha}}{\partial z_\beta^i}\biggr|_{z_\beta(\gamma p)}\right].$$
If both ${^\alpha}\gamma$ and ${^\beta}\gamma$ are block diagonal, then this is trivially equal to the right-hand side of \eqref{desire}, as desired, whereas otherwise there seems to be no reason for this to be true.

\bibliographystyle{amsplain}
\bibliography{ref2}

\def\cprime{$'$} \def\cprime{$'$}
\providecommand{\bysame}{\leavevmode\hbox to3em{\hrulefill}\thinspace}
\providecommand{\MR}{\relax\ifhmode\unskip\space\fi MR }
\providecommand{\MRhref}[2]{%
  \href{http://www.ams.org/mathscinet-getitem?mr=#1}{#2}
}
\providecommand{\href}[2]{#2}
\begin{thebibliography}{100}

\bibitem{abt}
M.~Abate, F.~Bracci, and F.~Tovena, \emph{Embeddings of submanifolds and normal
  bundles}, Adv. Math. \textbf{220} (2009), 620--656.

\bibitem{alessandrini}
L.~Alessandrini and G.~Bassanelli, \emph{On the embedding of 1-convex manifolds
  with 1-dimensional exceptional set}, Ann. Inst. Fourier (Grenoble)
  \textbf{51} (2001), 99--108.

\bibitem{altmannn}
K.~Altmann, \emph{Toric $\mathbb{Q}$-{G}orenstein singularities},
  \textnormal{arXiv:9403003}.

\bibitem{Alty}
\bysame, \emph{Minkowski sums and homogeneous deformations of toric varieties},
  Tohoku Math. J. (2) \textbf{47} (1995), 151--184.

\bibitem{altmann}
\bysame, \emph{The versal deformation of an isolated toric {G}orenstein
  singularity}, Invent. Math. \textbf{128} (1997), 443--479.

\bibitem{AndFund}
M.~Anderson, \emph{On the topology of complete manifolds of nonnegative {R}icci
  curvature}, Topology \textbf{29} (1990), 41--55.

\bibitem{AndersonSurvey}
\bysame, \emph{A survey of {E}instein metrics on 4-manifolds}, Handbook of
  Geometric Analysis, {N}o. 3, Adv. Lect. Math., vol.~14, Int. Press,
  Somerville, MA, 2010, pp.~1--39.

\bibitem{Artin2}
M.~Artin, \emph{Versal deformations and algebraic stacks}, Invent. Math.
  \textbf{27} (1974), 165--189.

\bibitem{Artin}
\bysame, \emph{Lectures on {D}eformations of {S}ingularities}, Tata Institute
  of Fundamental Research, Bombay, 1976.

\bibitem{Bai57}
W.~Baily, \emph{On the imbedding of {$V$}-manifolds in projective space}, Amer.
  J. Math. \textbf{79} (1957), 403--430.

\bibitem{ALE}
S.~Bando, A.~Kasue, and H.~Nakajima, \emph{On a construction of coordinates at
  infinity on manifolds with fast curvature decay and maximal volume growth},
  Invent. Math. \textbf{97} (1989), 313--349.

\bibitem{BK}
S.~Bando and R.~Kobayashi, \emph{Ricci-flat {K}\"{a}hler metrics on affine
  algebraic manifolds. {II}}, Math. Ann. \textbf{287} (1990), 175--180.

\bibitem{dP1-physics}
D.~Berenstein, P.~Ouyang, S.~Pinansky, and C.~Herzog, \emph{Supersymmetry
  breaking from a {C}alabi-{Y}au singularity}, J. High Energy Phys. (2005),
  no.~9, 084.

\bibitem{biq-delc}
O.~Biquard and T.~Delcroix, \emph{Ricci flat {K}\"{a}hler metrics on rank two
  complex symmetric spaces}, J. \'{E}c. polytech. Math. \textbf{6} (2019),
  163--201.

\bibitem{BiqGaud}
O.~Biquard and P.~Gauduchon, \emph{Hyper-{K}\"{a}hler metrics on cotangent
  bundles of {H}ermitian symmetric spaces}, Geometry and Physics ({A}arhus,
  1995), Lecture Notes in Pure and Appl. Math., vol. 184, Dekker, New York,
  1997, pp.~287--298.

\bibitem{book:Boyer}
C.~Boyer and K.~Galicki, \emph{Sasakian geometry}, Oxford Math.~Monographs,
  Oxford Univ.~Press, Oxford, 2008.

\bibitem{delaossa}
P.~Candelas and X.~de~la Ossa, \emph{Comments on conifolds}, Nuclear Phys. B
  \textbf{342} (1990), 246--268.

\bibitem{cartan}
H.~Cartan, \emph{Quotient d'un espace analytique par un groupe
  d'automorphismes}, Algebraic Geometry and Topology, Princeton Univ.~Press,
  Princeton, NJ, 1957, pp.~90--102.

\bibitem{ChBook}
J.~Cheeger, \emph{Degeneration of {R}iemannian metrics under {R}icci curvature
  bounds}, Lezioni Fermiane, Scuola Normale Superiore, Pisa, 2001.

\bibitem{CheegerSurvey}
\bysame, \emph{Degeneration of {E}instein metrics and metrics with special
  holonomy}, Surveys in Differential Geometry, {V}ol. {VIII} ({B}oston, {MA},
  2002), Int. Press, Somerville, MA, 2003, pp.~29--73.

\bibitem{CheegerColding}
J.~Cheeger and T.~Colding, \emph{Lower bounds on {R}icci curvature and the
  almost rigidity of warped products}, Ann. of Math. (2) \textbf{144} (1996),
  189--237.

\bibitem{ChNa}
J.~Cheeger and A.~Naber, \emph{Regularity of {E}instein manifolds and the
  codimension 4 conjecture}, Ann. of Math. (2) \textbf{182} (2015), 1093--1165.

\bibitem{Cheeger}
J.~Cheeger and G.~Tian, \emph{On the cone structure at infinity of {R}icci flat
  manifolds with {E}uclidean volume growth and quadratic curvature decay},
  Invent. Math. \textbf{118} (1994), 493--571.

\bibitem{Chiu}
S.-K. Chiu, \emph{Subquadratic harmonic functions on {C}alabi-{Y}au manifolds
  with {E}uclidean volume growth}, \textnormal{arXiv: 1905.12965}.

\bibitem{cho}
K.~Cho, A.~Futaki, and H.~Ono, \emph{Uniqueness and examples of compact toric
  {S}asaki-{E}instein metrics}, Comm. Math. Phys. \textbf{277} (2008),
  439--458.

\bibitem{CM}
T.~Colding and W.~Minicozzi, \emph{On uniqueness of tangent cones for
  {E}instein manifolds}, Invent. Math. \textbf{196} (2014), 515--588.

\bibitem{GaborTristan2}
T.~Collins and G.~Sz\'{e}kelyhidi, \emph{Sasaki-{E}instein metrics and
  {K}-stability}, Geom. Topol. \textbf{23} (2019), 1339--1413.

\bibitem{Tristan}
T.~Collins and V.~Tosatti, \emph{K\"{a}hler currents and null loci}, Invent.
  Math. \textbf{202} (2015), 1167--1198.

\bibitem{CollinsTosatti}
\bysame, \emph{A singular {D}emailly-{P}\u{a}un theorem}, C. R. Math. Acad.
  Sci. Paris \textbf{354} (2016), 91--95.

\bibitem{CDR}
R.~Conlon, A.~Degeratu, and F.~Rochon, \emph{Quasi-asymptotically conical
  {C}alabi-{Y}au manifolds}, Geom. Topol. \textbf{23} (2019), 29--100, with an
  appendix by R.~Conlon, F.~Rochon, and L.~Sektnan.

\bibitem{Conlon}
R.~Conlon and H.-J. Hein, \emph{Asymptotically conical {C}alabi-{Y}au
  manifolds, {I}}, Duke Math. J. \textbf{162} (2013), 2855--2902.

\bibitem{Conlon3}
\bysame, \emph{Asymptotically conical {C}alabi-{Y}au metrics on
  quasi-projective varieties}, Geom. Funct. Anal. \textbf{25} (2015), 517--552.

\bibitem{CR}
R.~Conlon and F.~Rochon, \emph{New examples of complete {C}alabi-{Y}au metrics
  on {$\Bbb C^n$} for {$n \geq 3$}}, Ann. Sci. \'{E}c. Norm. Sup\'{e}r. (4)
  \textbf{54} (2021), 259--303.

\bibitem{cox}
D.~Cox, J.~Little, and H.~Schenck, \emph{Toric varieties}, Graduate Studies in
  Mathematics, vol. 124, American Mathematical Society, Providence, RI, 2011.

\bibitem{doan}
A.~Doan, \emph{A counter-example to the equivariance structure on
  semi-universal deformation}, J. Geom. Anal. \textbf{31} (2021), 3698--3712.

\bibitem{donaldson}
S.~Donaldson, \emph{Scalar curvature and stability of toric varieties}, J.
  Differential Geom. \textbf{62} (2002), 289--349.

\bibitem{DS2}
S.~Donaldson and S.~Sun, \emph{Gromov-{H}ausdorff limits of {K}\"{a}hler
  manifolds and algebraic geometry, {II}}, J. Differential Geom. \textbf{107}
  (2017), 327--371.

\bibitem{durfee}
A.~Durfee, \emph{Fifteen characterizations of rational double points and simple
  critical points}, Enseign. Math. (2) \textbf{25} (1979), 131--163.

\bibitem{Elkik2}
R.~Elkik, \emph{Solutions d'\'{e}quations \`a coefficients dans un anneau
  hens\'{e}lien}, Ann. Sci. \'{E}cole Norm. Sup. (4) \textbf{6} (1973),
  553--603 (1974).

\bibitem{Epstein}
C.~Epstein and G.~Henkin, \emph{Stability of embeddings for pseudoconcave
  surfaces and their boundaries}, Acta Math. \textbf{185} (2000), 161--237.

\bibitem{faulk}
M.~Faulk, \emph{Some canonical metrics on {K}\"ahler orbifolds}, Ph.D. thesis,
  Columbia University, 2019, available at
  \url{https://academiccommons.columbia.edu/doi/10.7916/d8-2jm6-2b57}.

\bibitem{Fulton}
W.~Fulton, \emph{Intersection theory}, second ed., Ergebnisse der Mathematik
  und ihrer Grenzgebiete.~3.~Folge, vol.~2, Springer, Berlin, 1998.

\bibitem{futaki}
A.~Futaki, H.~Ono, and G.~Wang, \emph{Transverse {K}\"ahler geometry of
  {S}asaki manifolds and toric {S}asaki-{E}instein manifolds}, J. Differential
  Geom. \textbf{83} (2009), 585--635.

\bibitem{Sparks}
J.~Gauntlett, D.~Martelli, J.~Sparks, and D.~Waldram, \emph{Sasaki-{E}instein
  metrics on {$S^2\times S^3$}}, Adv. Theor. Math. Phys. \textbf{8} (2004),
  711--734.

\bibitem{goto}
R.~Goto, \emph{Calabi-{Y}au structures and {E}instein-{S}asakian structures on
  crepant resolutions of isolated singularities}, J. Math. Soc. Japan
  \textbf{64} (2012), 1005--1052.

\bibitem{Grau:62}
H.~Grauert, \emph{\"{U}ber {M}odifikationen und exzeptionelle analytische
  {M}engen}, Math. Ann. \textbf{146} (1962), 331--368.

\bibitem{Grauert-def}
\bysame, \emph{\"{U}ber die {D}eformation isolierter {S}ingularit\"{a}ten
  analytischer {M}engen}, Invent. Math. \textbf{15} (1972), 171--198.

\bibitem{GRie}
H.~Grauert and O.~Riemenschneider, \emph{Verschwindungss\"{a}tze f\"{u}r
  analytische {K}ohomologiegruppen auf komplexen {R}\"{a}umen}, Invent. Math.
  \textbf{11} (1970), 263--292.

\bibitem{singing}
G.-M. Greuel, C.~Lossen, and E.~Shustin, \emph{Introduction to singularities
  and deformations}, Springer Monographs in Mathematics, Springer, Berlin,
  2007.

\bibitem{Griffiths}
P.~Griffiths, \emph{The extension problem in complex analysis. {II}.
  {E}mbeddings with positive normal bundle}, Amer. J. Math. \textbf{88} (1966),
  366--446.

\bibitem{EGA2}
A.~Grothendieck, \emph{\'{E}l\'{e}ments de g\'{e}om\'{e}trie alg\'{e}brique.
  {II}. \'{E}tude globale \'{e}l\'{e}mentaire de quelques classes de
  morphismes}, Inst. Hautes \'{E}tudes Sci. Publ. Math. \textbf{8} (1961),
  5--205.

\bibitem{Hau}
H.~Hauser, \emph{La construction de la d\'{e}formation semi-universelle d'un
  germe de vari\'{e}t\'{e} analytique complexe}, Ann. Sci. \'{E}cole Norm. Sup.
  (4) \textbf{18} (1985), 1--56.

\bibitem{helb}
H.-J. Hein and C.~LeBrun, \emph{Mass in {K}\"{a}hler geometry}, Comm. Math.
  Phys. \textbf{347} (2016), 183--221.

\bibitem{HRS}
H.-J. Hein, R.~R\u{a}sdeaconu, and I.~\c{S}uvaina, \emph{On the classification
  of {ALE} {K}\"{a}hler manifolds}, Int. Math. Res. Not. IMRN (2021), no.~14,
  10957--10980.

\bibitem{HS}
H.-J. Hein and S.~Sun, \emph{Calabi-{Y}au manifolds with isolated conical
  singularities}, Publ. Math. Inst. Hautes \'{E}tudes Sci. \textbf{126} (2017),
  73--130.

\bibitem{HiRossi}
H.~Hironaka and H.~Rossi, \emph{On the equivalence of imbeddings of exceptional
  complex spaces}, Math. Ann. \textbf{156} (1964), 313--333.

\bibitem{Ilten}
N.~Ilten and R.~Vollmert, \emph{Deformations of rational {$T$}-varieties}, J.
  Algebraic Geom. \textbf{21} (2012), 531--562.

\bibitem{Ishii}
S.~Ishii, \emph{Introduction to singularities}, Springer, Tokyo, 2014.

\bibitem{JiangNaber}
W.~Jiang and A.~Naber, \emph{{$L^2$} curvature bounds on manifolds with bounded
  {R}icci curvature}, Ann. of Math. (2) \textbf{193} (2021).

\bibitem{JoyceALE}
D.~Joyce, \emph{Asymptotically {L}ocally {E}uclidean metrics with holonomy
  {${\rm SU}(m)$}}, Ann. Global Anal. Geom. \textbf{19} (2001), 55--73.

\bibitem{JoyceQALE}
\bysame, \emph{Quasi-{ALE} metrics with holonomy {${\rm SU}(m)$} and {${\rm
  Sp}(m)$}}, Ann. Global Anal. Geom. \textbf{19} (2001), 103--132.

\bibitem{SpiroJason}
S.~Karigiannis and J.~Lotay, \emph{Deformation theory of {$\rm G_2$}
  conifolds}, Comm. Anal. Geom. \textbf{28} (2020), 1057--1210.

\bibitem{Kas}
A.~Kas and M.~Schlessinger, \emph{On the versal deformation of a complex space
  with an isolated singularity}, Math. Ann. \textbf{196} (1972), 23--29.

\bibitem{Kollar}
J.~Koll\'{a}r, \emph{Flops}, Nagoya Math. J. \textbf{113} (1989), 15--36.

\bibitem{flops}
\bysame, \emph{Flips, flops, minimal models, etc}, Surveys in Differential
  Geometry ({C}ambridge, {MA}, 1990), Lehigh Univ., Bethlehem, PA, 1991,
  pp.~113--199.

\bibitem{KM}
J.~Koll\'{a}r and S.~Mori, \emph{Birational geometry of algebraic varieties},
  Cambridge Tracts in Mathematics, vol. 134, Cambridge Univ.~Press, Cambridge,
  1998.

\bibitem{Kronheimer}
P.~Kronheimer, \emph{The construction of {ALE} spaces as hyper-{K}\"ahler
  quotients}, J. Differential Geom. \textbf{29} (1989), 665--683.

\bibitem{Kronheimer2}
\bysame, \emph{A {T}orelli-type theorem for gravitational instantons}, J.
  Differential Geom. \textbf{29} (1989), 685--697.

\bibitem{lamotke}
K.~Lamotke, \emph{Regular solids and isolated singularities}, Advanced Lectures
  in Mathematics, Friedr. Vieweg \& Sohn, Braunschweig, 1986.

\bibitem{ChiLi}
C.~Li, \emph{On sharp rates and analytic compactifications of asymptotically
  conical {K}\"ahler metrics}, Duke Math.~J. \textbf{169} (2020), 1397--1483.

\bibitem{LiFund}
P.~Li, \emph{Large time behavior of the heat equation on complete manifolds
  with nonnegative {R}icci curvature}, Ann. of Math. (2) \textbf{124} (1986),
  1--21.

\bibitem{LiYang}
Y.~Li, \emph{A new complete {C}alabi-{Y}au metric on {$\Bbb C^3$}}, Invent.
  Math. \textbf{217} (2019), 1--34.

\bibitem{LiuGang}
G.~Liu, \emph{Compactification of certain {K}\"{a}hler manifolds with
  nonnegative {R}icci curvature}, Adv. Math. \textbf{382} (2021), Paper No.
  107652, 27.

\bibitem{LiuSz1}
G.~Liu and G.~Sz\'ekelyhidi, \emph{Gromov-{H}ausdorff limits of {K}\"ahler
  manifolds with {R}icci curvature bounded below},
  \textnormal{arXiv:1804.08567}.

\bibitem{LiuSz2}
G.~Liu and G.~Sz\'{e}kelyhidi, \emph{Gromov-{H}ausdorff limits of {K}\"{a}hler
  manifolds with {R}icci curvature bounded below {II}}, Comm. Pure Appl. Math.
  \textbf{74} (2021), 909--931.

\bibitem{Ma1}
X.~Ma and G.~Marinescu, \emph{Holomorphic {M}orse inequalities and {B}ergman
  kernels}, Progress in Mathematics, vol. 254, Birkh\"{a}user Verlag, Basel,
  2007.

\bibitem{sparky}
D.~Martelli, J.~Sparks, and S.-T. Yau, \emph{The geometric dual of
  {$a$}-maximisation for toric {S}asaki-{E}instein manifolds}, Comm. Math.
  Phys. \textbf{268} (2006), 39--65.

\bibitem{Sparks2}
\bysame, \emph{Sasaki-{E}instein manifolds and volume minimisation}, Comm.
  Math. Phys. \textbf{280} (2008), 611--673.

\bibitem{nitta}
Y.~Nitta and K.~Sekiya, \emph{Uniqueness of {S}asaki-{E}instein metrics},
  Tohoku Math. J. (2) \textbf{64} (2012), no.~3, 453--468.

\bibitem{dP2-physics}
S.~Pinansky, \emph{Quantum deformations from toric geometry}, J. High Energy
  Phys. (2006), no.~3, 055.

\bibitem{Pinkham}
H.~Pinkham, \emph{Deformations of algebraic varieties with {$G_{m}$} action},
  Ast\'{e}risque, vol.~20, Soci\'{e}t\'{e} Math\'{e}matique de France, Paris,
  1974.

\bibitem{Srinivas}
G.~Ravindra and V.~Srinivas, \emph{The {G}rothendieck-{L}efschetz theorem for
  normal projective varieties}, J. Algebraic Geom. \textbf{15} (2006),
  563--590.

\bibitem{Riemen2}
O.~Riemenschneider, \emph{Characterizing {M}oi\v{s}ezon spaces by almost
  positive coherent analytic sheaves}, Math. Z. \textbf{123} (1971), 263--284.

\bibitem{Rim}
D.~Rim, \emph{Equivariant {$G$}-structure on versal deformations}, Trans. Amer.
  Math. Soc. \textbf{257} (1980), 217--226.

\bibitem{ross-thomas}
J.~Ross and R.~Thomas, \emph{Weighted projective embeddings, stability of
  orbifolds, and constant scalar curvature {K}\"{a}hler metrics}, J.
  Differential Geom. \textbf{88} (2011), 109--159.

\bibitem{Rossi2}
H.~Rossi, \emph{Vector fields on analytic spaces}, Ann. of Math. (2)
  \textbf{78} (1963), 455--467.

\bibitem{schlessinger}
M.~Schlessinger, \emph{Functors of {A}rtin rings}, Trans. Amer. Math. Soc.
  \textbf{130} (1968), 208--222.

\bibitem{Schlessinger2}
\bysame, \emph{Rigidity of quotient singularities}, Invent. Math. \textbf{14}
  (1971), 17--26.

\bibitem{Segal}
G.~Segal, \emph{Equivariant {$K$}-theory}, Inst. Hautes \'{E}tudes Sci. Publ.
  Math. \textbf{34} (1968), 129--151.

\bibitem{AG5}
I.~Shafarevich (ed.), \emph{Algebraic geometry. {V}}, Encyclopaedia of Math.
  Sciences, vol.~47, Springer, Berlin, 1999.

\bibitem{Slodowy}
P.~Slodowy, \emph{Simple singularities and simple algebraic groups}, Lecture
  Notes in Mathematics, vol. 815, Springer, Berlin, 1980.

\bibitem{Stenzel}
M.~Stenzel, \emph{Ricci-flat metrics on the complexification of a compact rank
  one symmetric space}, Manuscripta Math. \textbf{80} (1993), 151--163.

\bibitem{suvaina}
I.~{\c{S}}uvaina, \emph{A{LE} {R}icci-flat {K}\"{a}hler metrics and
  deformations of quotient surface singularities}, Ann. Global Anal. Geom.
  \textbf{41} (2012), 109--123.

\bibitem{Szekel}
G.~Sz\'{e}kelyhidi, \emph{Degenerations of {$\bold C^n$} and {C}alabi-{Y}au
  metrics}, Duke Math. J. \textbf{168} (2019), 2651--2700.

\bibitem{Szekel2}
\bysame, \emph{Uniqueness of some {C}alabi-{Y}au metrics on {${\bf C}^n$}},
  Geom. Funct. Anal. \textbf{30} (2020), 1152--1182.

\bibitem{tak}
T.~Takahashi, \emph{Deformations of {S}asakian structures and its application
  to the {B}rieskorn manifolds}, Tohoku Math. J. (2) \textbf{30} (1978),
  37--43.

\bibitem{TianSurvey}
G.~Tian, \emph{Aspects of metric geometry of four manifolds}, Inspired by {S}.
  {S}. {C}hern, Nankai Tracts Math., vol.~11, World Sci. Publ., Hackensack, NJ,
  2006, pp.~381--397.

\bibitem{Tian}
G.~Tian and S.-T. Yau, \emph{Complete {K}\"ahler manifolds with zero {R}icci
  curvature. {II}}, Invent. Math. \textbf{106} (1991), 27--60.

\bibitem{vanC2}
C.~van Coevering, \emph{Ricci-flat {K}\"ahler metrics on crepant resolutions of
  {K}\"ahler cones}, Math. Ann. \textbf{347} (2010), 581--611.

\bibitem{vanC4}
\bysame, \emph{Examples of asymptotically conical {R}icci-flat {K}\"{a}hler
  manifolds}, Math. Z. \textbf{267} (2011), 465--496.

\bibitem{ViaclovskySurvey}
J.~Viaclovsky, \emph{Critical metrics for {R}iemannian curvature functionals},
  Geometric Analysis, IAS/Park City Math. Ser., vol.~22, Amer. Math. Soc.,
  Providence, RI, 2016, pp.~197--274.

\bibitem{wright}
E.~Wright, \emph{Quotients of gravitational instantons}, Ann. Global Anal.
  Geom. \textbf{41} (2012), 91--108.

\end{thebibliography}

\end{document}